\documentclass[moor,nonblindrev]{informs1} % for review, not blinded
\usepackage{natbib}
 \NatBibNumeric
 \def\bibfont{\small}%
 \def\BIBand{and}%
 \bibpunct[, ]{[}{]}{,}{n}{}{,}%

\usepackage[colorlinks=true,breaklinks=true,bookmarks=true,urlcolor=blue,
     citecolor=blue,linkcolor=blue,bookmarksopen=false,draft=false]{hyperref}

\def\EMAIL#1{\href{mailto:#1}{#1}}  % When hyperref is used, otherwise outcomment 
\TheoremsNumberedThrough     % Preferred (Theorem 1, Lemma 1, Theorem 2)
\EquationsNumberedThrough    % Default: (1), (2), ...
\usepackage{algorithm,algpseudocode}
\usepackage{booktabs} % Horizontal rules in tables
\usepackage{multicol, multirow} % Used for the two-column layout of the document
\usepackage{float} % Required for tables and figures in the multi-column environment - they need to be placed in specific locations with the [H] (e.g., \begin{table}[H])
\usepackage{caption} % Custom captions under/above floats in tables or figures

\newcommand{\R}{\mathbb{R}}
\newcommand{\Rext}{\mathbb{R}\cup\{+\infty\}}
\newcommand{\Espace}{\mathbb{R}^p}
\newcommand{\Id}{\mathbb{I}}

\newcommand{\set}[1]{\left\{#1\right\}}
\newcommand{\sets}[1]{\{#1\}}
\newcommand{\iprods}[1]{\langle{ #1}\rangle}

\newcommand{\norms}[1]{\Vert#1\Vert}
\newcommand{\dom}[1]{\mathrm{dom}(#1)}
\newcommand{\Exp}[1]{\mathbb{E}\left[#1\right]}
\newcommand{\Exps}[2]{\mathbb{E}_{#1}\left[#2\right]}
\newcommand{\Prob}[1]{\mathbf{Prob}\left(#1\right)}
\newcommand{\zer}[1]{\mathrm{zer}(#1)}
\newcommand{\gra}[1]{\mathrm{gra}(#1)}
\newcommand{\Eproof}{\hfill $\square$}
\newcommand{\BigO}[1]{\mathcal{O}\left(#1\right)}
\newcommand{\BigOs}[1]{\mathcal{O}\big(#1\big)}
\newcommand{\SmallO}[1]{o\left(#1\right)}

\newcommand{\E}{\mathbb{E}}
\newcommand{\Sc}{\mathcal{S}}
\newcommand{\Lc}{\mathcal{L}}
\newcommand{\Qc}{\mathcal{Q}}
\newcommand{\Tc}{\mathcal{T}}
\newcommand{\Fc}{\mathcal{F}}
\newcommand{\Pc}{\mathcal{P}}
\newcommand{\Kc}{\mathcal{K}}
\newcommand{\Nc}{\mathcal{N}}
\newcommand{\Xc}{\mathcal{X}}

\newcommand{\mbf}[1]{\mathbf{#1}}

\newcommand{\beforesec}{\vspace{-1ex}}
\newcommand{\aftersec}{\vspace{-1ex}}
\newcommand{\beforesubsec}{\vspace{-1ex}}
\newcommand{\aftersubsec}{\vspace{-1ex}}
\newcommand{\beforesubsubsec}{\vspace{-1ex}}
\newcommand{\aftersubsubsec}{\vspace{-1ex}}

\newcommand{\revised}[1]{{\color{black}#1}}

\newcommand{\myeq}[2]{\vspace{-0.5ex}
\begin{equation}\label{#1}
{#2}
\vspace{-0.5ex}
\end{equation}
}
\newcommand{\myeqn}[1]{\vspace{-0.5ex}
\begin{equation*}
{#1}
\vspace{-0.5ex}
\end{equation*}
}

\begin{document}
%\setcounter{page}{1}
%\DOI{10.1287/xxxx.0000.0000}%

\RUNAUTHOR{Quoc Tran-Dinh and Yang Luo}

\RUNTITLE{Randomized Block-Coordinate Optimistic Gradient Algorithms}

\TITLE{Randomized Block-Coordinate Optimistic Gradient Algorithms for Root-Finding Problems}

\ARTICLEAUTHORS{%
\AUTHOR{Quoc Tran-Dinh \textit{and} Yang Luo}
\AFF{Department of Statistics and Operations Research \\ The University of North Carolina at Chapel Hill (UNC), 318-Hanes Hall, Chapel Hill, NC27599-3260, USA.\\ 
\EMAIL{quoctd@email.unc.edu; yangluo@unc.edu}}
%\AUTHOR{Yang Luo}
%\AFF{The University of North Carolina at Chapel Hill (UNC), 333-Hanes Hall, Chapel Hill, NC27599-3260, USA.\\
%\EMAIL{yangluo@unc.edu}}
} % end of the block

\ABSTRACT{
In this paper, we develop two new randomized block-coordinate optimistic gradient algorithms to approximate a solution of nonlinear equations in large-scale settings, which are known as \textit{root-finding problems}.
Our first algorithm is non-accelerated with constant stepsizes, and achieves a $\mathcal{O}(1/k)$ best-iterate convergence rate on $\Exp{\norms{Gx^k}^2}$ when the underlying operator $G$ is Lipschitz continuous and possesses  a weak Minty solution, where $\Exp{\cdot}$ is the expectation and $k$ is the iteration counter.
Our second method is a new accelerated randomized block-coordinate optimistic gradient algorithm.
We establish both $\BigO{1/k^2}$ and $\SmallO{1/k^2}$ last-iterate convergence rates on both $\Exp{\norms{Gx^k}^2}$ and $\Exp{\norms{x^{k+1} - x^{k}}^2}$ for this algorithm under the co-coerciveness of $G$.
In addition, we prove that the iterate sequence  $\sets{x^k}$ converges to a solution almost surely and $\norms{Gx^k}$ attains a $\SmallO{1/k}$ almost sure convergence rate.
Then, we apply our methods to a class of large-scale finite-sum inclusions, which covers prominent applications in machine learning, statistical learning, and network optimization, especially in federated learning.
We obtain two new federated learning-type algorithms and their convergence rate guarantees for solving this problem class.
We test our algorithms on four numerical examples using both synthetic and real data, and compare them with related methods.
Our numerical experiments show some promising performance of the proposed methods against their competitors.
}

\MSCCLASS{90C25, 90C06, 90-08}
\ORMSCLASS{Operations research; mathematical programming; game theory}
%\HISTORY{Received March 4, 2016; revised December 22, 2016, July 26, 2017, and September 19, 2017.}

\maketitle

%%%%%%%%%%%%%%%%%%%%%%%%%%%%%%%%%%%%%%%%%%%%%%%%%
%% 1. Introduction.
%%%%%%%%%%%%%%%%%%%%%%%%%%%%%%%%%%%%%%%%%%%%%%%%%
\vspace{-1ex}
\beforesec
\section{Introduction.}\label{sec:intro}
\aftersec
In this paper, we first develop two new randomized block-coordinate optimistic gradient-type algorithms to solve the following [non]linear equation: 
\myeq{eq:NE}{
\textrm{Find $x^{\star} \in \dom{G}$ such that:} \quad  Gx^{\star} = 0,
\tag{NE}
}
where $G : \Espace \to \Espace$ is a single-valued operator satisfying a given assumption: either Assumption~\ref{as:A1} in Section~\ref{sec:Na_RCOG_method} or Assumption~\ref{as:A2} in Section~\ref{sec:AcRCOG_method}, and $\dom{G}$ is the domain of $G$.

Next, we apply our methods to solve the following finite-sum [non]linear inclusion (FNI):
\myeq{eq:FedNI}{
\text{Find $x^{\star} \in \R^p$ such that:} \ 0 \in Gx^{\star} + Tx^{\star} \equiv \frac{1}{n} \sum_{i=1}^nG_ix^{\star} + Tx^{\star},
\tag{FNI}
}
where $T : \R^p  \rightrightarrows 2^{\R^p}$ is maximally monotone and $G_i : \R^p \rightrightarrows 2^{\R^p}$ ($\forall i\in [n]$) satisfy given conditions (possibly nonmonotone) specified later in Section~\ref{sec:FedNI_algorithms}.
If $n = 1$, then \eqref{eq:FedNI} is just a classical [non]linear inclusion \cite{Bauschke2011}.
Both \eqref{eq:NE} and \eqref{eq:FedNI} are called \textit{root-finding problems} in the literature.

Since we will focus on algorithmic development and convergence analysis, we assume throughout this paper that the solution sets $\zer{G} := \sets{x^{\star} \in \R^p : Gx^{\star} = 0}$ of \eqref{eq:NE} and  $\zer{G+T} := \sets{x^{\star} \in \R^p : 0 \in Gx^{\star} + Tx^{\star}}$ of \eqref{eq:FedNI} are nonempty. 

Though \eqref{eq:NE} and \eqref{eq:FedNI} look simple, it is well-known that \eqref{eq:NE} is equivalent to the problem of finding a fixed-point $x^{\star}$ of the operator $T := \Id - \eta G$, i.e. $x^{\star} = Tx^{\star}$, where $\Id$ is the identity operator, and $\eta > 0$ is any positive constant. 
Similarly, \eqref{eq:FedNI} also covers variational inequality problems as special cases (see Subsection~\ref{subsec:examples} for more concrete examples).
Therefore, \eqref{eq:NE} and \eqref{eq:FedNI} cope with  many fundamental problems across different fields by appropriately reformulating them into special cases of \eqref{eq:NE} or \eqref{eq:FedNI} (see, e.g., \cite{Bauschke2011,davis2022variance,reginaset2008,Facchinei2003,peng2016arock,phelps2009convex,Rockafellar2004,Rockafellar1976b,ryu2016primer}).
In addition, in this paper, we can handle a class of  nonmonotone operators $G$ (i.e. $G$ possesses a weak Minty solution), which allows us to tackle a  broader class of problems than recent works in \cite{davis2022variance,peng2016arock}.
Note that the finite-sum structure in \eqref{eq:FedNI} provides a unified template to model various optimization, game theory, and minimax problems in networks and distributed systems, statistical learning, and machine learning, including federated learning (FL), see, e.g., \cite{Bertsekas1989b,davis2022variance,mcmahan2021advances,Nedic2009,ryu2022large}.

%%%% Motivation
\vspace{0.5ex}
\textbf{Motivation.}
This work is motivated by the following aspects.

\textit{Large-scale applications.}
We focus on settings where the operator $G$ in \eqref{eq:NE} resides in a high-dimensional space $\R^p$ ($p\gg 1$) making operations on full-dimensional vectors computationally intractable.
Such models often arise in modern applications of machine learning, statistics, and data science, especially when model parameters are matrices or even tensors \cite{Bottou2018,lan2020first,sra2012optimization}.
One common approach to tackle these problems is block-coordinate methods (BC), which iteratively update one or a small number of block-coordinates  instead of the full-dimensional vector or matrix of the model parameters. 
This approach is though very classical \cite{auslender1971methodes,Ortega2000}, it has garnered widespread attention in recent years in optimization, monotone inclusions, and fixed-point problems, see, e.g., \cite{beck2015cyclic,combettes2015stochastic,latafat2019block,Nesterov2012,nesterov2017efficiency,peng2016arock,richtarik2016parallel,wright2015coordinate}.
As a concrete example, \cite{hsieh2013big} illustrated that block-coordinate methods can solve large-scale inverse covariance estimation problems involving millions of variables (i.e. $p$ in a range of $10^{12}$).
Such a large-scale problem is very hard or prohibited   to handle by classical optimization techniques.
This aspect motivates us to develop block-coordinate methods for \eqref{eq:NE}.

\textit{Limitation of current BC methods.}
While deterministic methods for solving both \eqref{eq:NE} and \eqref{eq:FedNI} are well developed under both monotone and certain nonmonotone settings \cite{Bauschke2011,Combettes2005,Davis2015,Facchinei2003,Lions1979,malitsky2015projected,popov1980modification,tseng2000modified}, designing efficient block-coordinate variants of these methods for \eqref{eq:NE} and \eqref{eq:FedNI} remains \revised{limited}, especially when the underlying operators are nonmonotone, see \cite{combettes2018asynchronous,combettes2015stochastic,peng2016arock} as a few  examples.
Most existing works focus on special cases of \eqref{eq:NE} and \eqref{eq:FedNI} such as optimization, convex-concave minimax, and supervised learning problems, see e.g., \cite{beck2015cyclic,latafat2019block,nesterov2017efficiency,Nesterov2012,richtarik2016parallel,wright2015coordinate}.
Moreover, developing Nesterov's accelerated coordinate methods for  \eqref{eq:NE} and \eqref{eq:FedNI} is still open.

\textit{Limitation of convergence analysis.}
In deterministic methods for \eqref{eq:NE}, we often use the residual norm $\norms{Gx^k}$ to characterize an approximate solution $x^k$, making it difficult to adapt this analysis to randomized block-coordinate (RBC) schemes, especially in accelerated variants.
In fact, we can see from \cite{attouch2020convergence,yoon2021accelerated} that generalizing Nesterov's accelerated methods from convex optimization to root-finding problems  is not straightforward, especially for explicit methods such as forward or gradient-type schemes.
Apart from a momentum term, accelerated methods for \eqref{eq:NE} are often aided by a ``gradient'' correction term $Gx^k - Gx^{k-1}$, see \cite{tran2022connection} as an example.
Recent attempts on designing deterministic accelerated methods for monotone inclusions and  monotone VIPs have been made such as in \cite{attouch2020convergence,chen2017accelerated,kim2021accelerated,mainge2021fast,tran2021halpern}.
These algorithms often achieve faster convergence rates than their classical counterparts on an appropriate residual norm.
However, it remains unclear if one can still adapt the theoretical analysis of these algorithms to BC methods for \eqref{eq:NE} and \eqref{eq:FedNI}.

\textit{Better convergence guarantees.}
In terms of convergence rates, classical methods for solving \eqref{eq:NE} and \eqref{eq:FedNI} such as dual averaging, extragradient-type, and splitting schemes \revised{(including forward-backward and Douglas-Rachford splitting)} can achieve up to $\BigO{1/k}$  (or sometimes $\SmallO{1/k}$)  convergence rates on the squared residual norm or on a gap function under the monotonicity of $G$ or some special nonmonotone extensions, where $k$ is the iteration counter, see, e.g., \cite{bohm2022solving,Davis2015,diakonikolas2021efficient,Facchinei2003,Nemirovskii2004,Nesterov2007a,pethick2022escaping}.
In accelerated methods, these rates can be boosted up to $\BigO{1/k^2}$ (or faster, $\SmallO{1/k^2}$) under the monotonicity. 
The $\BigO{1/k^2}$ rate matches the lower bound of convergence rates in different settings, see \cite{kim2016optimized,Nesterov2004,ouyang2018lower}.
Moreover, a $\BigO{1/k^2}$ rate can be achieved if $G$ satisfies a co-hypomonotone condition as shown in \cite{lee2021fast,tran2023extragradient}.
Nevertheless, we are not aware of any accelerated RBC optimistic gradient-type method for \eqref{eq:NE} that achieves similar convergence rates as ours.
Note that the convergence of iterate sequences in accelerated RBC algorithms for \eqref{eq:NE} remains open.

\textit{Broader applications.}
As mentioned earlier, since \eqref{eq:NE} can be reformulated equivalently to a fixed-point problem, theory and solution methods from one field can be applied to another and vice versa.
This relation expands the applicable scope of our methods.
Note that the \textit{nonexpansiveness} in the fixed-point theory is equivalent to the co-coerciveness of $G$ in \eqref{eq:NE}, see \cite[Proposition 4.11]{Bauschke2011}.
Furthermore, both \eqref{eq:NE} and \eqref{eq:FedNI} can cover many common applications in scientific computing beyond optimization models as discussed in \cite{peng2016arock,ryu2016primer}. 
For instance, they can be specified to handle linear systems, [composite] smooth and nonsmooth  optimization, feasibility problems, decentralized optimization, federated learning, adversarial training and learning, and [distributionally] robust optimization  as discussed in Subsection \ref{subsec:examples}, see also \cite{Ben-Tal2009,madry2018towards,peng2016arock,ryu2016primer}.

%%% Our goal and  challenges.
\vspace{0.5ex}
\textbf{Our goal and  challenges.}
Our goal in this paper is to advance a recent development in variational inequality problems (VIPs) and apply it to design new randomized block-coordinate (RBC) schemes for solving \eqref{eq:NE} and \eqref{eq:FedNI}.
Our approach relies on Popov's past-extragradient method \cite{popov1980modification} (also equivalent to the optimistic gradient (OG) method from \cite{mokhtari2020unified} in our context).
We argue that developing RBC methods for \eqref{eq:NE} and \eqref{eq:FedNI} faces at least three main challenges.

First, unlike the gradient method in convex or nonconvex optimization, which guarantees to converge under only the Lipschitz continuity of the objective gradient, its counterpart, the forward or fixed-point scheme for \eqref{eq:NE} may not converge when $G$ is monotone and Lipschitz continuous, see, e.g., \cite{pethick2022escaping}.  
This issue prevents  us from adopting RBC methods for optimization to  \eqref{eq:NE}.

Second,  RBC methods for an optimization or a minimax problem often use its objective function to build a Lyapunov function for convergence analysis, while this object does not exist in \eqref{eq:NE}.
Hitherto, most convergence analysis of existing methods for \eqref{eq:NE} rely on a squared distance to solution $\norms{x^k - x^{\star}}^2$, a Bregman distance, or a gap or restricted gap function \cite{combettes2015stochastic,kotsalis2022simple,peng2016arock}. 
Nevertheless, it remains unclear if any of these metrics can play a similar role as the objective function in optimization to analyze the convergence of RBC methods for \eqref{eq:NE}.

Third, it is difficult to apply existing convergence analysis techniques from deterministic algorithms such as extragradient methods \cite{korpelevich1976extragradient} and Halpern's fixed-point iteration \cite{halpern1967fixed} to RBC variants for \eqref{eq:NE} due to the dependence on the operator value $Gx^{k+1}$ (but not on the iterate $x^{k+1}$ alone) of the underlying potential or Lyapunov function, where $x^{k+1}$ is random.
This dependence prevents us from passing the expectation operator $\E[\cdot]$ between $Gx^{k+1}$ and $x^{k+1}$.
In this paper, we aim at addressing these challenges by tackling two settings of \eqref{eq:NE} in the high-dimensional regime.

%%%% Contribution.
\vspace{0.5ex}
\textbf{Our contribution.}
Our main contribution in this paper can be summarized as follows.
\begin{itemize}
\item[(a)] First, we propose a new RBC optimistic gradient  scheme (called \ref{eq:naRCOG_scheme}) to approximate a solution of \eqref{eq:NE} when $G$ is Lipschitz continuous and possesses a weak Minty solution (\textit{cf.} Assumption~\ref{as:A1}) as in \cite{diakonikolas2021efficient}.
Under an appropriate condition on the problem parameters, we establish a $\BigO{1/k}$ best-iterate convergence rate on $\E[\norms{Gx^k}^2]$. 

\item[(b)] Second, we develop a novel accelerated RBC optimistic gradient  algorithm (called \ref{eq:ARCOG_scheme})  to solve \eqref{eq:NE}.
This method achieves both $\BigO{1/k^2}$ and $\SmallO{1/k^2}$ last-iterate convergence rates on $\E\big[ \norms{Gx^k}^2 \big]$ and $\E\big[\norms{x^{k+1} - x^k}^2\big]$ under the co-coerciveness of $G$ (\textit{cf.} Assumption~\ref{as:A2}).
In addition, we  establish that $\norms{Gx^k}$ attains a $\SmallO{1/k}$ almost sure convergence rate, and the iterate sequence $\sets{x^k}$ converges to $x^{\star}\in\zer{G}$ almost surely.
We also utilize a change of variable to derive a practical variant of \ref{eq:ARCOG_scheme}, which avoids full dimensional operations on the iterates. 

\item[(c)] Third, we apply our methods from (a) and (b) to solve \eqref{eq:FedNI} which forms a basic model of many supervised machine learning tasks, including federated learning \cite{konevcny2016federated,li2020federated,mcmahan_ramage_2017,mcmahan2021advances}. 
We obtain two new federated learning-type algorithms.
The first one achieves a $\BigO{1/k}$ convergence rate under the existence of a weak Minty-type solution (\emph{cf.} Assumption~\ref{as:A3}).
The second algorithm is an accelerated RBC variant of the Douglas-Rachford splitting method, which achieves both $\BigO{1/k^2}$ and $\SmallO{1/k^2}$ last-iterate convergence rates, as well as almost sure convergence results.
\end{itemize}

%%% Related work and comparison.
\vspace{0.5ex}
\textbf{Related work and comparison.}
Let us compare our contribution with the most related works.

First, our algorithms are very simple to implement due to their single loop and simple operations.
They also significantly differ from existing methods.
Indeed, one of the most related works to our methods is \cite{peng2016arock}, which extends the asynchronous RBC method to \eqref{eq:NE}.
Though this algorithm is asynchronous, it is non-accelerated, and therefore, in our context, achieves $\BigO{1/k}$ and at most $\SmallO{1/k}$ convergence rates on $\E[\norms{Gx^k}^2]$ under a quasi-co-coerciveness of $G$.
Moreover, the form of our algorithm is also different from  \cite{peng2016arock} due to a ``double forward'' step and a momentum term (see \ref{eq:ARCOG_scheme}), while achieving $\BigO{1/k^2}$ and $\SmallO{1/k^2}$ last-iterate convergence rates. 

Second, our algorithms are also different from the RBC fixed-point-type method in \cite{combettes2015stochastic}, which renders from fixed-point iteration $x^{k+1} := x^k - \eta Gx^k$ or its variants.
These methods use a stronger assumption (i.e. averaged operators \cite{Bauschke2011}) than Assumption~\ref{as:A1} in \ref{eq:naRCOG_scheme} and achieve asymptotically convergence results \textit{almost surely}.
Under Assumption~\ref{as:A1}, the method in \cite{combettes2015stochastic} does not have a convergence guarantee (or not known yet). 
Note that  \cite{combettes2015stochastic} uses equivalent assumptions as \ref{eq:ARCOG_scheme}, but it only has asymptotic convergence guarantees.
In contrast,  our \ref{eq:ARCOG_scheme} method achieves both $\BigO{1/k^2}$ and $\SmallO{1/k^2}$ convergence rates on $\E[\norms{Gx^k}^2]$,  a $\SmallO{1/k}$-\textit{almost sure} convergence rate on $\norms{x^{k+1}-x^k}$ and $\norms{Gx^k}$, and an almost sure convergence of  $\sets{x^k}$ to a solution of \eqref{eq:NE}.

Third, compared to a recent work \cite{kotsalis2022simple},  our non-accelerated scheme, \ref{eq:naRCOG_scheme}, has a similar form as Algorithm 3 in \cite{kotsalis2022simple}.
However,  \cite{kotsalis2022simple} requires two strong assumptions which exclude our class of problems covered by Assumption~\ref{as:A1}.
Theorem 4.2 in \cite{kotsalis2022simple} requires a quasi-strongly monotone condition (1.7), i.e. $\langle Gx, x - x^{\star}\rangle \geq \mu V(x, x^{\star})$ for all $x$, where $\mu > 0$ and $V$ is a given Bregman distance.
This is much stronger than Assumption~\ref{as:A1}, and is closely related to the strong monotonicity.
Alternatively, Theorem 4.4 in \cite{kotsalis2022simple} requires a monotone condition (2.28) to achieve a convergence guarantee on a gap function. 
Note that establishing a convergence rate via a gap function only works for the ``monotone'' case.
It is not applicable to the nonmonotone setting as in our Assumption~\ref{as:A1}. 
Therefore, the analysis in \cite{kotsalis2022simple} is not applicable to our \ref{eq:naRCOG_scheme} even when $\rho = 0$ in Assumption \ref{as:A1}.

Fourth, unlike methods for convex problems, convergence analysis of algorithms for nonlinear equations such as \eqref{eq:NE} fundamentally relies on an appropriate potential or Lyapunov function, which is often non-trivial  to construct.
Moreover, it is still unclear if some recent analysis techniques for VIPs and minimax problems, e.g., in \cite{diakonikolas2020halpern,lee2021fast,lieder2021convergence,yoon2021accelerated} can be extended to [randomized] BC variants. 
In this paper, we follow a different approach compared to those, including convergence analysis techniques. 
Our analysis is also different from works directly using dual averaging and mirror descent methods and gap function techniques for VIPs such as \cite{Nemirovskii2004,Nesterov2007a}, but our rates can be converted into gap function using a trick from \cite{diakonikolas2020halpern} provided that $\dom{G}$ is bounded.

Finally, our applications to the finite-sum inclusion \eqref{eq:FedNI} in Section~\ref{sec:FedNI_algorithms} are new compared to \cite{davis2022variance} since our problem setting is more general than that of \cite{davis2022variance} and covers also a nonmonotone case. 
In addition, the condition on the product $L\rho$ in Theorem~\ref{th:FedOG_convergence} of Algorithm~\ref{alg:A1} does not depend on the probability $\mbf{p}$ of choosing block coordinates as in Theorem~\ref{th:RCOG_convergence}.
Furthermore, our  Algorithm~\ref{alg:A2} for solving \eqref{eq:FedNI} relies on an accelerated RBC variant of the Douglas-Rachford splitting scheme instead of a forward-type method as in \cite{davis2022variance}.
Both proposed algorithms can be implemented in a [synchronous] federated learning fashion such as FedAvg \cite{mcmahan2017communication} and FedProx \cite{Li_MLSYS2020}, but they can overcome major challenges of FL (see Section~\ref{sec:FedNI_algorithms} for more details) and can solve more general problems than a majority of existing methods, including those in \cite{du2021fairness,Li_MLSYS2020,mcmahan2017communication,sharma2022federated,tarzanagh2022fednest}.

%%%% Paper organization.
\vspace{0.5ex}
\textbf{Paper organization.}
The rest of this paper is organized as follows.
In Section~\ref{sec:background} we briefly review some background related to \eqref{eq:NE} and \eqref{eq:FedNI}, recall the optimistic gradient method and its variants, and present several relevant examples.
Section~\ref{sec:Na_RCOG_method} develops a new non-accelerated randomized block-coordinate optimistic gradient (RCOG) method to solve \eqref{eq:NE} and establishes its convergence.
Section~\ref{sec:AcRCOG_method} proposes a novel accelerated variant of our RCOG method and analyzes its convergence.
Section~\ref{sec:FedNI_algorithms} applies our algorithms to solve \eqref{eq:FedNI} in the context of federated learning and investigates their convergence.
Section~\ref{sec:num_examples} presents three numerical examples to verify our algorithms.
For the sake of presentation, most technical proofs are deferred to the appendix.

%%%% Notations and terminologies.
\vspace{0.5ex}
\textbf{Notations and terminologies.}
We work with a finite dimensional space $\Espace$ equipped with the standard inner product $\iprods{\cdot, \cdot}$ and Euclidean norm $\norms{\cdot}$.
For a multivalued mapping $G : \Espace \rightrightarrows 2^{\Espace}$, $\dom{G} = \set{x \in \Espace : Gx \not= \emptyset}$ denotes its domain, $\gra{G} = \set{(x, u) \in \dom{G} \times \Espace : u \in Gx}$ denotes its graph, where $2^{\Espace}$ is the set of all subsets of $\Espace$.
Furthermore, we assume that the variable $x$ of \eqref{eq:NE} is decomposed into $n$-blocks as $x = [x_1, x_2, \cdots, x_n]$, where $x_i \in \R^{p_i}$ is the $i$-th block subvector for $i\in [n] := \sets{1, 2, \cdots, n}$, $1 \leq n \leq p$, and $p_1 + \cdots + p_n = p$.
Given $G$ in \eqref{eq:NE}, we denote $[Gx]_i$ the $i$-th block coordinate of $Gx$ such that $Gx = [[Gx]_1, \cdots, [Gx]_n]$.
We also denote $G_{[i]}x = [0, \cdots, 0, [Gx]_i, 0, \cdots, 0]$ so that only the $i$-th block is presented and other blocks are zero.
For any vector $x = [x_1,\cdots, x_n] \in \R^p$, we define a block-weighted norm of $x$ as $\norms{x}_{\sigma} := \left( \sum_{i=1}^n\sigma_i\norms{x_i}^2 \right)^{1/2}$ associated with a weight vector $\sigma \in \R^n_{++}$ (the set of all positive real-valued vectors). We also use the conventions  $\sigma \circ x := [\sigma_1x_1, \cdots, \sigma_nx_n]$ and $\sigma \circ \norms{x}^2 := \sum_{i=1}^n\sigma_i\norms{x_i}^2$.
For a symmetric matrix $Q$, $\lambda_{\min}(Q)$ stands for the smallest eigenvalue of $Q$.
We denote $\Fc_k$ to be the smallest $\sigma$-algebra generated by $\sets{x^0, x^1, \cdots, x^k}$ collecting all the randomness up to the $k$-th iteration of our algorithms.
We also use $\Exps{k}{\cdot} := \E_{i_k}[ \cdot \mid \Fc_k ]$ for the conditional expectation and $\Exp{\cdot}$ for the total or full expectation.

%%%%%%%%%%%%%%%%%%%%%%%%%%%%%%%%%%%%%%%%%%%%%
%%%% 2. Background and Preliminary Results.
%%%%%%%%%%%%%%%%%%%%%%%%%%%%%%%%%%%%%%%%%%%%%
\beforesec
\section{Background, Preliminary Results, and Examples.}\label{sec:background}
\aftersec
We first review necessary background on monotone operators and related concepts used in this paper.
Then, we recall the optimistic gradient method and its variants, including a Nesterov's accelerated variant.
Finally, we present the equivalence between \eqref{eq:NE} and \eqref{eq:FedNI} and some concrete examples of \eqref{eq:NE} and \eqref{eq:FedNI}.

%%%% 2.1. Monotone operators and related concepts
\beforesubsec
\subsection{Monotone operators and related concepts.}
\aftersubsec
%
%%% a. Monotonicity.
For a multivalued mapping $G : \Espace \rightrightarrows 2^{\Espace}$, we say that $G$ is monotone if $\iprods{u - v, x - y} \geq 0$ for all $(x, u), (y, v) \in \gra{G}$.
$G$ is said to be $\rho$-co-hypomonotone  if $\iprods{u - v, x - y} \geq -\rho\norms{u - v}^2$ for all $(x, u),  (y, v) \in \gra{G}$, where $\rho > 0$ is a given parameter.
Note that a $\rho$-co-hypomonotone operator is not necessarily monotone.
If $G$ is single-valued, then these conditions reduce to $\iprods{Gx - Gy, x - y} \geq 0$ and $\iprods{Gx - Gy, x - y} \geq -\rho\norms{Gx - Gy}^2$, respectively.
We say that $G$ is maximally [co-hypo]monotone if $\gra{G}$ is not properly contained in the graph of any other [co-hypo]monotone operator.

%%% b. Lipschitz continuity and co-coerciveness.
A single-valued operator $G$ is said to be $L$-Lipschitz continuous if $\norms{Gx - Gy} \leq L\norms{x - y}$ for all $x, y\in\dom{G}$, where $L \geq 0$ is a Lipschitz constant. 
If $L = 1$, then we say that $G$ is nonexpansive. %, while if $L \in [0, 1)$, then  $G$ is called $L$-contractive, and $L$ is its contraction factor.
We say that $G$ is $\beta$-co-coercive if $\iprods{u - v, x - y} \geq \beta\norms{u - v}^2$ for all $(x, u),  (y, v) \in\gra{G}$.
%If $\beta = 1$, then we say that $G$ is firmly nonexpansive.
If $G$ is $\beta$-co-coercive, then by using the Cauchy-Schwarz inequality, $G$ is also monotone and $\frac{1}{\beta}$-Lipschitz continuous, but the reverse statement is not true in general.

%%%% c. Resolvent and reflection operators.
The operator $J_Gx := \set{y \in \Espace : x \in y + Gy}$ is called the resolvent of $G$, often denoted by $J_Gx = (\Id + G)^{-1}x$, where $\Id$ is the identity mapping.
If $G$ is monotone, then $J_G$ is singled-valued, and if $G$ is maximally monotone, then $J_G$ is singled-valued and $\dom{J_G} = \Espace$.

\beforesubsec
\subsection{The optimistic gradient method and its variants.}\label{subsec:FRB_splitting}
\aftersubsec
Popov's past extragradient method in \cite{popov1980modification}  (also equivalent to the optimistic gradient \eqref{eq:OG_scheme}  scheme in \cite{mokhtari2020unified}) for solving \eqref{eq:NE} can be written as follows. 
Starting from $x^0 \in \dom{G}$, set $x^{-1} := x^0$, and for $k \geq 0$, we update
\myeq{eq:OG_scheme}{
x^{k+1} := x^k - \eta_k \big( Gx^k - \gamma_k  Gx^{k-1} \big),
\tag{OG}
}
where $\eta_k > 0$ and $\gamma_k \in (0, 1)$  are given parameters.
Though \ref{eq:OG_scheme} is a ``double forward'' scheme, it only requires one evaluation $Gx^k$ of $G$ per iteration.
If we denote $\hat{\eta}_k := \frac{\eta_k}{2}$ and choose $\gamma_k := \frac{1}{2}$, then \ref{eq:OG_scheme} becomes $x^{k+1} = x^k - \hat{\eta}_k(2Gx^k - Gx^{k-1})$, which exactly reduces to the forward-reflected-backward splitting method proposed in \cite{malitsky2020forward} for solving \eqref{eq:NE} (the backward step is identical) when $G$ is monotone and $L$-Lipschitz continuous, where $\hat{\eta}_k \in (0, \frac{1}{2L})$.
If $\hat{\eta}_k := \hat{\eta}$ is fixed, then \ref{eq:OG_scheme} can also be rewritten equivalently to $y^{k+1} = y^k - \hat{\eta} G(2y^k - y^{k-1})$ with $y^k := x^k + \hat{\eta} Gx^{k-1}$, which is known as the reflected gradient method in \cite{malitsky2015projected}.
The \ref{eq:OG_scheme} scheme can be rewritten as $x^{k+1} = x^k - \alpha_k Gx^k - \gamma_k\eta_k(Gx^k - Gx^{k-1})$, where $\alpha_k := \eta_k(1-\gamma_k)  > 0$.
This form can be viewed as a gradient/forward method with a ``gradient'' correction term $\gamma_k\eta_k(Gx^k - Gx^{k-1})$ to solve \eqref{eq:NE}.

We can modify \ref{eq:OG_scheme} by adding a momentum term $\theta_k(x^k - x^{k-1})$ to obtain the following momentum OG (or accelerated OG) method to solve \eqref{eq:NE}:
\myeq{eq:AOG_scheme}{
x^{k+1} := x^k + \theta_k(x^k - x^{k-1}) - \eta_k\big(Gx^k - \gamma_kGx^{k-1} \big),
\tag{AOG}
}
where $\theta_k \geq 0$ is a given parameter.
If we  additionally introduce $y^{k+1} := x^k - \alpha_kGx^k$, then we can equivalently transform \ref{eq:AOG_scheme} into the following form:
\myeq{eq:NesOG_scheme}{
\arraycolsep=0.2em
\left\{\begin{array}{lcl}
y^{k+1} &:= & x^k - \alpha_kGx^k, \vspace{1ex}\\
x^{k+1} & := & y^{k+1} + \theta_k(y^{k+1} - y^k) + \nu_k(x^k - y^{k+1}), 
\end{array}\right.
\tag{NesOG}
}
where $\alpha_k := \frac{\gamma_{k+1}\eta_{k+1}}{\theta_{k+1}}$ and $\nu_k := 1 + \theta_k - \frac{\theta_{k+1}\eta_k}{\gamma_{k+1}\eta_{k+1}}$.

If $\nu_k = 0$, then \ref{eq:NesOG_scheme} is identical to Nesterov's accelerated methods in convex optimization with $G(\cdot) = \nabla{f}(\cdot)$, see  \cite{Nesterov1983,Nesterov2004}.
However, as discussed in \cite{attouch2020convergence,mainge2021accelerated}, the term $\nu_k(x^k - y^{k+1})$ plays a crucial role to establish  the convergence of accelerated methods for solving \eqref{eq:NE}.

\beforesubsec
\subsection{The equivalence between \eqref{eq:NE} and \eqref{eq:FedNI}.}\label{subsec:NEvsNI}
\aftersubsec
It is clear that \eqref{eq:FedNI} covers \eqref{eq:NE} as a special case when $T = 0$.
However, under appropriate assumptions, \eqref{eq:FedNI} can equivalently be reformulated into \eqref{eq:NE}.
Let us discuss some common reformations below.

The first common equivalent form of \eqref{eq:FedNI} is the resolvent equation $J_{\lambda(G+T)}x = 0$, where $J_{\lambda(G+T)} := (\Id + \lambda(G+T))^{-1}$ is the resolvent of $\lambda(G+T)$ for any $\lambda > 0$.
This reformulation is used to develop proximal-point methods for solving \eqref{eq:FedNI}.
However, evaluating the resolvent $J_{\lambda(G+T)}$ is usually challenging in most cases, limiting the applicability of  proximal-point methods in practice.

Next, if $G$ is $\beta$-co-coercive, then finding a solution $x^{\star}$ of \eqref{eq:FedNI}  is equivalent to solving
\myeq{eq:FBS_operator}{
F_{\lambda}x^{\star} = 0, \qquad \text{where} \quad F_{\lambda}x :=  \lambda^{-1}(x - J_{\lambda T}(x - \lambda Gx)) \quad \text{for any} \quad \lambda > 0.
}
The operator $F_{\lambda}$ is called the forward-backward splitting (FBS) residual. 
It is well-known \cite{Bauschke2011} that $F_{\lambda}$ is $\frac{\lambda(4\beta - \lambda)}{4\beta}$-co-coercive, provided that $0 < \lambda < 4\beta$.
Hence, $F_{\lambda}$ fulfills Assumption~\ref{as:A2} below.
Alternatively, we can also reformulate \eqref{eq:FedNI} equivalently to $H_{\lambda}y^{\star} = 0$, where $H_{\lambda}(\cdot) := G ( J_{\lambda T}(\cdot)) + \lambda^{-1}(\Id - J_{\lambda T})(\cdot)$ is the backward-forward splitting (BFS) operator (see Proposition 2.4. in \cite{attouch2018backward}).
Then, $x^{\star}$ solves \eqref{eq:FedNI} iff $x^{\star} := J_{\lambda T}y^{\star}$ solves  $H_{\lambda}y^{\star} = 0$.
This reformulation is very useful to develop efficient stochastic algorithms for solving the finite-sum setting \eqref{eq:FedNI}, see \cite{davis2022variance} as an example.

If both $G$ and $T$ are maximally monotone, then we can reformulate \eqref{eq:FedNI} equivalently to 
\myeq{eq:DRS_operator}{
V_{\beta}u^{\star} = 0, \qquad \text{where} \quad V_{\beta}u := \beta^{-1}(J_{\beta T}u - J_{\beta T} (2J_{\beta G}u - u)) \quad \text{for any} \quad \beta > 0.
}
The operator $V_{\beta}$ is called the Douglas-Rachford splitting (DRS) residual, which is $\beta$-co-coercive.
Note that $u^{\star}$ solves \eqref{eq:DRS_operator} iff  $x^{\star} := J_{\beta T}u^{\star}$ solves \eqref{eq:FedNI}. 
Furthermore, if $T := C + \hat{T}$, where $\hat{T}$ is maximally monotone and $C$ is $\beta$-co-coercive, then we can reformulate \eqref{eq:FedNI} equivalently to \eqref{eq:NE} using  a three-operator splitting residual in \cite{Davis2015}.

Finally, if $G$ is possibly nonmonotone, e.g., $G$ satisfies a variant of Assumption~\ref{as:A1} below, then we can reformulate \eqref{eq:FedNI} equivalently to \eqref{eq:NE} using Tseng's forward-backward-forward splitting (FBFS) residual from \cite{tseng2000modified}.
We refer to Subsection~\ref{subsec:FOG} below for more details.

\beforesubsec
\subsection{Examples.}\label{subsec:examples}
\aftersubsec
Both settings \eqref{eq:NE} and \eqref{eq:FedNI} cover many fundamental problems in practice.
Let us recall some important special cases of these problems here.
We also refer to \cite{davis2022variance,peng2016arock,ryu2016primer} for further details and additional examples.

\beforesubsubsec
\subsubsection{Optimization problems.}\label{subsubsec:opt_probs}
\aftersubsubsec
Consider the following composite optimization problem:
\myeq{eq:opt_prob}{
\min_{x \in \R^p} \big\{ \phi(x) := f(x) + g(x) \big\},
}
where $f : \R^p \to \Rext$ is given and $g : \R^p\to\Rext$ is proper, closed, and convex.
Problem \eqref{eq:opt_prob} covers many applications such as linear and nonlinear least-squares, logistic regression, optimization over networks, and distributed optimization \cite{Bauschke2011,Bottou2018,Combettes2011,combettes2011proximal,sra2012optimization,wright2017optimization}.  

As a concrete example, if $f(x) := \frac{1}{n} \sum_{i=1}^n\ell(\iprods{Z_i, x}; y_i)$ represents an empirical  loss associated with a dataset $\sets{(Z_i, y_i)}_{i=1}^n$ and $g(x) := \tau R(x)$ is a regularizer to promote desired structures of $x$ (e.g., sparsity or low-rankness), then \eqref{eq:opt_prob} covers many learning problems such as linear regression with $\ell(\tau,s) := \frac{(\tau - s)^2}{2}$, logistic regression with $\ell(\tau, s) := \log(1 + \exp(-s\tau))$), and soft-margin support vector machine (SVM) with the hinge loss $\ell(\tau,s) := \max\sets{0, 1-\tau s}$.  

Let us recall $\partial{f}$ and $\partial{g}$ as the subdifferentials or gradients of $f$ and $g$, respectively.
Then, under appropriate regularity assumptions \cite{Bauschke2011,Rockafellar2004}, the optimality condition of \eqref{eq:opt_prob} is
\myeq{eq:opt_cond_of_opt_prob}{
0 \in \partial{f}(x^{\star}) + \partial{g}(x^{\star}).
}
If \eqref{eq:opt_prob} is convex, then $x^{\star}$ solves \eqref{eq:opt_cond_of_opt_prob} iff it solves \eqref{eq:opt_prob}.
Clearly, \eqref{eq:opt_prob} is a special case of \eqref{eq:FedNI} with $G := \partial{f}$ and $T := \partial{g}$.
If $g = \delta_{\Xc}$, the indicator of a nonempty and closed set $\Xc$, then \eqref{eq:opt_prob} also covers constrained optimization problems.
If $g = 0$ and $f$ is differentiable, then \eqref{eq:opt_cond_of_opt_prob} reduces to $\nabla{f}(x^{\star}) = 0$, which is a special case of \eqref{eq:NE} with $G := \nabla{f}$.
Our methods developed in this paper will reduce to new randomized block-coordinate gradient-type algorithms for solving \eqref{eq:opt_prob}.

\beforesubsubsec
\subsubsection{Minimax problems.}\label{subsubsec:minimax_probs}
\aftersubsubsec
Consider the following minimax problem:
\myeq{eq:minimax_prob}{
\min_{u\in\R^{p_1}}\max_{v \in \R^{p_2}}\Big\{ \Lc(u, v) := f(u) + \Psi(u, v) - g(v) \Big\},
}
where $f : \R^{p_1} \to\Rext$ and $g : \R^{p_2}\to\Rext$ are proper, closed, and convex, and $\Psi : \R^{p_1}\times\R^{p_2}\to\R$ is a given joint objective function.
This minimax problem is a central model in robust and distributionally robust optimization \cite{Ben-Tal2001,rahimian2019distributionally,namkoong2016stochastic}, two-person games \cite{ho2022game,kuhn1996work}, fairness machine learning \cite{du2021fairness,martinez2020minimax}, and generative adversarial networks (GANs) \cite{arjovsky2017wasserstein,daskalakis2018training,goodfellow2014generative}, among others.

As a concrete example, if $f(u) := \delta_{\Delta_{p_1}}(u)$ and $g(v) := \delta_{\Delta_{p_2}}(v)$ are the indicators of the standard simplexes $\Delta_{p_1}$ and $\Delta_{p_2}$, respectively, and $\Psi(u, v) := \iprods{Au, v}$ is a bilinear form, where $A$ is a pay-off matrix, then \eqref{eq:minimax_prob} covers the classical bilinear game problem.
Another common example is the following non-probabilistic robust optimization model obtained from  Wald's minimax model: $\min_{u \in \R^{p_1}}\big\{ \phi(u) := \psi(u) + f(u) \equiv \max_{ v \in \mathcal{V} } \Psi(u, v) + f(u) \big\}$, where $u$ is a decision variable, $v$ represents an uncertainty vector on $\mathcal{V}\subset\R^{p_2}$, and the objective term $\psi(u) := \max_{ v \in \mathcal{V} } \Psi(u, v)$ is the worst-case risk over all possible candidates $v$ in $\mathcal{V}$ (see \cite{Ben-Tal2001} for concrete instances of this model).

Again, under appropriate regularity conditions  \cite{Bauschke2011,Rockafellar2004}, the optimality condition of \eqref{eq:minimax_prob} becomes
\myeq{eq:minimax_opt_cond}{
0 \in \begin{bmatrix} \partial_u{\Psi}(u^{\star}, v^{\star})  \\  -\partial_v{\Psi}(u^{\star}, v^{\star}) \end{bmatrix} + \begin{bmatrix} \partial{f}(u^{\star}) \\   \partial{g}(v^{\star}) \end{bmatrix}.
}
If $\Psi$ is convex-concave, then \eqref{eq:minimax_opt_cond} is necessary and sufficient for $(u^{\star}, v^{\star})$ to be an optimal solution of \eqref{eq:minimax_prob}.
Otherwise, it is only a necessary condition.
If $f$ and $g$ are the indicators of given convex sets, then \eqref{eq:minimax_prob} also covers constrained minimax problems.
Clearly, if we define $x := [u, v]$, $G := [\partial_u{\Psi}, -\partial_v{\Psi}]$, and $T := [\partial{f}, \partial{g}]$, then \eqref{eq:minimax_opt_cond} can be viewed as a special case of \eqref{eq:FedNI}.
If $f = 0$ and $g = 0$, then the optimality condition \eqref{eq:minimax_opt_cond} reduces to $Gx^{\star} = 0$, a special case of \eqref{eq:NE}.

\beforesubsubsec
\subsubsection{Variational Inequality Problems (VIP).}\label{subsubsec:VIP}
\aftersubsubsec
If $T = \Nc_{\Xc}$, the normal cone of a nonempty, closed, and convex set $\Xc$, then \eqref{eq:FedNI} reduces to 
\myeq{eq:VIP}{
\textrm{Find $x^{\star} \in \Xc$ such that:} \   \iprods{Gx^{\star}, x - x^{\star}} \geq 0, \  \textrm{for all} \ x \in \Xc.
\tag{VIP}
}
More generally, if $T = \partial{g}$, the subdifferential of a convex function $g$, then \eqref{eq:FedNI} reduces to a mixed VIP.
The formulation \eqref{eq:VIP} covers many well-known problems in practice, including unconstrained and constrained optimization, minimax problems, complementarity problems, and Nash's equilibria, see also \cite{Facchinei2003,Konnov2001} for concrete examples in traffic networks and economics.

\beforesubsubsec
\subsubsection{Fixed-point problems.}\label{subsubsec:fixed_point}
\aftersubsubsec
The fixed-point problem is fundamental in computational mathematics, which has various applications in numerical analysis, ordinary and partial differential equations, engineering, and physics \cite{agarwal2001fixed,Bauschke2011,Combettes2011a}. 
The classical fixed-point problem is written as
\myeq{eq:fixed_point}{
\textrm{Find $x^{\star} \in \R^p$ such that:} \quad  x^{\star} = Fx^{\star},
\tag{FP}
}
where $F : \R^p \to \R^p$ is a given mapping.
This problem is equivalent to \eqref{eq:NE} with $G := \Id - F$, where $\Id$ is the identity mapping.
Our algorithms developed in this paper for \eqref{eq:NE} can be applied to solve certain classes of \eqref{eq:fixed_point}.
We can also generalize \eqref{eq:fixed_point} to a multivalued mapping fixed-point problem $x^{\star} \in Fx^{\star}$ (known as Kakutani's fixed-point problem), which is also equivalent to \eqref{eq:FedNI}.

%%%%%%%%%%%%%%%%%%%%%%%%%%%%%%%%%%%%%%%%
%%% 3. The Non-Accelerated Randomized Block-Coordinate OG Method
%%%%%%%%%%%%%%%%%%%%%%%%%%%%%%%%%%%%%%%%
\beforesec
\section{The Non-Accelerated Randomized Block-Coordinate OG Method.}\label{sec:Na_RCOG_method}
\aftersec
In this section, we develop a new randomized block-coordinate optimistic gradient (RCOG) method to solve \eqref{eq:NE} that satisfies the following assumption.

%%% Assumption A.1.
\begin{assumption}\label{as:A1}
The operator $G$ in \eqref{eq:NE}  is $L_i$-block coordinate-wise Lipschitz continuous, i.e. there exists $L := (L_1,\cdots, L_n) \in [0, +\infty)^n$  such that for any $x, y \in \dom{G}$:
\myeq{eq:CE_assumption_A1_a}{
\norms{ [Gx]_i - [Gy]_i }  \leq   L_i \norms{ x_i - y_i }, \quad \text{for all $i \in [n]$}. 
}
\textbf{Weak-Minty solution:} There exist $\rho := (\rho_1, \cdots, \rho_n) \in [0,  L]$ and $x^{\star}\in\zer{G}$ of \eqref{eq:NE} such that:
\myeq{eq:CE_assumption_A1_b}{
\begin{array}{ll}
\iprods{Gx, x - x^{\star}}  \geq - \rho \circ \norms{Gx}^2, \quad \text{for all $x \in \dom{G}$}.
\end{array}
}
\end{assumption}

The expression \eqref{eq:CE_assumption_A1_a} shows that $G$ is block-coordinate-wise Lipschitz continuous.
It also implies that $G$ is $\bar{L}$-Lipschitz continuous with $\bar{L} := (\sum_{i=1}^nL_i^2)^{1/2}$.
The Lipschitz continuity of $G$ is required in most numerical methods for solving \eqref{eq:NE} and its generalizations, see, e.g., \cite{Facchinei2003}. 

The condition \eqref{eq:CE_assumption_A1_b} is also referred to as a weak Minty solution of \eqref{eq:NE}, see \cite{diakonikolas2021efficient} for more details and examples.
If $\rho_i = \hat{\rho} \geq 0$ for all $i \in [n]$, then it exactly reduces to the weak Minty condition $\iprods{Gx, x - x^{\star}} \geq -\hat{\rho}\norms{Gx}^2$ used in many recent works, including \cite{bohm2022solving,diakonikolas2021efficient,pethick2022escaping}.
It is weaker than the $\hat{\rho}$-co-hypomonotonicity of $G$  \cite{lee2021fast,tran2023extragradient}.
We can refer to \eqref{eq:CE_assumption_A1_b} as a $\hat{\rho}$-``star'' co-hypomonotonicity of $G$.
Note that under condition \eqref{eq:CE_assumption_A1_b}, $G$ is not necessarily monotone.
If $\rho = 0$, then $G$ is star-monotone, i.e. $\iprods{Gx, x - x^{\star}} \geq 0$ for all $x \in \dom{G}$.
As a concrete example, if we consider a linear mapping $Gx := \mbf{G}x + \mbf{g}$ for an invertible symmetric matrix $\mbf{G}$ with $\lambda_{\min}(\mbf{G}^{-1}) < 0$, and a given $\mbf{g}$, then $G$ is not monotone, but $\hat{\rho}$-co-hypomonotone with $\hat{\rho} := -\lambda_{\min}(\mbf{G}^{-1}) > 0$. Hence, $G$ satisfies \eqref{eq:CE_assumption_A1_b}.

In contrast to \eqref{eq:CE_assumption_A1_a}, many works such as  \cite{kotsalis2022simple,peng2016arock} rely on a quasi-strongly monotone condition $\iprods{Gx, x - x^{\star}}  \geq \mu \norms{x - x^{\star}}^2$ for $x \in \dom{G}$, where $\mu > 0$.
This assumption is slightly weaker than the strong monotonicity of $G$, but significantly different from \eqref{eq:CE_assumption_A1_b} due to different right-hand side.
Note that the quasi-strong monotonicity of $G$ also implies the error bound condition $\norms{x - x^{\star}} \leq \mu^{-1} \norms{Gx}$ as in \cite{luo1993error}.
However, error bound conditions are often imposed locally to establish a linear convergence of  optimization methods, see, e.g., \cite{drusvyatskiy2018error} for further discussion.

\beforesubsec
\subsection{The derivation of algorithm.}\label{subsec:RCOG_development}
\aftersubsec
Since our method is randomized, we introduce a random variable $i_k$  on $[n] := \sets{1, 2, \cdots, n}$ to select the $i$-th block coordinate with the following probability:
\myeq{eq:prob_cond1}{
\Prob{i_k = i} = \mbf{p}_i, \quad \text{for all}~i \in [n],
}
where $\mbf{p}_i > 0$ for all $i\in [n]$ and $\sum_{i=1}^n\mbf{p}_i = 1$ is a given probability distribution.
We also denote $\mbf{p}_{\min} := \min_{i\in [n]} \mbf{p}_i > 0$.
If $\mbf{p}_i = \frac{1}{n}$, then $i_k$ is a uniform random variable.
Otherwise, we also cover non-uniform randomized block-coordinate methods.

Our \textbf{\textit{randomized block-coordinate optimistic gradient}} (RCOG) method is described as follows.
\textit{Given an initial point $x^0 \in \dom{G}$, we set $x^{-1} := x^0$, and at each iteration $k \geq 0$, we randomly sample an i.i.d. block-coordinate $i_k \in [n]$ satisfying \eqref{eq:prob_cond1}, and update:
\myeq{eq:naRCOG_scheme}{
x^{k+1} := x^k -  \tfrac{\eta_{i_k} }{\mbf{p}_{i_k}}  \left( G_{[i_k]}x^k - \gamma_{i_k}  G_{[i_k]}x^{k-1} \right),
\tag{RCOG}
}
where  $\eta_{i_k} > 0$ and $\gamma_{i_k} \in (0, 1)$ are given parameters corresponding to the $i_k$-th block coordinate $($determined later$)$, and $G_{[i]}x = [\mbf{0}, \mbf{0}, \cdots, \mbf{0}, [Gx]_i, \mbf{0}, \cdots, \mbf{0}]$ such that only the $i$th-block is $[Gx]_i$, and other blocks are zero.}

If $n = 1$ (i.e. one block), then \ref{eq:naRCOG_scheme} reduces to \ref{eq:OG_scheme} in the deterministic case.
For each block-coordinate $i$, we use different constant parameters $\eta_i$ and $\gamma_i$ at all iterations.
In terms of algorithmic form, \ref{eq:naRCOG_scheme} uses two block-coordinate evaluations $[Gx^k]_{i_k}$ and $[Gx^{k-1}]_{i_k}$ of $G$ evaluated at the two consecutive iterates $x^k$ and $x^{k-1}$, respectively.

\beforesubsec
\subsection{Lyapunov function and descent lemma.}\label{subsec:Na_RCOG}
\aftersubsec
We introduce a new Lyapunov function:
\begin{equation}\label{eq:RBC_PEG_potential}
\Pc_k := \norms{x^k  - x^{\star} + \gamma\eta  \circ Gx^{k-1} }^2 +  \norms{x^k - x^{k-1}}_{\sigma}^2,
\end{equation} 
where $\sigma := (\sigma_1,\cdots, \sigma_n) \in \R^n_{++}$ is a given weight vector. 
Then, the following lemma provides a key descent property to analyze the convergence of \ref{eq:naRCOG_scheme}, whose proof is in Appendix~\ref{apdx:subsec:le:RCOG_descent}.

%%% Lemma 4.1.
\begin{lemma}\label{le:RCOG_descent}
Suppose that $\zer{G} \neq\emptyset$ and Assumption~\ref{as:A1} holds.
Let $\sets{x^k}$ be generated by \ref{eq:naRCOG_scheme} to solve \eqref{eq:NE} and $\Pc_k$ be defined by \eqref{eq:RBC_PEG_potential}.
If $L_i^2\gamma_i^2\eta_i^2 \big(  2 - \mbf{p}_i  + 2\sigma_i \big)  \leq \mbf{p}_i \sigma_i$ for all $i \in [n]$, then 
\myeq{eq:RCOG_descent}{
\arraycolsep=0.2em
\begin{array}{lcl}
\E_k\big[\Pc_{k+1}\big] &\leq & \Pc_k   -    \sum_{i=1}^n 2\gamma_i\eta_i(1 - \gamma_i)\big[ \eta_i  - \frac{\rho_i}{\gamma_i}  -  \frac{\eta_i(1 - \gamma_i)(1 + \sigma_i) }{\gamma_i\mbf{p}_i} \big] \cdot \norms{ [Gx^k]_i}^2.
\end{array}
}
\end{lemma}

%%%%%%%%%%%%%%%%%%%%%%%%%%left here!!!
\beforesubsec
\subsection{Main result 1: Convergence of \ref{eq:naRCOG_scheme}.}\label{subsec:main_result1}
\aftersubsec
Utilizing Lemma~\ref{le:RCOG_descent}, we prove our first main result in the following theorem.

%%% Theorem 5.1.
\begin{theorem}\label{th:RCOG_convergence}
Suppose that $\zer{G} \neq\emptyset$ and Assumption~\ref{as:A1} holds for \eqref{eq:NE}.
Let  $\Pc_k$ be defined by \eqref{eq:RBC_PEG_potential} for some $\sigma \in \R^n_{++}$ and $\gamma_i$ is chosen such that  $\frac{1+ \sigma_i}{1+\sigma_i + \mbf{p}_i} < \gamma_i < 1$ for $i \in [n]$.
Suppose that 
\myeq{eq:stepsize_cond1}{
L_i\rho_i < \kappa_i\sqrt{\mbf{p}_i}\left[ 1 -  \tfrac{(1+\sigma_i)(1-\gamma_i)}{\gamma_i\mbf{p}_i } \right] \quad \text{and} \quad 0 < \eta_i \leq \tfrac{\kappa_i\sqrt{\mbf{p}_i}}{\gamma_i L_i},
}
for all $i \in [n]$, where $\kappa_i^2 :=   \frac{\sigma_i}{2(1 + \sigma_i)}$.
Then, the following estimate holds:
\myeq{eq:RCOG_descent1}{
\arraycolsep=0.2em
\begin{array}{lcl}
\E\big[\Pc_{k+1}\big] &\leq & \E\big[ \Pc_k \big]  -   \psi \circ \E\big[ \norms{ Gx^k }^2 \big], 
\end{array}
}
where $\psi_i := 2 \gamma_i\eta_i(1 - \gamma_i)\big[ \eta_i - \frac{\rho_i}{\gamma_i} - \frac{\eta_i(1 + \sigma_i)(1-\gamma_i)}{\gamma_i\mbf{p}_i } \big] > 0$ and $\psi := (\psi_1, \cdots, \psi_n)$.
Moreover, for $\psi_{\min} := \min\sets{ \psi_i : i \in [n]} > 0$, the following convergence bound also holds:
\myeq{eq:RBC_PEG_ergodic_rate}{
\frac{\psi_{\min} }{ K + 1 }\sum_{k=0}^K \E\big[ \norms{Gx^k}^2 \big] \leq  \frac{1}{K+1}\sum_{k=0}^{K}  \psi \circ \E\big[\norms{ Gx^k }^2 \big] \leq  \frac{3\norms{x^0 - x^{\star}}^2}{ K+1 }.
}
\end{theorem}

%%% Proof of Theorem 5.1.
\proof{\textbf{Proof.}}
First, we denote $\kappa_i^2 :=  \frac{  \sigma_i  }{  2(1 + \sigma_i) }$.
Since $\mbf{p}_i \in (0, 1)$, if $L_i\gamma_i\eta_i  \leq \kappa_i \sqrt{\mbf{p}_i}$ for all $i\in [n]$, then the condition $L_i^2\gamma_i^2\eta_i^2 [ 2(1+\sigma_i) - \mbf{p}_i  ]  \leq \mbf{p}_i\sigma_i$  in Lemma~\ref{le:RCOG_descent} holds.
Moreover, the condition $L_i\gamma_i\eta_i  \leq \kappa_i \sqrt{\mbf{p}_i}$  leads to $0 < \eta_i \leq \frac{\kappa_i\sqrt{\mbf{p}_i}}{\gamma_i L_i}$ as our choice of $\eta_i$ in \eqref{eq:stepsize_cond1} of Theorem~\ref{th:RCOG_convergence}.

Next, let us denote $\psi_i := 2\gamma_i\eta_i(1 - \gamma_i)\big[ \eta_i - \frac{\rho_i}{\gamma_i} - \frac{\eta_i(1 + \sigma_i)(1-\gamma_i)}{ \gamma_i\mbf{p}_i } \big]$.
To guarantee $\psi_i > 0$, we first need to impose that $\gamma_i\mbf{p}_i  > (1 + \sigma_i)(1-\gamma_i)$ and then choose $\eta_i > \frac{ \rho_i\mbf{p}_i}{ \gamma_i\mbf{p}_i  - (1 + \sigma_i)(1-\gamma_i)}$.
The first condition holds if we choose $\frac{1+ \sigma_i}{1+\sigma_i + \mbf{p}_i} < \gamma_i < 1$ as we stated.
Clearly, combining the second condition and $0 < \eta_i \leq \frac{\kappa_i\sqrt{\mbf{p}_i}}{ \gamma_i L_i }$, we have $L_i\rho_i <  \kappa_i\sqrt{\mbf{p}_i}\big[ 1 -  \frac{(1+\sigma_i)(1-\gamma_i)}{\gamma_i\mbf{p}_i } \big]$ as stated in \eqref{eq:stepsize_cond1} of Theorem~\ref{th:RCOG_convergence}.

Now, using $\psi_i := 2 \gamma_i\eta_i(1 - \gamma_i)\big[ \eta_i - \frac{\rho_i}{\gamma_i} - \frac{\eta_i(1 + \sigma_i)(1-\gamma_i)}{ \gamma_i\mbf{p}_i } \big] > 0$ and $\psi := (\psi_1,\cdots,\psi_n)$,  we obtain $\E_k\big[\Pc_{k+1}\big] \leq \Pc_k -  \psi  \circ \norms{Gx^k }^2$ from \eqref{eq:RCOG_descent}.
Taking the full expectation of this inequality, we get \eqref{eq:RCOG_descent1}.
Since $\Pc_k \geq 0$, letting  $\psi_{\min} := \min_{i \in [n]}\psi_i$, and then averaging \eqref{eq:RCOG_descent1} from $k := 0$ to $k := K$, we get
\myeqn{
\frac{\psi_{\min} }{K+1} \cdot \sum_{k=0}^{K} \E\big[\norms{Gx^k}^2 \big] \leq \frac{1}{K+1} \sum_{k=0}^K \psi \circ \E\big[ \norms{ Gx^k] }^2 \big]\leq  \frac{\Pc_0}{K+1}.
}
Finally, since $x^{-1} = x^0$ and $Gx^{\star} = 0$, by the Lipschitz continuity of $G$ and $2\gamma_i^2\eta_i^2L_i^2 \leq 2 \kappa_i\sqrt{ \mbf{p}_i } \leq 1$, we have $\Pc_0 = \norms{x^0 + \gamma\eta \circ Gx^0 - x^{\star}}^2 \leq 2\norms{x^0 - x^{\star}}^2 + 2\gamma^2\eta^2 \circ \norms{Gx^0 - Gx^{\star}}^2 \leq \sum_{i=1}^n(2 + 2\gamma_i^2\eta_i^2 L_i^2)\norms{x_i^0 - x_i^{\star}}^2 \leq 3\norms{x^0 - x^{\star}}^2$.
Substituting  this bound of $\Pc_0$ into the last inequality, we obtain \eqref{eq:RBC_PEG_ergodic_rate}.
\Eproof
\endproof
%%% End of Proof.

%%% Remark 4.1.
\vspace{1ex}
\begin{remark}[\textbf{Simple stepsizes and distribution}]\label{re:Alg1_rm1}
One simple choice of $\gamma_i$ and $\eta_i$ is 
\myeqn{
\gamma_i := \frac{2(1+ \sigma_i)}{2(1 + \sigma_i) + \mbf{p}_i} \quad \text{and} \quad \eta_i :=  \frac{\kappa_i \sqrt{\mbf{p}_i} }{ L_i }\Big( 1 + \frac{ \mbf{p}_i }{2(1 + \sigma_i)} \Big).
}
In this case, the condition on $L_i\rho_i$ becomes $2L_i\rho_i <  \kappa_i\sqrt{ \mbf{p}_i }$.
In particular, if we choose $\sigma_i := 1$, then $\gamma_i = \frac{4}{4+\mbf{p}_i}$ and $\eta_i = \frac{(4 + \mbf{p}_i) \sqrt{\mbf{p}_i}}{8 L_i}$, and the last condition is simplified as $4L_i\rho_i < \sqrt{\mbf{p}_i}$.
If $\rho_i = 0$, i.e. $\iprods{Gx, x - x^{\star}} \geq 0$ for all $x \in \dom{G}$, then the condition $4L_i\rho_i < \sqrt{\mbf{p}_i}$ automatically holds.

To see the dependence of the convergence bound \eqref{eq:RBC_PEG_ergodic_rate} on the number of block coordinates $n$, we again choose $\sigma_i := 1$ for all $i \in [n]$.
For simplicity, we also assume that $\rho = 0$ and $L_i = L$, and choose $\mbf{p}_i = \frac{1}{n}$  for all $i \in [n]$.
Then by Jensen's inequality, \eqref{eq:RBC_PEG_ergodic_rate} reduces to 
\myeq{eq:Alg1_rm1_bound}{
\min_{0 \leq k \leq K} \E\big[ \norms{Gx^k} \big]   \leq \Big( \frac{1}{K+1}\sum_{k=0}^{K} \E\big[\norms{Gx^k}^2 \big] \Big)^{1/2} \leq  \frac{4 \sqrt{3}  n L  \norms{x^0 - x^{\star}} }{ \sqrt{ K+1} }.
}
This bound shows that ${\displaystyle\min_{0 \leq k \leq K}} \E[ \norms{Gx^k} ] = \BigO{ \frac{n L\norms{x^0 - x^{\star}}}{ \sqrt{K} } }$, which means that the best-iterate convergence rate of $\E\big[ \norms{Gx^k} \big]$ depends linearly on the number of blocks $n$.
This linear dependence is consistent with previous known results in the literature, e.g., in \cite[Theorem 4.4]{kotsalis2022simple}, but for a different class of problems covered by Assumption~\ref{as:A1}.
Note that we can also choose the output $\hat{x}^k$ uniformly and randomly from $\sets{x^0, x^1, \cdots, x^k}$ generated by \ref{eq:naRCOG_scheme} to obtain $\E[ \norms{G\hat{x}^k} ] = \BigO{ \frac{n L\norms{x^0 - x^{\star}}}{ \sqrt{K} } }$.
From Lemma~\ref{le:RCOG_descent}, by applying the supermartingale theorem \cite[Theorem 1]{robbins1971convergence}, we can also prove that $\sum_{k=0}^{\infty} \norms{Gx^k}^2 < +\infty$ almost surely, leading to $\lim_{k\to\infty}\norms{Gx^k} = 0$ almost surely.
\end{remark}

%%%% Remark 4.2.
\begin{remark}[\textbf{Non-uniform distribution of $\mbf{p}$}]\label{re:Alg1_rm2}
If $L_i$ in \eqref{eq:CE_assumption_A1_a} are different for all $i \in [n]$,  then we can choose $\mbf{p}_i := \frac{L_i^{\nu} }{S_{\nu}} \in (0, 1)$ for $i \in [n]$, where $S_{\nu} := \sum_{i=1}^nL_i^{\nu}$ for some $\nu > 0$.
For $\sigma_i := 1$, if we choose $\gamma_i := \frac{4}{4 + \mbf{p}_i}$ and $\eta_i := \frac{(4 + \mbf{p}_i) \sqrt{\mbf{p}_i}}{8 L_i}$ as in Remark~\ref{re:Alg1_rm1}, then $\psi_i := \frac{\mbf{p}_i^2}{16L_i^2} = \frac{L_i^{2\nu-2}}{16S_{\nu}^2}$.
If $\nu = 1$, then $\psi_i = \frac{1}{16S_1^2}$, and \eqref{eq:Alg1_rm1_bound} becomes ${\displaystyle\min_{0 \leq k \leq K}}  \E[ \norms{Gx^k}] \leq \frac{4\sqrt{3} S_1 \cdot \norms{x^0 - x^{\star}} }{  \sqrt{ K+1} }$, which only depends on $S_1 := \sum_{i=1}^nL_i$.
If $\nu \neq 1$, then \eqref{eq:Alg1_rm1_bound} reduces to ${\displaystyle\min_{0 \leq k \leq K}}  \E[ \norms{Gx^k}]  \leq \frac{4\sqrt{3} S_{\nu} \cdot \norms{x^0 - x^{\star}} }{ L_{\min}^{\nu-1} \sqrt{ K+1} }$, where $L_{\min} := \min\sets{L_i : i \in [n]}$.
\end{remark}

%%%%%%%%%%%%%%%%%%%%%%%%%%%%%%%%%%%%%%%
%%%% 3. Randomized Block-Coordinate Halpern-Type Algorithms
%%%%%%%%%%%%%%%%%%%%%%%%%%%%%%%%%%%%%%%
\beforesec
\section{The Accelerated Randomized Block-Coordinate OG Method.}\label{sec:AcRCOG_method}
\aftersec
Alternative to the non-accelerated method in Section~\ref{sec:Na_RCOG_method}, in this section, we develop an accelerated variant of \ref{eq:naRCOG_scheme} for solving \eqref{eq:NE} under the following assumption.

%%% Assumption A.2.
\begin{assumption}\label{as:A2}
\revised{The operator $G$ in \eqref{eq:NE} is $\bar{\beta}$-co-coercive, i.e.  there exists $\bar{\beta} := (\bar{\beta}_1, \cdots, \bar{\beta}_n) \in [0, +\infty)^n$ such that for any $x, y \in \dom{G}$:}
\myeq{eq:CE_assumption_A2}{
\iprods{Gx  - Gy, x  - y} \geq  \bar{\beta} \circ \norms{Gx - Gy}^2.
\tag{CP}
}
\end{assumption}

Assumption~\ref{as:A2} states that $G$ is $1$-co-coercive w.r.t. a weighted norm $\norms{\cdot}_{\bar{\beta}}$.
Clearly, under \eqref{eq:CE_assumption_A2}, $\iprods{Gx - Gy, x - y} \geq \bar{\beta}_{\min}\norms{Gx - Gy}^2$, showing that $G$ is $\bar{\beta}_{\min}$-co-coercive, where $\bar{\beta}_{\min} := \min\sets{\bar{\beta}_i : i \in [n]}$.
Conversely, if $G$ is $\hat{\beta}$-co-coercive, then it satisfies \eqref{eq:CE_assumption_A2} with $\bar{\beta}_i = \hat{\beta}$ for all $i \in [n]$.
Note that \eqref{eq:CE_assumption_A2} is weaker than the block-coordinate-wise co-coerciveness of $G$, i.e. $\iprods{[Gx]_i - [Gy]_i, x_i - y_i} \geq \bar{\beta}_i\norms{[Gx]_i - [Gy]_i}^2$ for all $x, y \in \dom{G}$ and $i \in [n]$.
The condition~\eqref{eq:CE_assumption_A2} also implies that $G$ is monotone and $\bar{L}$-Lipschitz continuous with $\bar{L} := \big( \min\set{\bar{\beta}_i : i \in [n]} \big)^{-1}$.
It is stronger than Assumption~\ref{as:A1}.

As shown in Subsection~\ref{subsec:examples}, \eqref{eq:NE} is equivalent to the fixed-point problem \eqref{eq:fixed_point} of $F := \Id - \eta G$ for any $\eta > 0$.
Moreover, the co-coerciveness of $G$ is equivalent to the non-expansiveness of $F$ \cite[Proposition 4.11]{Bauschke2011}.
Therefore, the problem of finding a fixed-point of a non-expansive operator $F$, can be cast into \eqref{eq:NE} under Assumption~\ref{as:A2}, which covers a broad range of applications.

In smooth and convex optimization, if $Gx = \nabla{f}(x)$, the gradient of the objective function $f$, then the convexity of $f$ and the $L$-Lipschitz continuity of $\nabla{f}$ is equivalent to the $\frac{1}{L}$-co-coerciveness of $G$, see \cite[Theorem 2.1.5]{Nesterov2004}.
Thus Assumption~\ref{as:A2} is not restrictive. 

\beforesubsec
\subsection{The derivation of algorithm.}\label{subsec:ARCOG_scheme} 
\aftersubsec
Our \textit{\textbf{accelerated randomized block-coordinate optimistic gradient}} \eqref{eq:ARCOG_scheme} scheme for solving \eqref{eq:NE} is presented as follows:
\textit{Starting from $x^0 \in \dom{G}$, we set $x^{-1} := x^0$, and at each iteration $k \geq 0$, we randomly sample a block-coordinate $i_k \in [n]$ using \eqref{eq:prob_cond1} and update
\myeq{eq:ARCOG_scheme}{
x^{k+1} := x^k + \theta_k (x^k - x^{k-1})  - \frac{ \eta_k  }{\mbf{p}_{i_k}}  \big(  G_{[i_k]}x^k  -  \gamma_k G_{[i_k]}x^{k-1} \big),
\tag{ARCOG}
}
where $\theta_k \in (0, 1)$,  $\eta_k > 0$, and $\gamma_k \in (0, 1)$ are given parameters at each iteration $k$ and for all block-coordinates $($determined later$)$, and $G_{[i]}x^k = [\mbf{0}, \cdots, \mbf{0}, [Gx^k]_i, \mbf{0}, \cdots, \mbf{0}]$ such that $[Gx^k]_i$ is the $i$-the block of $Gx^k$ ($i\in [n]$).
}

Compared to \ref{eq:naRCOG_scheme}, our \ref{eq:ARCOG_scheme} still requires to use simultaneously two block-coordinates $[Gx^k]_{i_k}$ and $[Gx^{k-1}]_{i_k}$ evaluated at the  two consecutive iterates $x^{k}$ and $x^{k-1}$, respectively but it now additionally has a momentum or an inertial term $\theta_k (x^k - x^{k-1})$.
These two features  make it different from existing randomized [block]-coordinate methods in  \cite{combettes2015stochastic,peng2016arock} for monotone inclusions as well as randomized coordinate-type methods for optimization such as \cite{fercoq2015accelerated,latafat2019block,Nesterov2012,nesterov2017efficiency,richtarik2016parallel,wright2015coordinate}.

%%%%%%%%%%%%%%%%%%%%%%%%%%%%
%%% 3.2. Practical ARBC variant.
\beforesubsec
\subsection{Practical variant of \ref{eq:ARCOG_scheme}.}\label{subsec:pracitice_RBC_scheme}
\aftersubsec
%%%%%%%%%%%%%%%%%%%%%%%%%%%%
Due to the momentum term $\theta_k (x^k - x^{k-1})$, \ref{eq:ARCOG_scheme} still requires full vector update at each iteration $k$.
To resolve this issue, we derive an alternative form of \ref{eq:ARCOG_scheme} as in the following proposition so that it is easier to implement in practice. 
The proof of  Proposition~\ref{pro:practical_variant} is deferred to  Appendix~\ref{apdx:pro:practical_variant}.

%%% Proposition 1.
\begin{proposition}\label{pro:practical_variant}
Given $x^0 \in \dom{G}$, we set $z^0 = z^{-1} := x^0$ and $w^0 = w^{-1} := 0$, and at each iteration $k \geq 0$, we update
\myeq{eq:ARCOG_practice_variant}{
\arraycolsep=0.3em
\left\{\begin{array}{lcl}
d^k_{i_k} &:= &  [G(z^k + c_kw^k)]_{i_k}  -  \gamma_k[G(z^{k-1} + c_{k-1}w^{k-1})]_{i_k}, \ \ (\text{only block $i_k$}) \vspace{1ex}\\
w_i^{k+1} & := & \left\{\begin{array}{ll}
w_i^k - \frac{  \eta_k }{\mbf{p}_{i} \tau_{k+1} } d^k_i,  & \text{if $i = i_k$}, \vspace{1ex}\\
w_i^k, &\text{otherwise}, 
\end{array}\right. \vspace{1ex}\\
z_i^{k+1} & := & \left\{\begin{array}{ll}
z_i^k + \frac{  \eta_k c_k}{\mbf{p}_{i}\tau_{k+1}} d^k_{i}, &\text{if $i = i_k$}, \vspace{1ex}\\
z_i^k, &\text{otherwise},
\end{array}\right.
\end{array}\right.
}
where $c_0 = c_{-1} := 0$ and $\tau_0 := 1$ and $\tau_k$ and $c_k$ are respectively updated as
\myeq{eq:ARCOG_practice_variant_param}{
\arraycolsep=0.3em
\begin{array}{lcl}
\tau_{k+1} := \tau_k\theta_k \quad \text{and} \quad c_k := c_{k-1} + \tau_k.
\end{array}
}
Then, $x^k := z^k + c_kw^k$ is identical to the one generated by \ref{eq:ARCOG_scheme} for all $k\geq 0$.
\end{proposition}

The scheme \eqref{eq:ARCOG_practice_variant} is different from accelerated randomized block-coordinate methods for convex optimization such as \cite{Alacaoglu2017,fercoq2015accelerated,Nesterov2012,Tran-Dinh2019spd}.
Nevertheless, it still uses block-coordinate evaluations of $G$ at $z^k + c_kw^k$ and $z^{k-1} + c_{k-1}w^{k-1}$ as in  methods for convex optimization.
\revised{Such evaluations do not necessarily require one to form the full vectors $z^k + c_kw^k$ and $z^{k-1} + c_{k-1}w^{k-1}$, see, e.g., \cite{fercoq2015accelerated} for concrete examples.}
Since \eqref{eq:ARCOG_practice_variant} does not require full dimensional updates of $w^k$ and $z^k$, we only update $z^k$ and $w^k$ at the active block $i_k$, and then compute $x^k$ at the last iteration. 
Note that one can also extend \ref{eq:ARCOG_scheme} and variant \eqref{eq:ARCOG_practice_variant} to update multiple blocks by randomly choosing a subset of block coordinates $\Sc_k \subset [n]$ such that $\Prob{i \in \Sc_k} = \mbf{p}_i > 0$ for $i\in [n]$.

%%%%% 4.3. Lyapunov function and descent lemma.
\beforesubsec
\subsection{Lyapunov function and descent lemma.}\label{subsec:ARCOG_one_iter_analysis}
\aftersubsec
Given $\bar{\beta}$ in Assumption~\ref{as:A2}, let $\beta_i$ be given such that $0 < \beta_i \leq \bar{\beta}_i$ for all $i \in [n]$.
We introduce the following Lyapunov function:
\myeq{eq:ARCOG_scheme_Lyfunc}{
\arraycolsep=0.2em
\revised{\begin{array}{lcll}
\hat{\Pc}_k & := & 2 r t_k\eta_{k-1}\big[ \iprods{Gx^{k-1}, x^{k-1} - x^{\star}} -  \beta \circ \norms{Gx^{k-1}}^2 \big]  + \norms{r(x^{k-1} - x^{\star}) + t_k(x^k - x^{k-1})}^2 \vspace{1ex} \\
&& + {~} \mu r \norms{x^{k-1} - x^{\star}}^2, 
\end{array}}
}
where $x^{\star} \in \zer{G}$, $\beta := (\beta_1,\cdots, \beta_n)$, and $t_k > 0$, $r > 0$, and $\mu \geq 0$ are given parameters (determined later).
Clearly, under Assumption~\ref{as:A2} and the choice of $\beta_i$, we have $\hat{\Pc}_k \geq 0$ regardless the values of $x^{k-1}$ and $x^k$.
We first state the following descent lemma, whose proof is deferred to Appendix~\ref{apdx:subsec:le:ARCOG_key_est1}.

%%%% Lemma 2.1.
\begin{lemma}\label{le:ARCOG_key_est1}
Suppose that $\zer{G} \neq\emptyset$ and Assumption~\ref{as:A2} holds for \eqref{eq:NE}.
Let $\sets{x^k}$ be generated by \ref{eq:ARCOG_scheme} and $\hat{\Pc}_k$ be defined by \eqref{eq:ARCOG_scheme_Lyfunc}.
Let the parameters in \ref{eq:ARCOG_scheme} and \eqref{eq:ARCOG_scheme_Lyfunc} satisfy 
\myeq{eq:ARCOG_scheme_para_cond2}{
\hspace{-0ex}
\theta_k  :=  \frac{t_k - r - \mu}{t_{k+1}} \quad \text{and} \quad   \gamma_k :=  \frac{t_{k+1} \theta_k}{t_{k+1}\theta_k + r} \in (0, 1).
\hspace{-3ex}
}
Then, for all $k\geq 0$, the following estimate holds:
\myeq{eq:ARCOG_scheme_main_est1}{
\hspace{-4ex}
\arraycolsep=0.2em
\revised{\begin{array}{lcl}
\hat{\Pc}_k - \E_k\big[ \hat{\Pc}_{k+1} \big] & \geq & \mu (2t_k - r -  \mu) \norms{x^k - x^{k-1}}^2 + 2  t_{k+1}^2\theta_k\eta_k  (\bar{\beta} - \beta) \circ \norms{Gx^k  - Gx^{k-1} }^2 \vspace{1ex}\\
&& + {~} 2r \big( t_k\eta_{k-1} - t_{k+1}\gamma_k\eta_k \big)  \big[ \iprods{Gx^{k-1}, x^{k-1} - x^{\star}} -  \beta \circ \norms{ Gx^{k-1} }^2 \big] \vspace{1ex} \\
&& + {~}   \sum_{i=1}^n\frac{ t_{k+1}\eta_k [  2\beta_i \mbf{p}_i ( t_{k+1}\theta_k + r) - t_{k+1}\eta_k ] }{\mbf{p}_i} \norms{ [Gx^k]_i - \gamma_k [Gx^{k-1}]_i }^2.
\end{array}}
\hspace{-4ex}
}
\end{lemma}

\beforesubsec
\subsection{Main result 2: Convergence of \ref{eq:ARCOG_scheme}.}\label{subsec:main_result2}
\aftersubsec
For simplicity of our presentation, given $\bar{\beta}$ in  Assumption~\ref{as:A2} and $\mbf{p}$ in \eqref{eq:prob_cond1}, for $0 < \beta_i \leq \bar{\beta}_i$ ($i \in [n]$) and $0 < \omega < \min\sets{2\beta_i\mbf{p}_i : i \in [n]}$, we denote
\myeq{eq:ARCOG_scheme_Lambda_constants}{
\arraycolsep=0.2em
\left\{\begin{array}{lcllcllcl}
\Lambda_0 & :=  &   \max  \big\{\tfrac{1}{  2\beta_i} \; : \;  i\in[n] \big\},  & \quad\qquad & 
\Lambda_1^2  & := & \max \big\{  \frac{1}{ \bar{\beta}_i^2} \; : \; i\in[n] \big\}, \vspace{0.5ex}\\
\Lambda_2 & := &  \max \big\{ \frac{1 - \mbf{p}_i}{ 2\beta_i\mbf{p}_i - \omega} \; : \; i \in [n] \big\}, & \quad &
\Lambda_3 & := &  \max \big\{ \frac{1}{ 2\beta_i\mbf{p}_i - \omega} \; : \; i \in [n] \big\}.
\end{array}\right.
}
Next, for $r > 2$, we introduce the following constants:
\myeq{eq:ARCOG_scheme_constants}{
\arraycolsep=0.2em
\left\{\begin{array}{lcl}
C_0 & := &   \frac{\omega^2(r-1)^2(r-2) \Lambda_1^2 }{(r+1)} +  [ \omega r (r-1)\Lambda_0 +  r + 1 ] \cdot \max\big\{ \frac{4}{r}, \omega\Lambda_2\big\}, \vspace{1ex}\\
C_1 &:= & \frac{2(r+1)^2}{\omega^2(r-1)^2}\big[ \frac{C_0}{r-2} + \omega r(r-1)\Lambda_0 +  r + 1 \big], \vspace{1ex}\\
C_2 & := &    \Lambda_3 \left[  r(r-1)\Lambda_0 + \frac{r + 1 }{\omega} \right]  + r C_1.
\end{array}\right.
}
Now, we are ready to prove the convergence of \ref{eq:ARCOG_scheme} in the following theorem.

%%%% Theorem 4.1.
\begin{theorem}\label{th:ARCOG_convergence}
Suppose that $\zer{G} \neq\emptyset$  and Assumption~\ref{as:A2} holds for \eqref{eq:NE}.
Let $\beta_i$ be given such that $0 < \beta_i \leq \bar{\beta}_i$ for all $i \in [n]$.
For given $r > 2$ and $0 < \omega < \min \set{ 2\beta_i\mbf{p}_i  :  i \in [n]}$, let $\sets{x^k}$ be generated by \ref{eq:ARCOG_scheme} using the following parameters:
\myeq{eq:ARCOG_scheme_update_pars}{
\arraycolsep=0.2em
\theta_k := \frac{k}{k + r + 2}, \quad  \gamma_k := \frac{k}{k+r},  \quad\text{and} \quad \eta_k := \frac{\omega(k + r)}{k+r+2}. 
}
Then, for $\Lambda_0$, $C_0$, $C_1$, and $C_2$ given in \eqref{eq:ARCOG_scheme_Lambda_constants} and \eqref{eq:ARCOG_scheme_constants}, the following statements hold:
\vspace{0.5ex}
\begin{itemize}
\item[$\mathrm{(i)}$] \textbf{$[$Summable bounds$]$}
The following summable results are achieved:
\myeq{eq:ARCOG_scheme_result1}{
\hspace{-2ex}
\arraycolsep=0.2em
\begin{array}{lcl}
\sum_{k=1}^{+\infty}(2 k + r + 1) \cdot \Exp{\norms{x^k - x^{k-1}}^2}    & \leq & \big[ \omega r(r-1) \Lambda_0 +  r + 1 \big]  \cdot \norms{x^0 - x^{\star}}^2, \vspace{1.5ex} \\ 
\sum_{k=0}^{+\infty} (k + r)^2 \cdot \sum_{i=1}^n \frac{(2 \beta_i\mbf{p}_i - \omega)}{\mbf{p}_i} \cdot \Exp{\norms{ [Gx^k]_i -   \gamma_k[Gx^{k-1}]_i}^2} & \leq & \left[  r(r-1)\Lambda_0 + \frac{r + 1 }{\omega} \right]  \cdot \norms{x^0 - x^{\star}}^2, \vspace{1.5ex} \\
\sum_{k=0}^{+\infty}(k + r + 1) \cdot \Exp{\norms{Gx^k}^2 } & \leq & C_1 \cdot \norms{x^0 - x^{\star}}^2, \vspace{1.5ex} \\ 
\sum_{k=1}^{+\infty} k^2  \cdot \Exp{ \bar{\beta} \circ \norms{Gx^k  -  Gx^{k-1} }^2 }  & \leq & \frac{\omega C_2}{2} \cdot \norms{x^0 - x^{\star}}^2.
\end{array}
\hspace{-3ex}
}

\item[$\mathrm{(ii)}$] \textbf{$[$Big-O and small-o convergence rates$]$}~We have the following convergence rates:
\myeq{eq:ARCOG_scheme_main_est2b}{
\arraycolsep=0.3em
\left\{\begin{array}{lcl}
\Exp{\norms{x^{k+1} - x^k}^2 } &  \leq &   \frac{\omega^2 C_2}{2(k + r + 2)^2} \cdot \norms{x^0 - x^{\star}}^2,   \vspace{1.5ex}\\
\Exp{\norms{Gx^k}^2 } & \leq &     \frac{2(r+1)^2(C_0 + C_2)}{ (r-1)^2(k+r-1)(k+r+2)}  \cdot \norms{x^0 - x^{\star}}^2,   \vspace{1.5ex}\\
{\displaystyle\lim_{k\to\infty}} \big\{ (k + r + 2)^2 \cdot \Exp{\norms{x^{k+1} - x^k}^2 }  \big\} &   = & 0, \vspace{1.25ex}\\
{\displaystyle\lim_{k\to\infty}} \big\{ (k + r - 1)(k+r+2)   \cdot  \Exp{\norms{Gx^k}^2 } \big\}  &  = &   0.
\end{array}\right.
}

\item[ $\mathrm{(iii)}$] \textbf{$[$Almost sure convergence$]$}~
The following statements hold in the \textbf{almost sure} sense:
\myeqn{
\arraycolsep=0.3em
\begin{array}{ll}
& \sum_{k=0}^{+\infty}(k + r + 2)\norms{x^{k+1} - x^k}^2 < +\infty \quad \textrm{and} \quad \sum_{k=0}^{+\infty}(k + r + 1)\norms{Gx^k}^2 < +\infty, \vspace{1.5ex}\\
& \lim_{k\to\infty} k^2\norms{x^{k+1} - x^k}^2 = 0 \quad \textrm{and} \quad  \lim_{k\to\infty} k^2\norms{Gx^k}^2 = 0.
\end{array}
}
Moreover, $\sets{x^k}$ converges almost surely to $x^{\star} \in \zer{G}$.
\end{itemize}
\end{theorem}
%%%% End of Lemma 4.1.

Theorem~\ref{th:ARCOG_convergence} establishes the convergence rates of \ref{eq:ARCOG_scheme} on two main criteria $\Exp{\norms{Gx^k}^2}$ and $\Exp{\norms{x^{k+1} - x^{k}}^2}$, among other summable results as stated in \eqref{eq:ARCOG_scheme_result1} and the almost sure convergence stated in (iii).
While the first two lines of \eqref{eq:ARCOG_scheme_main_est2b} present the $\BigO{1/k^2}$-convergence rates, its last two lines show that $\Exp{\norms{x^{k+1} - x^k}^2 }  = \SmallO{\frac{1}{k^2}}$ and $\Exp{\norms{Gx^k}^2 } = \SmallO{\frac{1}{k^2}}$, respectively, which state $\SmallO{1/k^2}$-convergence rates.
Note that in accelerated randomized block-coordinate methods for convex optimization, it is not trivial to prove convergence of the iterative sequence.
The almost sure convergence of iterates was proven for non-accelerated randomized block-coordinate methods such as in \cite{combettes2015stochastic}.
However, to the best of our knowledge, Statement (iii) is the first result in accelerated randomized block-coordinate methods, even in the context of convex optimization.

%%%% Remark 5.1.
\begin{remark}[\textbf{Dependence on $n$}]\label{re:ARCOG_rates1}
To see the dependence of the first and second bounds in \eqref{eq:ARCOG_scheme_main_est2b} on the number of blocks $n$, we assume that $\beta_i := \beta = \bar{\beta} > 0$, and choose $\mbf{p}_i := \frac{1}{n}$ for all $i \in [n]$,  and $r := 3$ for simplicity of our analysis.
Moreover, we also choose $\omega := \frac{\beta}{n}$, which satisfies the conditions of Theorem~\ref{th:ARCOG_convergence}.
Then, we can easily show that $\Lambda_0 = \frac{1}{2\beta}$, $\Lambda_1^2 = \frac{1}{\beta^2}$, $\Lambda_2 = \frac{n-1}{\beta}$, and  $\Lambda_3 = \frac{n}{\beta}$.
In this case, we also have 
$C_0 = \frac{16}{3} + \frac{4}{n} + \frac{1}{n^2} = \BigOs{1}$, $C_1 = \frac{8n^2}{\beta^2}\big( \frac{28}{3} + \frac{7}{n} + \frac{1}{n^2} \big) = \BigOs{\frac{n^2}{\beta^2}}$, and $C_2 = \frac{n(3+4n)}{\beta^2} + \frac{24n^2}{\beta^2}\big(\frac{28}{3} + \frac{7}{n} + \frac{1}{n^2} \big) = \BigOs{\frac{n^2}{\beta^2}}$.

Using these expressions, we can show that
\myeqn{
\Exp{\norms{x^{k+1} - x^k}^2 } = \BigO{ \frac{\norms{x^0 - x^{\star}}^2}{k^2} } \quad \text{and} \quad \Exp{\norms{Gx^k}^2 } = \BigO{ \frac{n^2 \cdot \norms{x^0 - x^{\star}}^2}{\beta^2 \cdot k^2} }.
}
Clearly, the first estimate is independent of both $n$ and $\beta$, while the second one depends on both $n$ and $\beta$.
The dependence of convergence rates on $n^2$ has been observed in Nesterov's accelerated RBC first-order methods for convex optimization, see, e.g., \cite{Alacaoglu2017,fercoq2015accelerated,Nesterov2012}.
\end{remark}

Note that we have not tried to tighten the constants defined in \eqref{eq:ARCOG_scheme_Lambda_constants} and \eqref{eq:ARCOG_scheme_constants} of Theorem~\ref{th:ARCOG_convergence} and Remark~\ref{re:ARCOG_rates1}, leading to some undesirably large values of $C_0$, $C_1$, and $C_2$.
They can be further improved by refining the proof of Theorem~\ref{th:ARCOG_convergence}.
In addition, one can choose different values for $\mbf{p}_i$, $\omega$, and $r$ to reduce these constants.
For example, we can choose $\mbf{p}_i := \frac{1}{\bar{\beta}_i^{\nu}S_{\nu}}$ for all $i \in [n]$ and some $\nu > 0$, where $S_{\nu} := \sum_{i=1}^n \bar{\beta}_i^{-\nu}$, $\omega := \frac{1}{S_{\nu}}$, and $r := 2.1$.
This choice of $\mbf{p}_i$ takes into account the non-homogeneity of $\bar{\beta}$ in \eqref{eq:CE_assumption_A2}, and can improve the dependence on $n$ of our convergence rates.

%%%% Proof of Theorem 4.1.
\vspace{0.75ex}
\proof{\textbf{The proof of Theorem~\ref{th:ARCOG_convergence}}.}
We divide our proof into several steps as follows.

%%%% (a)-Choice of parameters.
\vspace{0.5ex}
(a)~\textit{The choice of parameters.}
First, for simplicity of our proof, we choose $\mu := 1$ in $\hat{\Pc}_k$ given by \eqref{eq:ARCOG_scheme_Lyfunc}.
Then, from \eqref{eq:ARCOG_scheme_para_cond2}, we have $\theta_k := \frac{t_k - r - 1}{t_{k+1}}$ and $\gamma_k := \frac{t_k - r - 1}{t_k - 1}$.
Next, to guarantee the non-negativity of the right-hand side of \eqref{eq:ARCOG_scheme_main_est1}, we impose the following two conditions:
\begin{equation}\label{eq:th41_proof0}
t_k\eta_{k-1} \geq t_{k+1}\gamma_k\eta_k \quad \text{and} \quad 2\beta_i \mbf{p}_i ( t_{k+1}\theta_k + r) - t_{k+1}\eta_k \geq 0.
\end{equation}
Now,  for a given $r \geq 1$, we update $t_k := k + r + 1$ and $\eta_k := \frac{\omega(t_k - 1)}{ t_{k+1} } = \frac{\omega(k+r)}{ k + r + 2}$ as in \eqref{eq:ARCOG_scheme_update_pars}, where $0 < \omega \leq \min \sets{ 2\beta_i \mbf{p}_i : i \in [n] }$.
Then, both conditions in \eqref{eq:th41_proof0} are satisfied. 
Finally, since  $t_k := k + r + 1$, we can easily show that $\theta_k = \frac{k}{k + r + 2}$ and $\gamma_k = \frac{k}{k + r}$ as stated in the update rule \eqref{eq:ARCOG_scheme_update_pars}.

%%%% (b) Upper bounding P0 and |Gx^0|.
(b)~\textit{Upper bounding $\hat{\Pc}_0$ and $\norms{Gx^0}$.}
From \eqref{eq:CE_assumption_A2} and $Gx^{\star} = 0$, we have $\bar{\beta}_i\norms{[Gx^0]_i} \leq \norms{x_i^0 - x_i^{\star}}$ for all $i \in [n]$.
By this relation and the Cauchy-Schwarz and Young inequalities, we can show that
\myeqn{
\arraycolsep=0.2em
\begin{array}{lcl}
\iprods{Gx^0, x^0 - x^{\star}} \leq \sum_{i=1}^n\norms{[Gx^0]_i}\norms{x_i^0 - x^{\star}_i} \leq \sum_{i=1}^n \beta_i\norms{[Gx^k]_i}^2 + \tfrac{1}{4} \cdot {\displaystyle\max_{i\in[n]}} \big\{ \frac{1}{\beta_i} \big\} \cdot \norms{x^0 - x^{\star}}^2.
\end{array}
}
Moreover, from \eqref{eq:ARCOG_scheme_update_pars} we have $t_0\eta_{-1} = \omega(r-1)$.
Using this fact, $\mu = 1$, $x^0 = x^{-1}$, $\Lambda_0 := \max \sets{ \frac{1}{2\beta_i} : i\in[n] }$, and $\Lambda_1^2 := \max\sets{  \frac{1}{ \bar{\beta}_i^2 } : i\in[n] }$ from \eqref{eq:ARCOG_scheme_Lambda_constants}, we can derive from \eqref{eq:ARCOG_scheme_Lyfunc} that
\myeq{eq:P0_ub}{
\arraycolsep=0.2em
\begin{array}{lcl}
\hat{\Pc}_0 & \leq &  \big[ \omega r(r-1)\Lambda_0 +  r + 1 \big] \cdot \norms{x^0 - x^{\star}}^2 \quad \text{and} \quad \norms{Gx^0}^2  \leq    \Lambda_1^2 \cdot \norms{x^0 - x^{\star}}^2.
\end{array}
}

%%%% (c) The first three summable results.
(c)~\textit{The first two summable bounds of \eqref{eq:ARCOG_scheme_result1}.}
By \eqref{eq:ARCOG_scheme_update_pars}, \eqref{eq:th41_proof0} holds.
Hence, \eqref{eq:ARCOG_scheme_main_est1} reduces to
\myeqn{
\arraycolsep=0.2em
\begin{array}{lcl}
\hat{\Pc}_k - \E_k\big[ \hat{\Pc}_{k+1} \big] & \geq & \omega(k+r)^2 \cdot \sum_{i=1}^n\frac{  ( 2\beta_i \mbf{p}_i  - \omega ) }{\mbf{p}_i} \norms{ [Gx^k]_i - \gamma_k [Gx^{k-1}]_i }^2 +  (2k + r + 1) \cdot \norms{x^k - x^{k-1}}^2.
\end{array}
}
Taking the full expectation of this inequality, we obtain
\myeq{eq:ARCOG_scheme_proof2001}{
\arraycolsep=0.2em
\begin{array}{lcl}
\E \big[ \hat{\Pc}_k \big] - \E\big[ \hat{\Pc}_{k+1} \big] & \geq &  \omega(k+r)^2 \cdot \sum_{i=1}^n\frac{  ( 2\beta_i \mbf{p}_i  - \omega ) }{\mbf{p}_i} \cdot \E\big[\norms{ [Gx^k]_i - \gamma_k [Gx^{k-1}]_i }^2 \big] \vspace{1ex}\\
&& + {~} (2k + r + 1) \cdot \E\big[ \norms{x^k - x^{k-1}}^2 \big].
\end{array}
}
Since $\E\big[ \hat{\Pc}_k \big] \geq 0$, summing up \eqref{eq:ARCOG_scheme_proof2001} from $k := 0$ to $k := K$, and noting that $x^{-1} = x^0$, we obtain
\myeq{eq:ARCOG_scheme_proof2002}{
\arraycolsep=0.2em
\begin{array}{lcl}
\sum_{k=0}^{K} (k+r)^2 \cdot \sum_{i=1}^n \frac{(2 \beta_i\mbf{p}_i - \omega)}{\mbf{p}_i} \cdot \E\big[\norms{   [Gx^k]_i -  \gamma_k [Gx^{k-1}]_i}^2 \big] &\leq & \frac{1}{ \omega} \cdot \E\big[\hat{\Pc}_0 \big], \vspace{1.5ex}\\
\sum_{k=1}^{K}(2k + r + 1) \cdot \Exp{\norms{x^k - x^{k-1}}^2} &\leq &  \E\big[\hat{\Pc}_0 \big].
\end{array}
}
These inequalities imply the first two lines of \eqref{eq:ARCOG_scheme_result1} after taking the limit as $K\to+\infty$ and using $\E\big[\hat{\Pc}_0 \big] = \hat{\Pc}_0 \leq  [ \omega r(r-1) \Lambda_0 +  r + 1 ]  \cdot \norms{x^0 - x^{\star}}^2$ from  \eqref{eq:P0_ub}.

%%%% (d) The third summable result.
(d)~\textit{The third summable bound of \eqref{eq:ARCOG_scheme_result1}.}
Let us define the following full vector:
\myeq{eq:RCHP_scheme00_barx_plus}{
\arraycolsep=0.3em
\begin{array}{lcl}
\bar{x}^{k+1} := x^k + \theta_k(x^k - x^{k-1}) - \eta_k ( Gx^k  -  \gamma_k Gx^{k-1} ) = z^k - \eta_k d^k,
\end{array}
}
where $z^k := x^k + \theta_k(x^k - x^{k-1})$ and $d^k :=   Gx^k  -  \gamma_kGx^{k-1}$.
Then, from \ref{eq:ARCOG_scheme} we have $x^{k+1} = z^k - \frac{\eta_k }{\mbf{p}_{i_k}} d_{[i_k]}^k$.
Therefore, for any vector $u^k$ independent of $i_k$, we can derive that
\myeq{eq:RCHP_scheme00b}{
\hspace{-4ex}
\arraycolsep=0.2em
\begin{array}{lcl}
\Exps{k}{\norms{x^{k+1} - u^k}^2} &= & \norms{z^k - u^k}^2 - 2 \eta_k \E_k\big[ \iprods{\mbf{p}_{i_k}^{-1}d_{[i_k]}^k, z^k - u^k} \big] + \eta_k^2 \cdot \E_k\big[ \norms{\mbf{p}_{i_k}^{-1} d_{[i_k]}^k}^2 \big] \vspace{0.5ex}\\
& \overset{\tiny\eqref{eq:le4.1_proof2}}{=} & \norms{z^k - u^k}^2 - 2 \eta_k \iprods{d^k, z^k - u^k} + \eta_k^2 \cdot \sum_{i=1}^n\frac{1}{\mbf{p}_i}\norms{d^k_i}^2 \vspace{1ex}\\
&= & \norms{z^k - \eta_k d^k - u^k}^2 + \eta_k^2 \cdot \sum_{i=1}^n \big( \frac{1}{\mbf{p}_i} - 1 \big)\norms{d^k_i}^2 \vspace{1ex}\\
&= & \norms{\bar{x}^{k+1} - u^k}^2 +  \eta_k^2 \cdot \sum_{i=1}^n\big(\frac{1}{\mbf{p}_i} - 1 \big)\norms{d^k_i}^2.
\end{array}
\hspace{-4ex}
}
Now, if we denote $\tau_k := \frac{\gamma_k \eta_k }{\eta_{k-1}}$, then, from \eqref{eq:RCHP_scheme00_barx_plus} we can easily show that 
\myeqn{
\arraycolsep=0.3em
\begin{array}{lcl}
\bar{x}^{k+1} - x^k + \eta_k Gx^k & = &  \tau_k (x^k - x^{k-1} +  \eta_{k-1} Gx^{k-1}) +  (1 - \tau_k ) \cdot \frac{\tau_k - \theta_k}{1-\tau_k} (x^{k-1} - x^k ).
\end{array}
}
Since $r \geq 1$, utilizing \eqref{eq:ARCOG_scheme_update_pars},  we can show that $\tau_k = \frac{(k+r+1)k}{(k+r+2)(k+r-1)}  \in (0, 1)$, $1 - \tau_k =  \frac{rk + (r-1)(r+2)}{(k+r-1)(k+r+2)} \leq \frac{r(k+r+1)}{(k+r-1)(k+r+2)}$, and $\frac{\tau_k - \theta_k}{1-\tau_k} = \frac{2k}{rk + (r+2)(r-1)} \leq \frac{2}{r}$.
%\myred{We have $\tau_k = \frac{(k+r+1)(k+r)k}{(k+r+2)(k+r-1)(k+r)} = \frac{(k+r+1)k}{(k+r+2)(k+r-1)}$ and $1 - \tau_k = \frac{(k+r+2)(k+r-1) - k(k+r+1)}{(k+r-1)(k+r+2)} = \frac{k^2 + (2r + 1)k + (r-1)(r+2) - k^2 - (r+1)k}{(k+r+2)(k+r-1)} = \frac{rk + (r-1)(r+2)}{(k+r-1)(k+r+2)}$. We also have $\frac{\tau_k - \theta_k}{1-\tau_k} = \frac{(k+r+2)(k+r-1)}{rk + (r+2)(r-1)} \cdot \big[ \frac{k(k+r+1)}{(k+r+2)(k+r-1)} - \frac{k(k+r-1)}{(k+r+2)(k+r-1)}\big] = \frac{k[k+r+1 - k - r + 1]}{rk + (r-1)(r+2)} = \frac{2k}{rk + (r-1)(r+2)}$.}
Hence, by the convexity of $\norms{\cdot}^2$, we have
\myeqn{
\hspace{-1ex}
\arraycolsep=0.2em
\begin{array}{lcl}
\norms{\bar{x}^{k+1} - x^k + \eta_k Gx^k}^2 \leq  \tau_k \norms{x^k - x^{k-1} + \eta_{k-1} Gx^{k-1}}^2 + \frac{4(1-\tau_k)}{r^2}\norms{x^k - x^{k-1}}^2.
\end{array}
\hspace{-1ex}
}
Using a shorthand $v^k := x^{k+1} - x^k$ and  substituting $u^k := x^k -  \eta_k  Gx^k$ and $d^k :=   Gx^k  -  \gamma_kGx^{k-1}$ into \eqref{eq:RCHP_scheme00b}, and then combining the result with the last inequality, we can show that
\myeqn{
\arraycolsep=0.1em
\begin{array}{lcl}
\Exps{k}{ \norms{v^k +  \eta_k Gx^k}^2 } & \leq & \frac{k (k+r+1)}{ (k+r-1)(k+r+2) } \cdot \norms{v^{k-1} + \eta_{k-1} Gx^{k-1}}^2  +  \frac{4(k + r + 1)}{r (k+r-1) (k+r+2) } \cdot \norms{x^k - x^{k-1}}^2 \vspace{1ex}\\
&& + {~} \frac{\omega^2(k+r)^2}{(k+r+2)^2} \cdot \sum_{i=1}^n  \big( \frac{1-\mbf{p}_i}{\mbf{p}_i} \big) \cdot \norms{  [Gx^k]_i   -  \gamma_k [Gx^{k-1}]_i}^2. 
\end{array}
}
Taking the full expectation  both sides of this inequality and using $k+r-1 \leq k + r + 2$, we get
\myeqn{
\arraycolsep=0.3em
\begin{array}{lcl}
\Exp{ \norms{v^k +  \eta_k  Gx^k}^2 } & \leq &   \frac{k (k+r+1)}{ (k+r-1)(k+r+2) } \cdot \Exp{\norms{v^{k-1} +  \eta_{k-1}Gx^{k-1}}^2} +  \frac{4(k + r + 1)}{r (k+r-1) (k+r+2) } \cdot  \Exp{\norms{x^k - x^{k-1}}^2} \vspace{1ex}\\
&& + {~}  \frac{\omega^2(k+r)^2}{ (k+r-1)(k+r+2) } \cdot \sum_{i=1}^n \big( \frac{1-\mbf{p}_i}{\mbf{p}_i} \big) \cdot \Exp{\norms{ [Gx^k]_i  -  \gamma_k [Gx^{k-1}]_i }^2}. 
\end{array}
}
Multiplying this inequality by $(k+r-1)(k+r+2)$ and rearranging the result, we obtain
\myeq{eq:ARCOG_scheme_2023_proof1}{
\hspace{-3ex}
\arraycolsep=0.1em
\begin{array}{ll}
(k+r-1)(k+r+2) & \Exp{\norms{v^k + \eta_k Gx^k}^2}  \leq   (k+r-2)(k + r + 1) \Exp{\norms{v^{k-1} +  \eta_{k-1} Gx^{k-1}}^2}  \vspace{1ex}\\
& - {~} (r-2)(k + r + 1) \Exp{\norms{ v^{k-1} + \eta_{k-1} Gx^{k-1}}^2} \vspace{1ex}\\
& + {~}  \omega^2 (k + r)^2 \cdot \sum_{i=1}^n \big( \frac{1-\mbf{p}_i}{\mbf{p}_i} \big) \cdot  \Exp{\norms{  [Gx^k]_i  -  \gamma_k [Gx^{k-1}]_i }^2} \vspace{1ex}\\
&  + {~} \frac{ 4(k+r+1)}{r} \cdot \Exp{\norms{x^k - x^{k-1}}^2}.
\end{array}
\hspace{-6ex}
}
Summing up \eqref{eq:ARCOG_scheme_2023_proof1} from $k := 0$ to $k := K$, and then using \eqref{eq:ARCOG_scheme_proof2002}, $x^{-1} = x^0$, $\eta_{-1} = \frac{\omega(r-1)}{r+1}$, \eqref{eq:P0_ub}, and $\Lambda_1^2$ and $\Lambda_2$ from \eqref{eq:ARCOG_scheme_Lambda_constants}, we get 
\myeq{eq:RCHP_scheme00_proof3}{
\hspace{-4ex}
\arraycolsep=0.2em
\begin{array}{ll}
(K + r - 1)(K+r+2) & \Exp{\norms{v^K +  \eta_K Gx^K}^2} + (r - 2) \sum_{k=0}^{K} (k + r + 1) \cdot \Exp{ \norms{v^{k-1} + \eta_{k-1} Gx^{k-1}}^2} \vspace{1ex}\\
& \leq  \frac{\omega^2(r-1)^2(r-2)}{(r+1)} \cdot \Exp{\norms{Gx^0}^2} + \frac{4}{r} \cdot \sum_{k=0}^K  (2k+r+1) \cdot \Exp{\norms{x^k - x^{k-1}}^2} \vspace{1ex}\\
& +  \sum_{k=0}^K \omega (k + r)^2 \sum_{i=1}^n  \frac{2\beta_i\mbf{p}_i - \omega}{\mbf{p}_i } \cdot \frac{\omega( 1-\mbf{p}_i) }{ 2\beta_i\mbf{p}_i - \omega }  \cdot  \Exp{\norms{  [Gx^k]_i  -  \gamma_k [Gx^{k-1}]_i }^2} \vspace{1ex}\\ 
& \overset{\tiny\eqref{eq:ARCOG_scheme_proof2002}}{\leq}   \frac{\omega^2(r-1)^2(r-2)}{(r+1)} \cdot \Exp{\norms{Gx^0}^2}  + \max\Big\{ \frac{4}{r} ,\   {\displaystyle\max_{i \in [n]}} \big\{ \frac{\omega(1 - \mbf{p}_i)}{ 2\beta_i\mbf{p}_i - \omega} \big\}  \Big\} \cdot \E\big[\hat{\Pc}_0 \big] \vspace{0ex}\\  
&\overset{\tiny\eqref{eq:P0_ub}}{\leq} \Big[ \frac{\omega^2(r-1)^2(r-2) \Lambda_1^2 }{(r+1)}  +  [ \omega r (r-1)\Lambda_0 +  r + 1 ] \cdot \max\sets{ \frac{4}{r}\; , \; \omega \Lambda_2 }  \Big] \cdot \norms{x^0 - x^{\star}}^2.
\end{array}
\hspace{-4ex}
}
If we define $C_0 :=  \frac{\omega^2(r-1)^2(r-2) \Lambda_1^2 }{(r+1)} +  [ \omega r (r-1)\Lambda_0 +  r + 1 ] \cdot \max\big\{ \frac{4}{r}, \omega\Lambda_2\big\} $ as in \eqref{eq:ARCOG_scheme_constants}, then \eqref{eq:RCHP_scheme00_proof3} implies 
\myeq{eq:RCHP_scheme2023_proof2}{
\arraycolsep=0.3em
\begin{array}{ll}
& \Exp{\norms{x^{k+1} - x^k +  \eta_k Gx^k}^2} \leq \frac{C_0\norms{x^0 - x^{\star} }^2}{(k+r-1)(k+r+2)}, \vspace{1ex}\\
& (r - 2) \sum_{k=0}^{+\infty} (k + r + 1) \cdot \Exp{ \norms{ x^k - x^{k-1} + \eta_{k-1} Gx^{k-1}}^2} \leq C_0 \cdot \norms{x^0 - x^{\star}}^2.
\end{array}
}
Next, noting that $\eta_{k-1} = \frac{\omega(k+r-1)}{k+r+1} \geq \frac{\omega(r-1)}{r+1}$, by Young's inequality, we can show that
\myeq{eq:RCHP_scheme00_proof4}{
\arraycolsep=0.3em
\begin{array}{lcl}
\frac{\omega^2(r-1)^2}{(r+1)^2}\norms{Gx^{k-1}}^2 \leq \eta_{k-1}^2\norms{Gx^{k-1}}^2 & \leq & 2\norms{x^k - x^{k-1} + \eta_{k-1} Gx^{k-1}}^2 + 2\norms{x^k - x^{k-1}}^2.
\end{array}
}
Utilizing this inequality, \eqref{eq:RCHP_scheme2023_proof2} and the third line of \eqref{eq:ARCOG_scheme_proof2002}, we obtain
\myeq{eq:RCHP_scheme00_proof5}{
\arraycolsep=0.3em
\begin{array}{lcl}
\sum_{k=1}^{K}(k + r + 1) \cdot \Exp{\norms{Gx^k}^2 } \leq  \frac{2(r+1)^2}{\omega^2(r-1)^2}\big[ \frac{C_0}{r-2} + \omega r(r-1)\Lambda_0 +  r + 1 \big] \cdot \norms{x^0 - x^{\star}}^2.
\end{array}
}
Using $C_1$ defined by \eqref{eq:ARCOG_scheme_constants}, this inequality proves the third line of \eqref{eq:ARCOG_scheme_result1}.

%%% (e) The fourth summable result.
(e)~\textit{The fourth summable bounds of \eqref{eq:ARCOG_scheme_result1}.}
From $x^{k+1} - x^k = \theta_k(x^k - x^{k-1}) - \eta_k \mbf{p}^{-1}_{i_k}d^k_{[i_k]}$ of \ref{eq:ARCOG_scheme}, we can easily show that
\myeqn{
\arraycolsep=0.3em
\begin{array}{lcl}
\norms{x^{k+1} - x^k}^2 &= & \theta_k^2\norms{x^k - x^{k-1}}^2 - 2\eta_k \theta_k\iprods{\mbf{p}^{-1}_{i_k}d^k_{[i_k]}, x^k - x^{k-1}} + \eta_k^2\norms{\mbf{p}^{-1}_{i_k}d^k_{[i_k]}}^2.
\end{array}
}
Taking the conditional expectation $\Exps{k}{\cdot}$, and using \eqref{eq:le4.1_proof2} below, we have
\myeqn{
\arraycolsep=0.2em
\begin{array}{lcl}
\Exps{k}{\norms{x^{k+1} - x^k}^2} &= & \theta_k^2\norms{x^k - x^{k-1}}^2 - 2\eta_k \theta_k\iprods{d^k, x^k - x^{k-1}} + \eta_k^2\sum_{i=1}^n\frac{1}{\mbf{p}_i}\norms{d^k_i}^2 \vspace{1ex}\\
&= & \theta_k^2\norms{x^k - x^{k-1}}^2 + \eta_k^2 \cdot \sum_{i=1}^n \frac{1}{\mbf{p}_i} \norms{ [Gx^k]_i - \gamma_k[Gx^{k-1}]_i}^2 \vspace{1ex}\\
&& - {~} 2\gamma_k \eta_k\theta_k\iprods{Gx^k - Gx^{k-1}, x^k - x^{k-1}} - 2(1 - \gamma_k)\eta_k \theta_k \iprods{Gx^k, x^k - x^{k-1}}.
\end{array}
}
Utilizing \eqref{eq:CE_assumption_A2} and Young's inequality into the last expression, we can show that
\myeqn{
\arraycolsep=0.3em
\begin{array}{lcl}
\Exps{k}{\norms{x^{k+1} - x^k}^2} &\leq & \big[ \theta_k^2 +  \frac{(1 - \gamma_k)\eta_k\theta_k}{\omega} \big] \norms{x^k - x^{k-1}}^2 +  \omega(1 - \gamma_k)\eta_k\theta_k\norms{Gx^k}^2 \vspace{1ex}\\
&& - {~} 2 \gamma_k \eta_k\theta_k \cdot \sum_{i=1}^n \bar{\beta}_i\norms{[Gx^k]_i - [Gx^{k-1}]_i}^2 \vspace{1ex}\\
&&  + {~} \eta_k^2 \cdot \sum_{i=1}^n \frac{1}{\mbf{p}_i} \norms{ [Gx^k]_i - \gamma_k[Gx^{k-1}]_i}^2.
\end{array}
}
By \eqref{eq:ARCOG_scheme_update_pars}, we have $\gamma_k \eta_k\theta_k = \frac{\omega k^2}{(k+r+2)^2}$, $\theta_k^2 + \frac{ \eta_k\theta_k(1 - \gamma_k) }{\omega} = \frac{k(k+r)}{(k+r+2)^2}$, and $(1 - \gamma_k)\eta_k\theta_k = \frac{\omega rk }{(k+r+2)^2}$.
Using these expressions into the last inequality, and taking the full expectation, we can show that
\myeq{eq:RCHP_scheme00_proof7_01}{
\hspace{-2ex}
\arraycolsep=0.2em
\begin{array}{lcl}
(k+r+2)^2\Exp{\norms{x^{k+1} - x^k}^2} & \leq & (k+r+1)^2 \Exp{\norms{x^k - x^{k-1}}^2} - (r+2)(k + r) \Exp{\norms{x^k - x^{k-1}}^2} \vspace{1ex}\\
&& + {~} \sum_{i=1}^n\frac{\omega^2(k+r)^2}{\mbf{p}_i} \cdot \Exp{\norms{[Gx^k]_i - \gamma_k [Gx^{k-1}]_i}^2} \vspace{1ex}\\
&& -  {~} 2 \omega k^2\sum_{i=1}^n \bar{\beta}_i \Exp{\norms{[Gx^k]_i - [Gx^{k-1}]_i}^2}  + \omega^2r k \cdot \Exp{\norms{Gx^k}^2}.
\end{array}
\hspace{-4ex}
}
Summing up \eqref{eq:RCHP_scheme00_proof7_01}  from $k := 0$ to $k := K$ and noting that  $x^{0} = x^{-1}$, we obtain
\myeq{eq:RCHP_scheme00_proof7}{  
\hspace{-2ex}
\arraycolsep=0.1em
\begin{array}{ll}
(K + r + 2)^2 & \Exp{\norms{x^{K+1} - x^K}^2} +   (r+2) \sum_{k=1}^K (k + r) \Exp{\norms{x^k - x^{k-1}}^2} \vspace{1ex}\\
& + {~} 2\omega \sum_{k=1}^K k^2 \sum_{i=1}^n \bar{\beta}_i \Exp{ \norms{ [Gx^k]_i  - [Gx^{k-1}]_i }^2 } \vspace{1ex}\\
& \leq ~ \sum_{k=0}^K \omega^2 (k + r)^2 \sum_{i=1}^n\frac{1}{\mbf{p}_i}  \E\big[ \norms{ [Gx^k]_i  - \gamma_k [Gx^{k-1}]_i }^2 \big] +  \omega^2 r \sum_{k=1}^K k \cdot \Exp{\norms{Gx^k}^2} \vspace{1ex}\\
& \leq ~ {\displaystyle\max_{i \in [n]}} \big\{\frac{\omega^2 }{2\beta_i\mbf{p}_i - \omega} \big\} \cdot \sum_{k=0}^K  (k + r)^2 \sum_{i=1}^n\frac{2\beta_i\mbf{p}_i - \omega}{\mbf{p}_i} \cdot  \E\big[ \norms{ [Gx^k]_i  - \gamma_k [Gx^{k-1}]_i }^2 \big]  \vspace{1ex} \\
& + {~} \omega^2 r \sum_{k=1}^K (k + r + 1) \cdot \Exp{\norms{Gx^k}^2} \vspace{1ex}\\
&\overset{\tiny\eqref{eq:ARCOG_scheme_constants}}{\leq } \Big[ \omega^2\Lambda_3\left[  r(r-1)\Lambda_0 + \frac{r + 1 }{\omega} \right]  + \omega^2 r C_1 \Big] \cdot \norms{x^0 - x^{\star}}^2 \equiv \omega^2 C_2 \cdot \norms{x^0 - x^{\star}}^2,
\end{array}
\hspace{-2ex}
}
where $\Lambda_3 :=  {\displaystyle\max_{i \in [n]}} \big\{ \frac{1}{ 2\beta_i\mbf{p}_i - \omega} \big\}$ as defined in \eqref{eq:ARCOG_scheme_Lambda_constants} and $C_2$ is given in \eqref{eq:ARCOG_scheme_constants}.
In this case, \eqref{eq:RCHP_scheme00_proof7} leads to the fourth line of \eqref{eq:ARCOG_scheme_result1}.

%%% (f) The proof of the O-rates and o-rates
(f)~\textit{The $\BigO{1/k^2}$ and $\SmallO{1/k^2}$ rates in \eqref{eq:ARCOG_scheme_main_est2b}.}
The first bound of \eqref{eq:ARCOG_scheme_main_est2b} comes from \eqref{eq:RCHP_scheme00_proof7}.
Now, utilizing \eqref{eq:RCHP_scheme00_proof7} and the second line of \eqref{eq:ARCOG_scheme_proof2002}, we can conclude that  $\lim_{k\to\infty}(k + r + 2)^2 \cdot \Exp{\norms{x^{k+1} - x^k}^2}$ exists.
Combining the existence of this limit and the summable result $\sum_{k=1}^{\infty}(2k + r + 1) \cdot \Exp{\norms{x^k - x^{k-1}}^2} \leq  [ \omega r(r-1) \Lambda_0 +  r + 1 ] \cdot \norms{x^0 - x^{\star}}^2$, we obtain $\lim_{k\to\infty}(k + r + 2)^2 \cdot \Exp{\norms{x^{k+1} - x^k}}^2 = 0$, which proves a $\SmallO{1/k^2}$-convergence rate of $\Exp{\norms{x^{k+1} - x^k}^2}$ in \eqref{eq:ARCOG_scheme_main_est2b}.

Utilizing the first line of \eqref{eq:RCHP_scheme2023_proof2} and the first bound of \eqref{eq:ARCOG_scheme_main_est2b}, we can show that
\myeqn{
\arraycolsep=-0.0em
\begin{array}{lcl}
\Exp{\norms{Gx^{k}}^2} & \overset{\tiny\eqref{eq:RCHP_scheme00_proof4}}{\leq} &  \frac{2(r+1)^2}{\omega^2(r-1)^2} \cdot \Exp{\norms{x^{k+1} - x^k + \eta_{k} Gx^{k}}^2} + \frac{2(r+1)^2}{\omega^2(r-1)^2} \cdot \Exp{\norms{x^{k+1} - x^k}^2} \vspace{1ex}\\
& \leq & \frac{2(r+1)^2}{ (r-1)^2}  \Big[ \frac{C_0}{(k+r-1)(k+r+2)} + \frac{ C_2}{(k+r+2)^2} \Big] \cdot \norms{x^0 - x^{\star}}^2,
\end{array}
}
which proves the second bound of \eqref{eq:ARCOG_scheme_main_est2b}.
From the first and second lines  of \eqref{eq:RCHP_scheme2023_proof2}, we get
\myeqn{
\lim_{k\to\infty}(k + r - 1)(k+r+2) \cdot \Exp{\norms{x^{k+1} - x^k +  \eta_k Gx^k}^2} = 0.
}
Exploiting this limit, the third line $\lim_{k\to\infty}(k + r + 2)^2 \cdot \Exp{ \norms{x^{k+1} - x^k }^2 } = 0$ of \eqref{eq:ARCOG_scheme_main_est2b}, and \eqref{eq:RCHP_scheme00_proof4}, we can deduce that $\lim_{k\to\infty}(k + r - 1)(k+r+2) \cdot \Exp{ \norms{Gx^k}^2 } = 0$, which proves the last line of \eqref{eq:ARCOG_scheme_main_est2b}, showing a $\SmallO{1/k^2}$-convergence rate of $\Exp{ \norms{Gx^k}^2 }$.

%%% (g) The almost sure convergence of the iterates.
(g)~\textit{Almost sure convergence.}
As proven in Lemma~\ref{le:ARCOG_xk_convergence}, we have that $\sets{x^k}$ is bounded almost surely.
Let $x^{\star}$ be an almost sure cluster point of $\sets{x^k}$, i.e. there exists a subsequence $\sets{x^k}_{k\in\Kc}$ such that $\lim_{k\in\Kc\to\infty}x^{k} = x^{\star}$ almost surely.
Since we have proven  in Lemma~\ref{le:ARCOG_xk_convergence} that $\lim_{k\to\infty}k^2 \norms{Gx^{k}}^2 = 0$ almost surely, we have $\lim_{k\in\Kc\to\infty} \norms{Gx^{k}} = 0$ almost surely.
%By the continuity of $G$, we get $Gx^{\star} = 0$ almost surely, showing that $x^{\star} \in \zer{G}$.
By Lemma~\ref{le:ARCOG_xk_convergence}, $\lim_{k\to\infty}  \norms{x^k - x^{\star}}^2 $ exists almost surely for any $x^{\star} \in \zer{G}$.
Applying Lemma~\ref{le:opial_lemma}, we conclude that $\sets{x^k}$ converges almost surely to $x^{\star} \in \zer{G}$ (a $\zer{G}$-valued random variable).
The remaining conclusions in Statement (iii) are already proven in \eqref{eq:ARCOG_almost_sure} of Lemma~\ref{le:ARCOG_xk_convergence}.
%%%
\Eproof
\endproof
%%%% End of Theorem 3.1.

\vspace{1ex}
\begin{remark}[\textbf{Application to fixed-point problems}]\label{re:fixed_point}
We can apply \ref{eq:ARCOG_scheme} to approximate a fixed-point $x^{\star}$ of a nonexpansive operator $F$, i.e. $x^{\star} = Fx^{\star}$.
Indeed, $F$ is nonexpansive iff $G := \Id - F$ is $1/2$-co-coercive \cite[Proposition 4.11]{Bauschke2011}, and $x^{\star} = Fx^{\star}$ iff $Gx^{\star} = 0$, where $\Id$ is the identity operator.
We can apply \ref{eq:ARCOG_scheme} to solve $Gx^{\star} = 0$, which is equivalent to finding a fixed point $x^{\star}$ of $F$.
Alternatively, we can also apply \ref{eq:naRCOG_scheme} to solve this fixed-point problem, but under the condition $\iprods{x  - Fx, x - x^{\star}}  \geq -\rho \circ \norms{x - Fx}^2$ for all $x\in\R^p$ and the Lipschitz continuity of $\Id - F$. 
\end{remark}

\vspace{1ex}
\begin{remark}[\textbf{Application to operator splitting methods}]\label{re:monotone_inclusion}
Subsection~\ref{subsec:NEvsNI} shows several ways to reformulate \eqref{eq:FedNI} equivalently to \eqref{eq:NE} by imposing appropriate assumptions on $G$ and $T$.
However, specifying \ref{eq:naRCOG_scheme} and \ref{eq:ARCOG_scheme} to these reformulations are straightforward, so that we omit their derivation and convergence analysis here, and only highlight the main idea.

If the sum $G+T$ of \eqref{eq:FedNI} satisfies Assumption~\ref{as:A1} and $G$ is Lipschitz continuous, then we can reformulate \eqref{eq:FedNI} into $S_{\lambda}x = 0$, where $S_{\lambda}$ is Tseng's FBFS operator defined by \eqref{eq:S_operator2} in Subsection~\ref{subsec:FOG} below.
In this case, we can apply \ref{eq:naRCOG_scheme} to solve $S_{\lambda}x^{\star} = 0$ and obtain a new variant of \ref{eq:naRCOG_scheme} for solving a class of nonmonotone \eqref{eq:FedNI} (see also Section~\ref{sec:num_examples}).

If $G$ is $\beta$-co-coercive, then finding a solution $x^{\star}$ of \eqref{eq:FedNI} is equivalent to solving $F_{\lambda}x^{\star} = 0$, where $F_{\lambda}$ is the FBS operator defined by \eqref{eq:FBS_operator}.
In this case, if we apply \ref{eq:ARCOG_scheme} to $F_{\lambda}x^{\star} = 0$, then we obtain a new accelerated RBC-FBS algorithm for solving \eqref{eq:FedNI}.
Alternatively, we can reformulate \eqref{eq:FedNI} into $H_{\lambda}y^{\star} = 0$, where $H_{\lambda}$ is the BFS operator defined in \cite{attouch2018backward}.
If we apply \ref{eq:ARCOG_scheme} to $H_{\lambda}y^{\star} = 0$, then we obtain a new accelerated RBC-BFS variant to solve \eqref{eq:FedNI}.

If both $G$ and $T$ in \eqref{eq:FedNI} are maximally monotone, then by applying \ref{eq:ARCOG_scheme} to solve $V_{\beta}u^{\star} = 0$ defined by the DRS operator $V_{\beta}$ in \eqref{eq:DRS_operator}, we obtain a new randomized block-coordinate DRS method.
Note that it is also possible to apply \ref{eq:ARCOG_scheme} to three-operator splitting methods in \cite{Davis2015}.
\end{remark}

%%%%%%%%%%%%%%%%%%%%%%%%%%%%%%%%%%%%%%%%
%%% 5. Application to Finite-Sum Monotone Inclusions
%%%%%%%%%%%%%%%%%%%%%%%%%%%%%%%%%%%%%%%%
\beforesec
\section{Application to Finite-Sum Root-Finding Problems.}\label{sec:FedNI_algorithms}
\aftersec
The finite-sum inclusion \eqref{eq:FedNI} provides a unified template to cope with many applications in different fields, including networks, distributed systems, statistical learning, and machine learning as mentioned earlier.
In this work, we are interested in instances of \eqref{eq:FedNI} arising from \textit{federated learning} (FL), where the number of users $n$ is sufficiently large.

Briefly, FL is a recent paradigm in machine learning to train a learning algorithm that coordinates different users (also called workers or agents) to perform a given task without sharing  users' raw data.
A key point that distinguishes FL from traditional distributed or parallel algorithms is the privacy of data, where the raw data of users must not be shared.
Moreover, due to a large number of users and a heterogeneity of the system, only one or a small number of users can participate in each communication round (i.e. each iteration $k$). 
Existing FL algorithms are often designed to solve FL optimization problems such as \cite{konevcny2016federated,li2020federated,mcmahan_ramage_2017,mcmahan2021advances}. 
There is only a limited number of works studying minimax problems and other extensions, e.g., \cite{du2021fairness,sharma2022federated,tarzanagh2022fednest} but these works often require strong assumptions and generally do not cover our model  \eqref{eq:FedNI}.

Our goal  is to apply both \ref{eq:naRCOG_scheme} and \ref{eq:ARCOG_scheme} to develop new FL-type algorithms for solving \eqref{eq:FedNI}.
Our algorithms can solve a more general class of problems covered by \eqref{eq:FedNI} beyond optimization under milder assumptions than many previous works.
Moreover, they only require exchanging certain intermediate vectors between users, which fulfills the privacy requirement.
They also do not require the full participation of all users at each communication round $k$.

%%%%
\beforesubsec
\subsection{Reformulation and solution characterization.}\label{subsec:FedNI_reform}
\aftersubsec
%Let us recall the finite-sum root-finding problem \eqref{eq:FedNI}.
To develop new variants of  \ref{eq:naRCOG_scheme} and \ref{eq:ARCOG_scheme} for solving \eqref{eq:FedNI}, we need to reformulate \eqref{eq:FedNI} into an appropriate form.
By duplicating $x$ as $\mbf{x} := [x_1, x_2, \cdots, x_n]$, where $x_i \in\R^p$,  we can reformulate \eqref{eq:FedNI} equivalently to
\myeq{eq:FedNI_reform}{
0 \in  \mbf{G}\mbf{x}^{\star} + \mbf{T}_1\mbf{x}^{\star} + \partial{\delta}_{\Lc}(\mbf{x^{\star}}),
}
where $\mbf{G}\mbf{x} := [G_1 x_1, \cdots, G_nx_n]$, $\mbf{T}_1\mbf{x} := [n Tx_1, 0, \cdots, 0]$, and $\partial{\delta}_{\Lc}$ is the subdifferential of the indicator $\delta_{\Lc}$ of the linear subspace 
$\Lc := \sets{\mbf{x} = [x_1, \cdots, x_n] \in \R^{np} : x_i = x_1, ~\forall i = 2, \cdots, n}$.
It is obvious to show that $x^{\star}$ is a solution of \eqref{eq:FedNI} if and only if $\mbf{x}^{\star} = [x^{\star}, \cdots, x^{\star}]$ solves \eqref{eq:FedNI_reform}, see, e.g., \cite{pham2021federated}.
Moreover, $\mbf{G}$ and $\mbf{T}_1$ in \eqref{eq:FedNI_reform} are separable. 

%%% Lemma 5.
\begin{lemma}\label{le:appox_sol}
We characterize a solution of \eqref{eq:FedNI} using \eqref{eq:FedNI_reform} in two cases as follows.
\begin{itemize}
\item[$\mathrm{(a)}$] For all $i \in [n]$, given $(x^{\star}_i, v_i^{\star}) \in \gra{G_i}$, denote $u_i^{\star} := x_i^{\star} - \lambda v^{\star}_i$, and for any $\lambda > 0$ define
\myeq{eq:u_star}{
\hat{u}^{\star} := J_{\lambda T}\Big( \tfrac{1}{n}\sum_{i=1}^nu_i^{\star} \Big).
}
Then, if $\sum_{i=1}^n\norms{x_i^{\star} - \hat{u}^{\star}}^2 = 0$,  then $x_i^{\star} = \hat{u}^{\star}$ for all $i\in [n]$,  and $\hat{u}^{\star}$ is a solution of \eqref{eq:FedNI}.

\item[$\mathrm{(b)}$]
For all $i \in [n]$, given $u^{\star}_i \in \dom{G_i}$, we denote $\hat{u}^{\star}$ as in \eqref{eq:u_star}.
Then, if we have $\sum_{i=1}^n\norms{\hat{u}^{\star} - J_{\lambda G_i}(2\hat{u}^{\star} - u_i^{\star})}^2 = 0$,  then $\hat{u}^{\star}$ is a solution of \eqref{eq:FedNI}.
\end{itemize}
\end{lemma}

%%% Proof of Lemma 5.
\proof{\textbf{Proof.}}
(a)~
Since $(x^{\star}_i, v_i^{\star}) \in \gra{G_i}$ for $i \in [n]$, we have $u^{\star}_i := x_i^{\star} - \lambda v_i^{\star} \in x_i^{\star} - \lambda G_ix_i^{\star}$ for all $i\in [n]$.
However, since $\sum_{i=1}^n\norms{x_i^{\star} - \hat{u}^{\star}}^2 = 0$, we obtain $x_i^{\star} = \hat{u}^{\star}$ for all $i \in [n]$.
Using this relation, we can show that $\frac{1}{n}\sum_{i=1}^nu_i^{\star} \in \hat{u}^{\star} - \frac{\lambda}{n}\sum_{i=1}^n G_i\hat{u}^{\star}$.
Alternatively,  $\hat{u}^{\star} := J_{\lambda T}(\frac{1}{n}\sum_{i=1}^nu^{\star}_i)$ is equivalent to $\frac{1}{n}\sum_{i=1}^nu_i^{\star} \in \hat{u}^{\star} + \lambda T\hat{u}^{\star}$.
Combining the last two expressions, we arrive at $0 \in \frac{1}{n}\sum_{i=1}^nG_i\hat{u}^{\star} + T\hat{u}^{\star}$, showing that $\hat{u}^{\star}$ solves  \eqref{eq:FedNI}.

(b)~For $\hat{u}^{\star} := J_{\lambda T}(\frac{1}{n}\sum_{i=1}^nu_i^{\star})$, if $\sum_{i=1}^n\norms{ \hat{u}^{\star} - J_{\lambda G_i}(2\hat{u}^{\star} - u_i^{\star})}^2 = 0$, then $\hat{u}^{\star} = J_{\lambda G_i}(2\hat{u}^{\star} - u_i^{\star})$, which is equivalent to $\hat{u}^{\star} - u_i^{\star} \in \lambda G_i\hat{u}^{\star}$ for $i\in [n]$.
Hence, $\hat{u}^{\star} + \lambda T\hat{u}^{\star} - \frac{1}{n}\sum_{i=1}^nu_i^{\star} \in \frac{\lambda}{n} \sum_{i=1}^nG_i\hat{u}^{\star} + \lambda T\hat{u}^{\star}$.
However, $\hat{u}^{\star} := J_{\lambda T}(\frac{1}{n}\sum_{i=1}^nu_i^{\star})$ is equivalent to $0 \in \hat{u}^{\star} + \lambda T\hat{u}^{\star} - \frac{1}{n}\sum_{i=1}^nu_i^{\star}$.
Summing up the last two expressions, we obtain $0 \in  \frac{1}{n} \sum_{i=1}^nG_i\hat{u}^{\star} +  T\hat{u}^{\star}$.
Therefore, $\hat{u}^{\star}$ is a solution to \eqref{eq:FedNI}.
\Eproof
\endproof
%%% End of proof.

For a given tolerance $\epsilon > 0$, from Lemma~\ref{le:appox_sol} we can say that if $\sum_{i=1}^n\norms{x_i^{*} - \hat{u}^{*}}^2 \leq \epsilon^2$, then $\hat{u}^{*}$ can be viewed as an $\epsilon$-[approximate] solution of \eqref{eq:FedNI}.
Alternatively, if $\sum_{i=1}^n\norms{\hat{u}^{*} - J_{\lambda G_i}(2\hat{u}^{*} - u_i^{*})}^2 \leq \epsilon^2$, then $\hat{u}^{*}$ is also an $\epsilon$-solution to \eqref{eq:FedNI}.
These criteria will be used in the sequel. 

In order to apply \ref{eq:naRCOG_scheme} and \ref{eq:ARCOG_scheme} to solve \eqref{eq:FedNI_reform}, we need to reformulate it equivalently to \eqref{eq:FedNI}.
Let us  exploit the structures of $\mbf{G}$ and $ \partial{\delta}_{\Lc}$ and consider two cases: applying the FBFS operator to  \eqref{eq:FedNI_reform} in Subsection~\ref{subsec:FOG} and applying the DRS operator to  \eqref{eq:FedNI_reform} in Subsection~\ref{subsec:FedDR}, respectively.
Then, we customize \ref{eq:naRCOG_scheme} for the first setting and \ref{eq:ARCOG_scheme} for the second one. 

\beforesubsec
\subsection{Federated optimistic gradient algorithm.}\label{subsec:FOG}
\aftersubsec
For given $\mbf{T}_1$ and $\partial{\delta}_{\Lc}$ from \eqref{eq:FedNI_reform}, let us denote $\mbf{T} := \mbf{T}_1 + \partial{\delta}_{\Lc}$ and define 
\myeq{eq:S_operator2}{
\revised{\mbf{S}_{\lambda}\mbf{x} := \mbf{x} - J_{\lambda\mbf{T}}(\mbf{x} - \lambda \mbf{G}\mbf{x}) - \lambda\big(\mbf{G}\mbf{x} - \mbf{G}\circ \mbf{J}_{\lambda \mbf{T}}( \mbf{x} - \lambda \mbf{G}\mbf{x}) \big), }
}
where $\lambda > 0$ is determined later and $\circ$ is the composition operator.
Note that we use the FBFS operator $\mbf{S}_{\lambda}$ instead of a forward-backward splitting or DRS operator to avoid computing the resolvent of $G_i$ and to handle a nonmonotone case of  $G_i$ for $i \in [n]$ (see Assumption~\ref{as:A3} below).
By Lemma~\ref{le:property_of_S_oper}, $\mbf{x}^{\star}$ solves \eqref{eq:FedNI_reform} iff $\mbf{S}_{\lambda}\mbf{x}^{\star} = 0$.

First, we apply \ref{eq:naRCOG_scheme} to solve $\mbf{S}_{\lambda}\mbf{x}^{\star} = 0$, leading to the following scheme:
\myeq{eq:FedOG_scheme1}{
\mbf{x}^{k+1} := \mbf{x}^k - \tfrac{\eta}{\mbf{p}_{i_k}}\big( (\mbf{S}_{\lambda})_{[i_k]}\mbf{x}^k - \gamma (\mbf{S}_{\lambda})_{[i_k]}\mbf{x}^{k-1}\big),
}
where $(\mbf{S}_{\lambda})_{[i_k]}\mbf{x} = [0, 0, \cdots, 0, [\mbf{S}_{\lambda}\mbf{x} ]_{i_k}, 0, \cdots, 0]$ and $[\mbf{S}_{\lambda}\mbf{x} ]_{i} = x_i - [J_{\lambda\mbf{T}}(\mbf{x} - \lambda\mbf{G}\mbf{x})]_i - \lambda( G_ix_i - G_i([J_{\lambda\mbf{T}}(\mbf{x} - \lambda \mbf{A}\mbf{x})]_i))$.
Here, we also choose the same parameters $\eta$ and $\gamma$ for all block-coordinates.

Next, by Lemma~\ref{le:resolvent_oper}, if we define $u_i := x_i - \lambda G_ix_i$ for $i \in [n]$ and $\hat{u} := J_{\lambda T}(\frac{1}{n}\sum_{i=1}^nu_i)$, then  $[\mbf{S}_{\lambda}\mbf{x} ]_{i} = x_i - \hat{u} - \lambda( G_ix_i - G_i\hat{u})$.
Moreover, since at the $k$-th iteration, only $u^k_{i_k}$ is updated to $u_{i_k}^{k+1}$, we have $\bar{u}^{k+1} := \frac{1}{n}\sum_{i=1}^nu_i^{k+1} = \frac{1}{n}\sum_{i=1}^nu_i^k + \frac{1}{n}(u_{i_k}^{k+1} - u_{i_k}^k) = \bar{u}^k + \frac{1}{n}\Delta{u}^k_{i_k}$, where $\Delta{u}^k_{i_k} := u_{i_k}^{k+1} - u_{i_k}^k$.
Utilizing these expressions, we can describe our first \textit{federated optimistic gradient method}
(FedOG) for solving \eqref{eq:FedNI} as in Algorithm~\ref{alg:A1}.

\begin{algorithm}[hpt!]\caption{(Federated Optimistic Gradient Algorithm (\textbf{FedOG}))}\label{alg:A1}
\normalsize
\begin{algorithmic}[1]
\State\label{step:A1_i0}{\bfseries Initialization:} Input an initial point $x^0 \in \R^p$.
\State \hspace{2.5ex}Initialize each user $i$ with $x_i^0 = x_i^{-1}  := x^0$ and compute $u_i^0 := x^0 - \lambda G_ix^0$.
\State \hspace{2.5ex}Initialize sever with $\bar{u}^0 := \frac{1}{n}\sum_{i=1}^nu_i^0$ and $\hat{u}^0 := J_{\lambda T}(\bar{u}^0 )$.
Set  $\hat{u}^{-1} := \hat{u}^0$.
\State\hspace{0ex}\label{step:A1_o1}{\bfseries For $k := 0,\cdots, k_{\max}$ do}
\vspace{0.25ex}   
\State\hspace{2.5ex}\label{step:A1_o2}Sample an active user $i_k \in [n]$ following the probability distribution \eqref{eq:prob_cond1}.
\State\hspace{2.5ex}\label{step:A1_o3}[\textit{Communication}] Server sends $\hat{u}^k$ to user $i_k$.
\State\hspace{2.5ex}\label{step:A1_o4}[\textit{Local update}] User $i_k$ updates its iterates $x_{i_k}^k$ and $u_{i_k}^k$ as
\vspace{-0.5ex}
\myeqn{
\arraycolsep=0.2em
\left\{\begin{array}{lcl}
g_{i_k}^{k-1} & := &  x^{k-1}_{i_k} - \hat{u}^{k-1} - \lambda(G_{i_k}x^{k-1}_{i_k} - G_{i_k}\hat{u}^{k-1}), \vspace{1ex}\\
g_{i_k}^k & := &  x^k_{i_k} - \hat{u}^k - \lambda(G_{i_k}x^k_{i_k} - G_{i_k}\hat{u}^k), \vspace{1ex}\\
x^{k+1}_{i_k} & := & x_{i_k}^k - \frac{\eta }{\mbf{p}_{i_k}}(  g_{i_k}^k - \gamma g_{i_k}^{k-1}), \vspace{1ex}\\
u_{i_k}^{k+1} & := & x^{k+1}_{i_k} - \lambda G_{i_k}x^{k+1}_{i_k}.
\end{array}\right.
\vspace{-0.5ex}
}
\State\hspace{2.5ex}\label{step:A1_o6}[\textit{Communication}] User $i_k$ sends  $\Delta{u}_{i_k}^k := u^{k+1}_{i_k} - u^k_{i_k}$ back to server.
\State\hspace{2.5ex}\label{step:A1_o7}[\textit{Server update}] Server updates $\bar{u}^{k+1} := \bar{u}^k + \tfrac{1}{n}\Delta{u}^k_{i_k}$ and $\hat{u}^{k+1} := J_{\lambda T}(\bar{u}^{k+1})$. 
\State\hspace{0ex}{\bfseries End For}
\end{algorithmic}
\end{algorithm}

Finally, we apply Theorem~\ref{th:RCOG_convergence} to establish the convergence of Algorithm~\ref{alg:A1}.
However, we need to interpret Assumption~\ref{as:A1} in the context of \eqref{eq:FedNI} and \eqref{eq:S_operator2} as follows. 

%%% Assumption A.3.
\begin{assumption}\label{as:A3}
For \eqref{eq:FedNI}, each $G_i$ is $L$-Lipschitz continuous for $i\in [n]$ and $T$ is maximally monotone.
For any $(x_1, u_1) \in \gra{G_1 + n T}$ and $(x_i, u_i) \in \gra{G_i}$ for $i = 2, \cdots, n$ and $s_i \in \Espace$ $(i\in [n])$ such that $\sum_{i=1}^ns_i = 0$,  there exist $\rho > 0$ and  a solution $x^{\star} \in \zer{G+T}$ such that $8L\rho \leq 1$ and $\iprods{u_i + s_i, x_i - x^{\star}} \geq -\rho\norms{u_i + s_i}^2$ for all $i\in [n]$.
\end{assumption}
Note that Assumption~\ref{as:A3} does not guarantee the monotonicity of $G_i$.
It is equivalent to a $\rho$-``star'' co-hypomonotone of $\mbf{G} + \mbf{T}_1 + \partial{\delta}_{\Lc}$ in \eqref{eq:FedNI_reform} (or equivalently, a weak Minty solution of \eqref{eq:FedNI_reform} exists).
Hence, it also covers a class of nonmonotone operators $G_i$ for all $i \in [n]$.

%%% Theorem 5.
\begin{theorem}\label{th:FedOG_convergence}
\revised{
Assume that $\zer{G + T}\neq\emptyset$ and Assumption~\ref{as:A3} holds for \eqref{eq:FedNI}.
Let $\set{ (x_i^k, \hat{u}^k)}$ be generated by Algorithm~\ref{alg:A1} so that its parameters satisfy
\myeq{eq:FedOG_param}{
\begin{array}{l}
\underline{\lambda} \leq \lambda \leq \bar{\lambda}, \quad \frac{4}{4 + \mbf{p}_{\min}} <  \gamma < 1, \quad \text{and} \quad 0 < \eta \leq \frac{(4+\mbf{p}_{\min} )\sqrt{\mbf{p}_{\min}}}{8\gamma L_s},
\end{array}
}
where $\underline{\lambda} := \frac{1 - 2\rho L - \sqrt{1 - 8L\rho}}{2L(1+L\rho)} > 0$, $\bar{\lambda} := \frac{1 - 2\rho L + \sqrt{1 - 8L\rho}}{2L(1+L\rho)} \geq \underline{\lambda}$, and $L_s := (1 + \lambda L)(2 + \lambda L)$.
Then, for given $\psi_{\min} := \frac{\mbf{p}_{\min}^2 }{16L_s^2} > 0$, the following bound holds:
\myeq{eq:FedOG_convergence}{
\frac{1}{K+1}\sum_{k=0}^K \E\Big[ \sum_{i=1}^n \norms{x_i^k - \hat{u}^k}^2 \Big] \leq \frac{3 \norms{x^0 - x^{\star}}^2 }{\psi_{\min} (1 - L\lambda)^2 ( K+1) }.
}
}
\end{theorem}

%%% Proof of Theorem 5.
\proof{\textbf{Proof.}}
First, given $\mbf{G}$ in \eqref{eq:FedNI_reform} and $\mbf{T}$ in \eqref{eq:S_operator2}, since $G_i$ is $L$-Lipschitz continuous for all $i \in [n]$, so is $\mbf{G}$. 
Alternatively, since $T$ is maximally monotone, so is $\mbf{T}$.
Moreover, $x^{\star}$ solves \eqref{eq:FedNI} iff $\mbf{x}^{\star} = [x^{\star}, \cdots, x^{\star}]$ solves \eqref{eq:FedNI_reform}.

Next, since $\mbf{S}_{\lambda}$ defined by \eqref{eq:S_operator2} is the same as the one in \eqref{eq:S_operator2_new} of Lemma~\ref{le:property_of_S_oper},  for any $\mbf{x} = [x_1, \cdots, x_n]$, if we take any $\mbf{v} \in \mbf{G}\mbf{x} + \mbf{T}_1\mbf{x} + \delta_{\Lc}\mbf{x}$, then by Assumption~\ref{as:A3}, we have $v_1 = u_1 + s_1 \in G_1x_1 + nTx_1 + s_1$ and $v_i = u_i + s_i \in G_ix_i + s_i$ for $i=2,\cdots, n$ and $\sum_{i=1}^ns_i = 0$.
Thus the condition$\iprods{u_i + s_i, x_i - x^{\star}} \geq -\rho\norms{u_i + s_i}^2$ of Assumption~\ref{as:A3} is equivalent to the condition $\iprods{\mbf{v}, \mbf{x} - \mbf{x}^{\star}} \geq -\rho\norms{\mbf{v}}^2$ of Lemma~\ref{le:property_of_S_oper}.

Now, since $8L\rho \leq 1$, let choose $\underline{\lambda} := \frac{1 - 2\rho L - \sqrt{1 - 8L\rho}}{2L(1+L\rho)} \leq \lambda \leq \bar{\lambda} := \frac{1 - 2\rho L + \sqrt{1 - 8L\rho}}{2L(1+L\rho)}$.
Then, it is obvious to check that $\underline{\lambda} > 0$ and $L\lambda  \leq L\bar{\lambda} < 1$.
By Lemma~\ref{le:property_of_S_oper}, we have $\iprods{\mbf{S}_{\lambda}\mbf{x}, \mbf{x} - \mbf{x}^{\star}} \geq 0$. 

From our derivation, we know that Algorithm~\ref{as:A1} is exactly equivalent to \ref{eq:naRCOG_scheme} applying to solve $\mbf{S}_{\lambda}\mbf{x}^{\star} = 0$ for $\mbf{S}_{\lambda}$ defined by \eqref{eq:S_operator2}.
By Theorem~\ref{th:RCOG_convergence}, since $\sigma_i = 1$, we can choose $\frac{4}{4 + \mbf{p}_{\min}} <  \gamma < 1$ and  $0 < \eta \leq \frac{(4+\mbf{p}_{\min} )\sqrt{\mbf{p}_{\min}}}{8\gamma L_s}$ as shown in \eqref{eq:FedOG_param}, where $L_s := (1 + \lambda L)(2 + \lambda L) > 0$.
Then, given $\psi_{\min}$ as in Theorem~\ref{th:FedOG_convergence}, utilizing \eqref{eq:RBC_PEG_ergodic_rate}, we get 
\myeqn{
\frac{1}{K+1}\sum_{k=0}^{K} \E\big[\norms{\mbf{S}_{\lambda}\mbf{x}^k}^2 \big] \leq  \frac{3}{\psi_{\min} ( K+1) }\norms{\mbf{x}^0 - \mbf{x}^{\star}}^2.
}
Finally, by the $L$-Lipschitz continuity of $G_i$, we have $\norms{\mbf{S}_{\lambda}\mbf{x}^k}^2 = \sum_{i=1}^n\norms{x_i^k - \hat{u}^k - \lambda(G_ix^k_i - G_i\bar{u}^k)}^2 \geq \sum_{i=1}^n(1 - L\lambda)^2\norms{x_i^k - \hat{u}^k}^2$.
In addition, it is obvious that $\norms{\mbf{x}^0 - \mbf{x}^{\star}}^2 = n\norms{x^0 - x^{\star}}^2$.
Using these relations into the last inequality, we obtain \eqref{eq:FedOG_convergence}.
\Eproof
\endproof
%%% End of proof.

\beforesubsec
\subsection{Accelerated federated Douglas-Rachford splitting algorithm.}\label{subsec:FedDR}
\aftersubsec
We consider the case $G_i$ ($i\in [n]$) in \eqref{eq:FedNI} are maximally monotone and possibly multivalued.
Hence, we use the following DRS operator to reformulate \eqref{eq:FedNI_reform} into \eqref{eq:FedNI}:
\myeq{eq:DR_residual}{
\mbf{V}_{\beta}\mbf{u} := \tfrac{1}{\beta}\left( J_{\beta\mbf{T}}\mbf{u} - J_{\beta\mbf{G}}(2 J_{\beta\mbf{T}}\mbf{u} - \mbf{u}) \right),
}
where $\beta > 0$ is given, and $J_{\beta \mbf{G}}$ and $J_{\beta\mbf{T}}$ are the resolvents of $\beta\mbf{G}$ and $\beta\mbf{T}$, respectively with $\mbf{T} := \mbf{T}_1 + \partial{\delta}_{\Lc}$.
Here, we also use $\mbf{u} := [u_1, \cdots, u_n]$.
As shown in \cite{tran2022connection}, $\mbf{V}_{\beta}$ is $\beta$-co-coercive and $\dom{\mbf{V}_{\beta}} = \R^{np}$.
Moreover, $\mbf{x}^{\star}$ is a solution of \eqref{eq:FedNI_reform} iff there exists $\mbf{u}^{\star}$ such that $\mbf{V}_{\beta}\mbf{u}^{\star} = 0$ and $\mbf{x}^{\star} = J_{\beta \mbf{T}}\mbf{u}^{\star}$.

First, we directly apply  \ref{eq:ARCOG_scheme} to solve $\mbf{V}_{\beta}\mbf{u}^{\star} = 0$ and exploit the separable structure of $\mbf{G}$ and the form of $\mbf{T}$ to obtain the following scheme:
\myeq{eq:ArcDRS_scheme}{
\arraycolsep=0.2em
\left\{\begin{array}{lcl}
\hat{u}^{k-1} & := & J_{\beta T}\big( \frac{1}{n}\sum_{i=1}^nu_i^{k-1} \big), \quad\qquad\qquad\text{(previously computed -- no cost)}\vspace{1ex}\\
 \hat{u}^k & := & J_{\beta T}\big( \frac{1}{n}\sum_{i=1}^nu_i^k \big), \qquad\qquad\qquad\text{(centralized step on server)}\vspace{1ex}\\
\hat{g}^{k-1}_{i_k} &:= & \hat{u}^{k-1} - J_{\beta G_{i_k}}(2\hat{u}^{k-1} - u_{i_k}^{k-1}), \ \quad \text{only block $i_k$}, \vspace{1ex}\\
g^k_{i_k} &:= & \hat{u}^k -  J_{\beta G_{i_k}}(2\hat{u}^k - u^k_{i_k}), \ \ \, \quad\qquad \text{only block $i_k$}, \vspace{1ex}\\
u^{k+1}_i & := & 
\left\{\begin{array}{ll}
u^k_i + \theta_k(u^k_i - u^{k-1}_i) - \frac{\eta_k }{\beta\mbf{p}_{i}}\left( g^k_i - \gamma_k\hat{g}^{k-1}_i \right), &\text{if $i = i_k$}, \vspace{1ex}\\
u^k_i + \theta_k(u^k_i - u^{k-1}_i), &\text{otherwise},
\end{array}\right.
\end{array}\right.
\tag{AcFedDR}
}
where $\mbf{u}^0 := [u^0, \cdots, u^0]$ with $u^0 \in \R^p$ is given, $\mbf{u}^{-1} := \mbf{u}^0$, $i_k \in [n]$ is randomly sampled  based on \eqref{eq:prob_cond1}, and $\theta_k$, $\eta_k$, and $\gamma_k$ are chosen as in \ref{eq:ARCOG_scheme}.
We call this scheme the \underline{ac}celerated \underline{fed}erated \underline{DR} method (abbreviated by \textbf{\ref{eq:ArcDRS_scheme}}).

Next, we use the trick in Subsection~\ref{subsec:pracitice_RBC_scheme} to eliminate $u_i^k$ in  \ref{eq:ArcDRS_scheme}.
Since $u_i^k = z^k_i + c_kw^k_i$, we have $\bar{u}^k := \frac{1}{n}\sum_{i=1}^nu_i^k = \frac{1}{n}\sum_{i=1}^n(z^k_i + c_kw_i^k) = \bar{z}^k + c_k\bar{w}^k$, where $\bar{z}^k := \frac{1}{n}\sum_{i=1}^nz^k_i$ and $\bar{w}^k := \frac{1}{n}\sum_{i=1}^nw_i^k$.
However, at each iteration $k$, only $w_{i_k}^k$ and $z_{i_k}^k$ are updated, we have
\myeqn{
\arraycolsep=0.2em
\left\{\begin{array}{lcl}
\bar{z}^{k+1} := \frac{1}{n}\sum_{i=1}^nz^{k+1}_i  = \frac{1}{n}\sum_{i=1}^nz_i^k + \frac{1}{n}(z^{k+1}_{i_k} - z^k_{i_k}) = \bar{z}^k + \frac{1}{n}\Delta{z}^k_{i_k}, \vspace{1.5ex}\\
\bar{w}^{k+1} := \frac{1}{n}\sum_{i=1}^nw^{k+1}_i  = \frac{1}{n}\sum_{i=1}^nw_i^k + \frac{1}{n}(w^{k+1}_{i_k} - w^k_{i_k}) = \bar{w}^k + \frac{1}{n}\Delta{w}^k_{i_k},
\end{array}\right.
}
where $\Delta{z}_{i_k}^k := z^{k+1}_{i_k} - z^k_{i_k}$ and $\Delta{w}^k_{i_k} := w^{k+1}_{i_k} - w^k_{i_k}$.

Now, we are ready to specify  \ref{eq:ArcDRS_scheme} for solving \eqref{eq:FedNI} as in Algorithm~\ref{alg:A2}.

%%%% Algorithm 2.
\begin{algorithm}[hpt!]\caption{(Accelerated Federated Douglas-Rachford Algorithm (\textbf{\ref{eq:ArcDRS_scheme}}))}\label{alg:A2}
\normalsize
\begin{algorithmic}[1]
\State\label{step:i0}{\bfseries Initialization:} Input an initial point $u^0 \in \R^p$ and set $c_0 := c_{-1} := 0$ and $\tau_0 := 1$.
\State \hspace{2.5ex}Initialize each user $i$ with $z_i^0 = z_i^{-1} := u^0$ and $w^0_i = w_i^{-1} := 0$ for $i \in [n]$.
\State \hspace{2.5ex}Initialize sever with $\hat{u}^0 = \hat{u}^{-1} := u^0$, $\bar{z}^0 := 0$, and $\bar{w}^0 := 0$.
\State\hspace{0ex}\label{step:o1}{\bfseries For $k := 0,\cdots, k_{\max}$ do}
\vspace{0.25ex}   
\State\hspace{2.5ex}\label{step:o2}Sample an active user $i_k \in [n]$ following the probability distribution \eqref{eq:prob_cond1}.
\State\hspace{2.5ex}\label{step:o3}[\textit{Communication}] Server sends $\hat{u}^k$ and $\hat{u}^{k-1}$ to the active user $i_k$.
\State\hspace{2.5ex}\label{step:o4}[\textit{Local update}] User $i_k$ updates its iterates $w_{i_k}^{k+1}$ and $z^{k+1}_{i_k}$ as
\vspace{-0.5ex}
\myeqn{
\arraycolsep=0.2em
\left\{\begin{array}{lcl}
\hat{g}^{k-1}_{i_k} &:= & \hat{u}^{k-1} -  J_{\beta G_{i_k}}(2\hat{u}^{k-1} - z^{k-1}_{i_k} - c_{k-1}w^{k-1}_{i_k}), \vspace{1ex}\\
g^k_{i_k} &:= & \hat{u}^k -  J_{\beta G_{i_k}}(2\hat{u}^k - z^k_{i_k} - c_kw^k_{i_k}), \vspace{1ex}\\
w^{k+1}_{i_k} &:= & w^k_{i_k} +  \Delta{w}^k_{i_k}, \quad \text{where} \quad \Delta{w}^k_{i_k} := - \frac{\eta_k}{\beta\tau_{k+1}\mbf{p}_{i_k}}( g^k_{i_k} - \gamma_k \hat{g}^{k-1}_{i_k}), \vspace{1ex}\\
z^{k+1}_{i_k} &:= & z^k_{i_k} + \Delta{z}_{i_k}^k, \quad \text{where}\quad \Delta{z}_{i_k}^k := \frac{\eta_k c_k}{\beta\tau_{k+1}\mbf{p}_{i_k}}( g^k_{i_k} - \gamma_k\hat{g}^{k-1}_{i_k}).
\end{array}\right.
\vspace{-0.5ex}
}
\State\hspace{2.5ex}\label{step:o5}[\textit{Communication}] User $i_k$ sends $\Delta{w}^k_{i_k}$ and $\Delta{z}_{i_k}^k$ back to server.
\State\hspace{2.5ex}\label{step:o6}[\textit{Server update}] Server updates $\bar{w}^{k+1} := \bar{w}^k + \frac{1}{n}\Delta{w}^k_{i_k}$ and $\bar{z}^{k+1} := \bar{z}^k + \frac{1}{n}\Delta{z}^k_{i_k}$.
\Statex\hspace{18ex}\label{step:o6b}Then, it computes $\hat{u}^{k+1} := J_{\beta T}(\bar{z}^{k+1} + c_{k+1}\bar{w}^{k+1})$.
\State\hspace{0ex}{\bfseries End For}
\end{algorithmic}
\end{algorithm}

\begin{remark}\label{re:A1_and_A2_algs}
Both Algorithms~\ref{alg:A1} and \ref{alg:A2} are still synchronous but they only require the participation of one user $i_k$ $($or a subset of users$)$ at each communication round $k$.
They share some similarity  as SAGA \cite{Defazio2014} and its variant for co-coercive equations in \cite{davis2022variance}.
However, our algorithms can solve a more general class of problems  \eqref{eq:FedNI}, where $G$ is not necessarily monotone or co-coercive as in \cite{davis2022variance}.
In particular, since \eqref{eq:FedNI} is more general than that of \cite{davis2022variance}, we can specialize our algorithms to solve several special cases of \eqref{eq:FedNI} as in \cite{davis2022variance}.
Nevertheless, we omit these derivations.
\end{remark}

Finally, we are ready to state the convergence of  Algorithm~\ref{alg:A2}.

%%%%  Theorem 4.
\begin{theorem}\label{th:AcFedDR_convergence}
Assume that $\zer{G + T}\neq\emptyset$ and $G_i$ $(i\in [n])$ and $T$ in \eqref{eq:FedNI} are maximally monotone.
For given $r > 3$, $\beta > 0$, and $0 < \omega <  2\beta \cdot \min\{ \mbf{p}_i : i \in [n] \}$, let $\sets{ (w_i^k , z_i^k)}$ be generated by  Algorithm~\ref{alg:A2} using the update rules of $\theta_k$, $\eta_k$, and $\gamma_k$ as in \eqref{eq:ARCOG_scheme_update_pars}.
Let $u_i^k := z_i^k + c_kw^k_i$ for all $i\in[n]$ and $\hat{u}^k := J_{\beta T}\left(\frac{1}{n}\sum_{i=1}^nu_i^k \right)$.
Then, we have the following summable results:
\myeq{eq:ARCOG_scheme_result1_app3}{
\hspace{-2ex}
\arraycolsep=0.2em
\begin{array}{lcl}
\sum_{k=0}^{+\infty}(k + r + 1) \cdot \Exp{ \sum_{i=1}^n\norms{\hat{u}^k - J_{\beta G_i}(2\hat{u}^k - u_i^k) }^2 } & \leq & n \beta^2 C_1 \cdot \norms{u^0 - u^{\star}}^2, \vspace{1.5ex}\\
\sum_{k=0}^{+\infty}(2k + r + 1) \cdot \Exp{\sum_{i=1}^n\norms{u_i^{k+1} - u_i^k}^2} & \leq  & n [\omega r(r-1)\Lambda_0 + r + 1 ] \cdot \norms{u^0 - u^{\star}}^2,
\end{array}
\hspace{-3ex}
}
where $\Lambda_0$ and $C_1$ are given in Theorem~\ref{th:ARCOG_convergence} with $\bar{\beta}_i := \beta$ for $i \in [n]$.

Moreover, for given $C_0$ and $C_2$ in Theorem~\ref{th:ARCOG_convergence}, the following $\BigO{1/k^2}$-rates holds:
\myeq{eq:ARCOG_scheme_main_est2b_app3_BigO}{
\arraycolsep=0.3em
\left\{\begin{array}{lclclcl}
\Exp{ \sum_{i=1}^n\norms{\hat{u}^k - J_{\beta G_i}(2\hat{u}^k - u_i^k) }^2 }  & \leq &   \frac{2 n \beta^2 (r+1)^2 (C_0 +  C_2)  }{(r-1)^2 (k+ r - 1)(k + r + 2) } \ \cdot \norms{u^0 - u^{\star}}^2, \vspace{1.5ex}\\
\Exp{ \sum_{i=1}^n \norms{u_i^{k+1} - u_i^k}^2} & \leq &  \frac{ n \omega^2 C_2}{2(k + r + 2)^2} \cdot \norms{u^0 - u^{\star}}^2, \vspace{1.5ex}\\
\Exp{\norms{\hat{u}^{k+1} - \hat{u}^k }^2} & \leq &   \frac{ \omega^2 C_2}{2(k + r + 2)^2} \cdot \norms{u^0 - u^{\star}}^2, 
\end{array}\right.
}
and the following $\SmallO{1/k^2}$-convergence rates also hold:
\myeq{eq:ARCOG_scheme_main_est2b_app3_SmallO}{
\arraycolsep=0.3em
\left\{\begin{array}{lclclcl}
{\displaystyle\lim_{k\to\infty}} \Big\{ (k + r - 1)(k + r + 2) \cdot \Exp{ \sum_{i=1}^n\norms{\hat{u}^k - J_{\beta G_i}(2\hat{u}^k - u_i^k) }^2 } \Big\} & = & 0, \vspace{1ex}\\
{\displaystyle\lim_{k\to\infty}} \Big\{ (k + r + 2)^2 \cdot \Exp{ \sum_{i=1}^n \norms{u_i^{k+1} - u_i^k}^2} \Big\} & = & 0, \vspace{1ex}\\
{\displaystyle\lim_{k\to\infty}} \Big\{ (k + r + 2)^2 \cdot \Exp{\norms{\hat{u}^{k+1} - \hat{u}^k }^2} \Big\} & = &   0.
\end{array}\right.
}
Moreover, the iterate sequence $\sets{\hat{u}^k}$ also converges almost surely to $x^{\star} \in \zer{G+T}$.
\end{theorem}
%%%% End of Theorem 4.

%%% Proof of Theorem 4.
\proof{\textbf{Proof.}}
Since Algorithm~\ref{alg:A2} is equivalent to \ref{eq:ARCOG_scheme} applying to $\mbf{V}_{\beta}\mbf{u}^{\star} = 0$ for $\mbf{V}_{\beta}$ defined by \eqref{eq:DR_residual}, utilizing Theorem~\ref{th:ARCOG_convergence}, we obtain $\sum_{k=0}^{+\infty}(k + r + 1) \cdot \Exp{\norms{ \mbf{V}_{\beta}\mbf{u}^k }^2 } \leq C_1 \norms{\mbf{u}^0 - \mbf{u}^{\star}}^2$.
However, by Lemma~\ref{le:resolvent_oper}, we have $\sum_{i=1}^n\norms{\hat{u}_i^k - J_{\beta G_i}( 2\hat{u}^k - u_i^k)}^2 = \beta^2\norms{ \mbf{V}_{\beta}\mbf{u}^k }^2$.
Combining both expressions and the fact that $\norms{\mbf{u}^0 - \mbf{u}^{\star}}^2 = n \norms{u^0 - u^{\star}}^2$, we obtain the first line of \eqref{eq:ARCOG_scheme_result1_app3}.

The second line of \eqref{eq:ARCOG_scheme_result1_app3} is a direct consequence of the first line of \eqref{eq:ARCOG_scheme_result1} and $\norms{\mbf{u}^0 - \mbf{u}^{\star}}^2 = n \norms{u^0 - u^{\star}}^2$.
The bounds in \eqref{eq:ARCOG_scheme_main_est2b_app3_BigO} and the limits in \eqref{eq:ARCOG_scheme_main_est2b_app3_SmallO}  are consequences of \eqref{eq:ARCOG_scheme_main_est2b} and the fact that 
\myeqn{
\arraycolsep=0.2em
\begin{array}{lcl}
\norms{\hat{u}^{k+1} - \hat{u}^k}^2 & = & \norms{J_{\beta T}(\frac{1}{n}\sum_{i=1}^nu_i^{k+1}) - J_{\beta T}(\frac{1}{n}\sum_{i=1}^nu_i^k ) }^2 \vspace{1ex}\\
& \leq & \frac{1}{n^2}\norms{\sum_{i=1}^n( u_i^{k+1} - u_i^k)}^2 \vspace{1ex}\\
& \leq & \frac{1}{n}\sum_{i=1}^n\norms{u_i^{k+1} - u_i^k}^2.
\end{array}
} 

Similar to the proof of Statement (iii) in Theorem~\ref{th:ARCOG_convergence} we can show that $\sets{\mbf{u}^k}$ converges to $\mbf{u}^{\star} \in \zer{\mbf{V}_{\beta}}$ almost surely.
However, we have 
\myeqn{
\arraycolsep=0.2em
\begin{array}{lcl}
\norms{\hat{u}^k - x^{\star}}^2 & = &  \norms{J_{\beta T}(\frac{1}{n}\sum_{i=1}^nu_i^k) - J_{\beta T}(u^{\star}) }^2 \leq \frac{1}{n^2}\norms{\sum_{i=1}^n( u_i^k - u_i^{\star})}^2 \vspace{1ex}\\
& \leq & \frac{1}{n}\sum_{i=1}^n\norms{u_i^k - u_i^{\star}}^2 = \frac{1}{n}\norms{\mbf{u}^k - \mbf{u}^{\star}}^2.
\end{array}
} 
This relation and the almost sure convergence of $\sets{\mbf{u}^k}$ to $\mbf{u}^{\star}$ imply that $\sets{\hat{u}^k}$ also converges almost surely to $x^{\star} \in \zer{G + T}$, a solution of \eqref{eq:FedNI}.
\Eproof
\endproof
%%% End of proof.

By  Lemma~\ref{le:appox_sol}, we can see that $\hat{u}^k$ obtained from either Theorem~\ref{th:FedOG_convergence} or Theorem~\ref{th:AcFedDR_convergence} is an approximation to a solution $x^{\star}$ of \eqref{eq:FedNI}, but with different convergence rates.

%%%%%%%%%%%%%%%%%%%%%%%%%%%%%%%%%%%%%%%%%%%%%%%%%
%%% 5. Numerical Experiments.
%%%%%%%%%%%%%%%%%%%%%%%%%%%%%%%%%%%%%%%%%%%%%%%%%
\beforesec
\section{Numerical Experiments.}\label{sec:num_examples}
\aftersec
In this section, we verify both \ref{eq:naRCOG_scheme} and \ref{eq:ARCOG_scheme}  and compare them with the most related existing methods  in \cite{combettes2015stochastic,daskalakis2018training,kotsalis2022simple,peng2016arock}.
We also compare Algorithm~\ref{alg:A1} and Algorithm~\ref{alg:A2} with \texttt{FedAvg} a baseline FL algorithm in  \cite{mcmahan2017communication}.
All algorithms are implemented in Python and run on a MacBookPro. 2.8GHz Quad-Core Intel Core I7, 16Gb Memory.
%Our code can be found on github at \url{https://github.com/unc-optimization}.

%%%%%%%%%%%%%%%%%%%%%%%%%%%%%%%%%%%%%%%%%%%%%%%%%
%%% 5.1. Non-accelerated methods.
\beforesubsec
\subsection{Minimax problems.}\label{subsec:num_exp_RCOG}
\aftersubsec
Our first mathematical model for numerical experiments is the following minimax problem, which is possibly nonconvex-nonconcave:
\myeq{eq:minimax_exam1}{
\min_{u \in \R^{p_1}}\max_{v \in \R^{p_2}}\Big\{ \Lc(u, v) := f(u) + \tfrac{1}{2} u^{\top}Pu + b^{\top}u + u^{\top}Hv - \tfrac{1}{2}v^{\top}Qv - c^{\top}v - g(v) \Big\},
}
where $f : \R^{p_1}\to\Rext$ and $g : \R^{p_2}\to\Rext$ are two proper, closed, and convex functions, often used to handle constraints or regularizers,  $P\in \R^{p_1\times p_1}$ and $Q \in \R^{p_2\times p_2}$ are two given symmetric matrices, $H \in \R^{p_1\times p_2}$ is a given matrix, and $b\in\R^{p_1}$ and $c \in \R^{p_2}$ are given vectors.

We denote $x := [u, v] \in \R^p$ for $p := p_1 + p_2$ as the concatenation of the primal variable $u$ and its dual variable $v$.
We also define  $\mbf{G} := [[P, H], [-H, Q]]$ as the KKT matrix in $\R^{p\times p}$ formed from 4 blocks $P$, $H$, $-H$, and $Q$, and $\mbf{g} := [b, c] \in \R^p$ as the concatenation of $b$ and $c$.
When $f$ and $g$ are present, we denote $T := [\partial{f}; \partial{g}]$ as a maximally monotone mapping formed from the subdifferentials of $f$ and $g$.
Then,  the optimality condition of \eqref{eq:minimax_exam1} can be written as $0 \in \mbf{G}x + \mbf{g} + Tx \equiv Gx + Tx$, which is a special form of \eqref{eq:FedNI}, where $Gx := \mbf{G}x + \mbf{g}$ is a linear mapping.
If $T = 0$, i.e. $f$ and $g$ are absent in \eqref{eq:minimax_exam1}, then this optimality condition reduces to $Gx = 0$ (i.e. $\mbf{G}x + \mbf{g} = 0$ -- a linear system) as a special case of \eqref{eq:NE}.
Since $H \neq 0$,  $\mbf{G}$ is non-symmetric.
In addition, $P$ and $Q$ are possibly not positive semidefinite, implying that $G$ may not be monotone.

\beforesubsec
\subsection{Unconstrained minimax problems with synthetic data.}\label{subsec:unconstrained_minimax}
\aftersubsec
Since our \ref{eq:naRCOG_scheme} can solve a class of non-monotone \eqref{eq:NE}, we first apply it to solve a nonconvex-nonconcave instance of \eqref{eq:minimax_exam1} when $f = 0$ and $g = 0$.
As we are not aware of existing coordinate methods for non-monotone \eqref{eq:NE}, we only compare  \ref{eq:naRCOG_scheme} with its deterministic counterpart, i.e. the optimistic gradient method \eqref{eq:OG_scheme} in Subsection~\ref{subsec:FRB_splitting}.
Next, we apply our \ref{eq:ARCOG_scheme} to solve a convex-concave instance of \eqref{eq:minimax_exam1} as a special case of \eqref{eq:NE} under Assumption~\ref{as:A2} and compare it with methods in \cite{combettes2015stochastic,kotsalis2022simple,peng2016arock}.

Our synthetic data is generated as follows.
All matrices and vectors are generated randomly from standard normal distribution.
Matrix $P$ is averaged from $N$ samples as $P := \frac{1}{N}\sum_{i=1}^NU_iD_iU_i^{\top}$, where $U_i$ is an orthonormal matrix obtained from a QR factorization of a random matrix, and $D_i := \mathrm{diag}(d_{i1}, \cdots, d_{ip_1})$ is a diagonal matrix such that $d_{ij}$ is generated randomly and then clipped so that $d_{ij} \geq \underline{d}$.
Here, the lower bound $\underline{d}$ is used to control the nonconvexity-nonconcavity of \eqref{eq:minimax_exam1}.
If $\underline{d} < 0$, then \eqref{eq:minimax_exam1} is possibly nonconvex-nonconcave, or equivalently $G$ is possibly nonmonotone. 
Matrix $Q$ is also generated by the same way as $P$.
Matrix $H := \frac{1}{N}\sum_{i=1}^NH_i$, and vectors $b := \frac{1}{N}\sum_{i=1}^Nb_i$ and $c := \frac{1}{N}\sum_{i=1}^Nc_i$ are generated randomly.
Note that these finite-sums often appear in statistical learning and machine learning applications such as linear least-squares and support vector machine.

In our first test, we carry out two experiments: \textit{Experiment} 1 with $N=5000$, $p=1000$, and $\underline{d} = -0.1$, and \textit{Experiment 2} with $N = 10000$, $p = 2000$, and $\underline{d} := -0.1$.
Since it is challenging to estimate the parameter $\rho$ in Assumption~\ref{as:A1} and $\bar{\beta}$ in Assumption~\ref{as:A2}, we have manually tuned the stepsizes in all algorithms by a grid search to obtain the best possible performance for these algorithms in each experiment.
We always choose the same starting point $x^0 := 0.01\cdot \texttt{ones}(p)$.

For each experiment, we run \ref{eq:OG_scheme} and three variants of \ref{eq:naRCOG_scheme} corresponding to $n=10$, $n=50$, and $n=100$ block-coordinates on 10 problem instances and report the mean of the relative operator norm $\norms{Gx^k}/\norms{Gx^0}$ against the number of full dimension passes (i.e. one pass corresponds to $n$-block coordinates, which is equal to the full dimension $p$).
The results are depicted in Figure~\ref{fig:exam1} for two experiments after 200 full dimension passes (the same cost as 200 iterations of \ref{eq:OG_scheme}).

\begin{figure}[ht!]
\vspace{-3ex}
\centering
\includegraphics[width=1\textwidth]{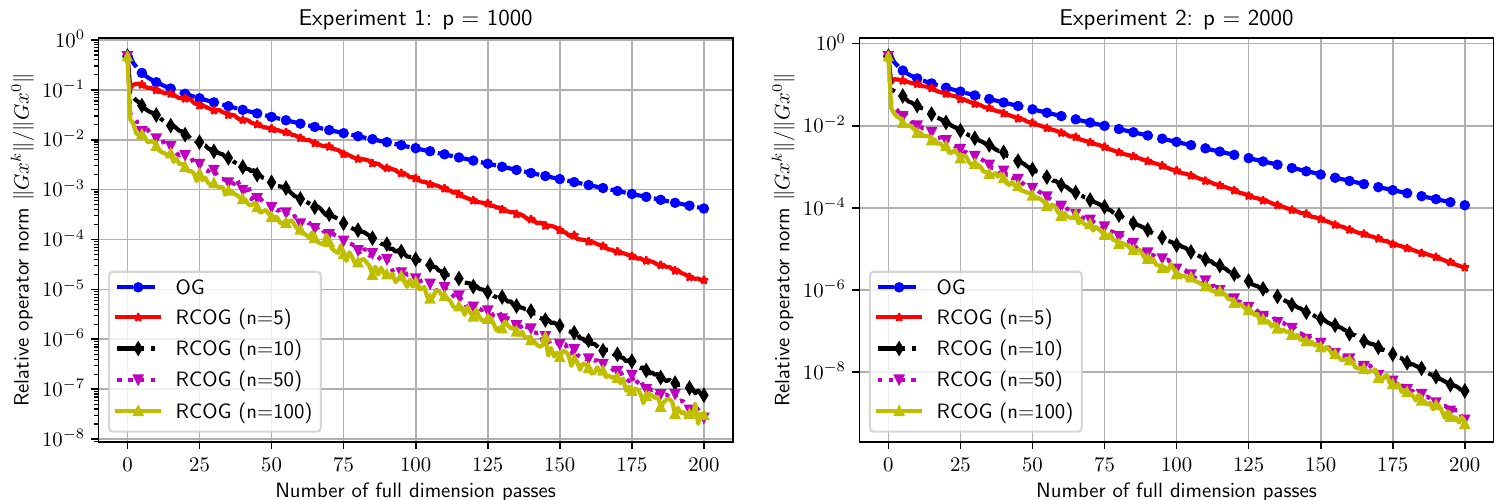}
\vspace{-2ex}
\caption{
Comparison of \ref{eq:OG_scheme} and three variants of \ref{eq:naRCOG_scheme} to solve a unconstrained nonconvex-nonconcave instance of \eqref{eq:minimax_exam1} on $2$ experiments.
The plot reveals  the mean of $10$ problem instances.
}
\label{fig:exam1}
\vspace{-2.5ex}
\end{figure}

As we can see from Figure~\ref{fig:exam1} that our \ref{eq:naRCOG_scheme} with $n \geq 10$ works well, and outperforms \ref{eq:OG_scheme}.
Its performance increases when we increase the number of block coordinates, but not significantly when $n$ goes from $10$ to $100$ in this experiment.

Next, we implement the randomized coordinate forward algorithm (RCFA) in \cite{combettes2015stochastic} (which is also equivalent to the synchronous variant of ARock in \cite{peng2016arock}), the Stochastic Block Operator Extrapolation (SBOE) in \cite[Algorithm 3]{kotsalis2022simple},  and our \ref{eq:ARCOG_scheme} to solve a convex-concave instance of \eqref{eq:minimax_exam1}.
The data is generated by the same way as in the first test, but we choose $\underline{d} = 0$ to guarantee the monotonicity of $G$.
Figure \ref{fig:exam2} shows the results of three algorithms on two experiments using $10$ problem instances.
Each algorithm has two variants corresponding to $n=50$ and $n=100$.

\begin{figure}[ht!]
\vspace{-3ex}
\centering
\includegraphics[width=1\textwidth]{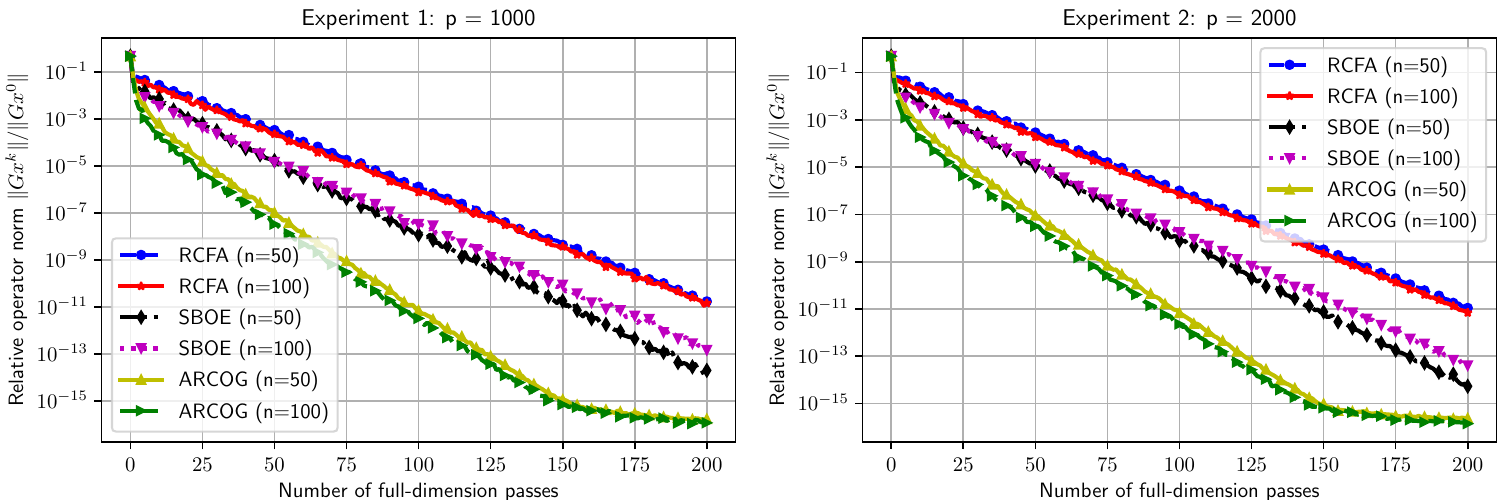}
\vspace{-2ex}
\caption{
Comparison of three algorithms: RCFA  \cite{combettes2015stochastic,peng2016arock}, SBOE \cite{kotsalis2022simple}, and \ref{eq:ARCOG_scheme}  to solve a convex-concave case of \eqref{eq:minimax_exam1} on $2$ experiments.
The plot is the mean of $10$ problem instances.
}
\label{fig:exam2}
\vspace{-2ex}
\end{figure}

Figure~\ref{fig:exam2} shows that our \ref{eq:ARCOG_scheme} highly outperforms its competitors in both experiments.
The solution accuracy of our method is already in the range of $10^{-15}$ after $200$ passes, which is sufficiently high.
We also see that SBOE has a better performance than RCFA in these two experiments. 
However, as a compensation, the cost per-pass of \ref{eq:ARCOG_scheme} and SBOE is twice that of RCFA.

\beforesubsec
\subsection{Constrained minimax problems with synthetic data.}\label{subsec:constrained_minimax_problem}
\aftersubsec
In this test, we consider the minimax problem \eqref{eq:minimax_prob} with constraints $u \in \Delta_{p_1}$ and $v \in\Delta_{p_2}$, where $\Delta_{p_1}$ and $\Delta_{p_2}$ are the standard simplexes in $\R^{p_1}$ and $\R^{p_2}$, respectively.
These constraints correspond to  $f(u) := \delta_{\Delta_{p_1}}(u)$ and $g(v) := \delta_{\Delta_{p_2}}(v)$ in \eqref{eq:minimax_exam1}, where $\delta_{\Xc}$ is the indicator of $\Xc$.
Therefore, if we define $x := [u, v] \in \R^p$ with $p := p_1+p_2$, $G$ as in Subsection~\ref{subsec:num_exp_RCOG}, and $T := [ \partial\delta_{\Delta_{p_1}}, \partial{\delta}_{\Delta_{p_2}}]$, then  the optimality condition of \eqref{eq:minimax_exam1} becomes $0 \in Gx + Tx$, which is in the form \eqref{eq:FedNI}.
Now, we reformulate $0 \in Gx + Tx$ into \eqref{eq:NE} as  $S_{\lambda}x^{\star} = 0$  using Tseng's FBFS operator $S_{\lambda}$ defined by\eqref{eq:S_operator2} and $\lambda := 0.5$.

We first apply our \ref{eq:naRCOG_scheme} to solve $S_{\lambda}x = 0$ and compare it with \ref{eq:OG_scheme} as in Subsection~\ref{subsec:num_exp_RCOG}.
We generate $Gx = \mbf{G}x + \mbf{g}$, where $\mbf{G} := [P, H; -H; Q]$ and $\mbf{g} := [b, c]$.
Matrix $P$ is constructed as $P := UD_1U^{\top}$, where $U$ is an orthonormal matrix computed from a QR factorization of a random matrix, and $D_1 = \mathrm{diag}(d_{11}, \cdots, d_{1p_1})$ is a diagonal matrix generated randomly and clipped such that $d_{ij} \geq \underline{d}$. 
Matrix $Q$ is generated similarly as $P$.
Matrix $H$ and vectors $b$ and $c$ are generated randomly as in Subsection~\ref{subsec:num_exp_RCOG}, but we do not take average of samples.
Then, we normalize $\mbf{G}$ as $\mbf{G}/\sqrt{p}$.
We first choose $\underline{d} := -0.5$ so that $G$ is possibly nonmonotone.

We carry out two experiments as in Subsection~\ref{subsec:num_exp_RCOG} but with large dimensions: \textit{Experiment 1} with $p = 5000$ and \textit{Experiment 2} with $p = 10000$.
We run all algorithms on these two experiments for $k_{\max} := 100$ full dimension passes using $10$ problem instances.
Then, we report the mean of $\Vert S_{\lambda}x^k\Vert/\Vert S_{\lambda}x^0\Vert$ against the number of full dimension passes.
The results are shown in Figure~\ref{fig:exam3}.

\begin{figure}[ht!]
\vspace{-2ex}
\centering
\includegraphics[width=1\textwidth]{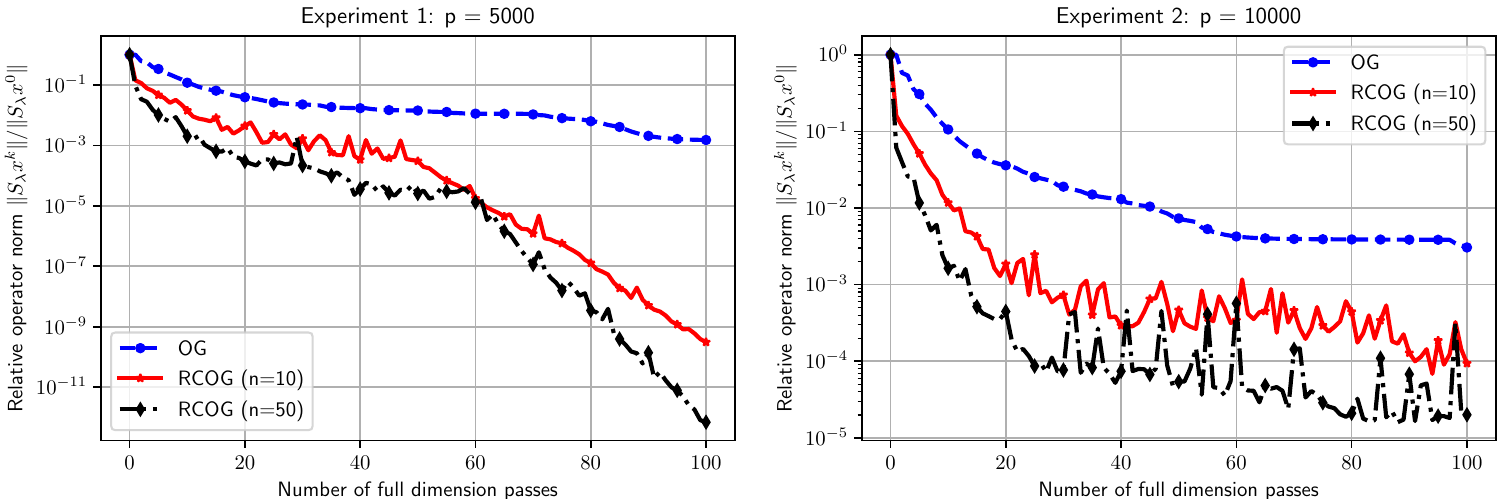}
\vspace{-1ex}
\caption{
Comparison of \ref{eq:OG_scheme} and two variants of \ref{eq:naRCOG_scheme} to solve a constrained nonconvex-nonconcave instance of \eqref{eq:minimax_exam1} on $2$ experiments with $p=5000$ and $p=10000$, respectively.  
The plot is on the average of $10$ problem instances.
}
\label{fig:exam3}
\vspace{-2ex}
\end{figure}

We can see that in both experiments, \ref{eq:naRCOG_scheme} performs much better than \ref{eq:OG_scheme}.
Moreover, \ref{eq:naRCOG_scheme} can also achieve high accuracy in the first experiment, while \ref{eq:OG_scheme} is saturated for a long period.

Next, we compare our \ref{eq:ARCOG_scheme} and SBOE in \cite{kotsalis2022simple} on the constrained convex-concave setting of \eqref{eq:minimax_exam1}.
We generate our data as in the previous test, but using $\underline{d} = 0$ to ensure that $G$ is monotone.
The results of this test is plotted in Figure~\ref{fig:exam4} for the case of $n=50$ block coordinates.

\begin{figure}[ht!]
\vspace{-0ex}
\centering
\includegraphics[width=1\textwidth]{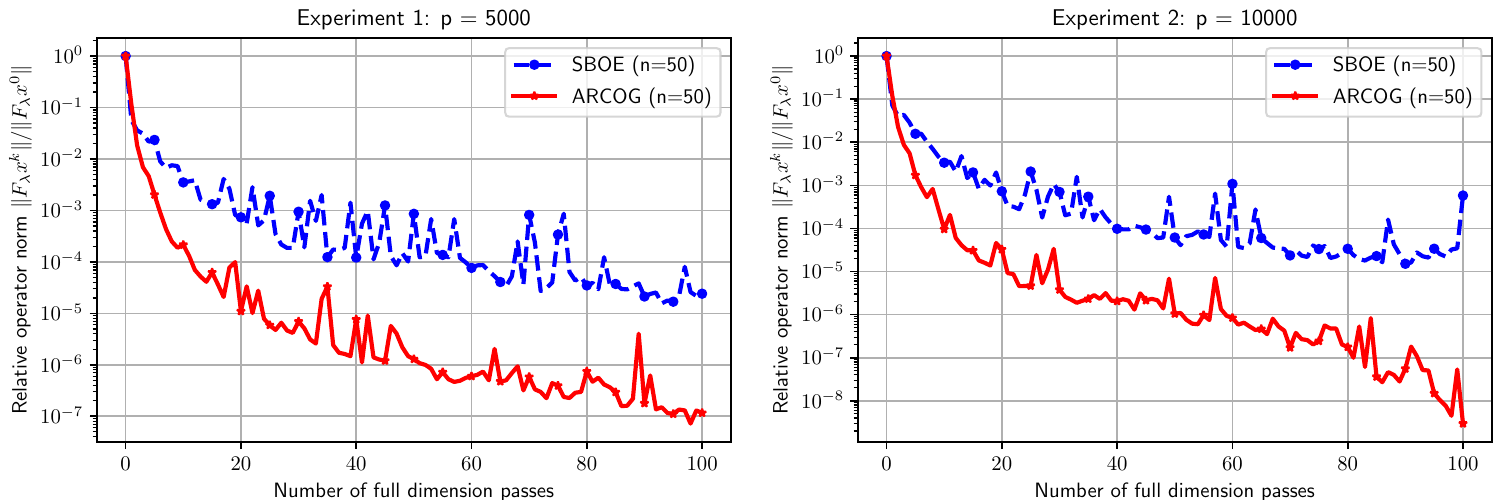}
\vspace{-2ex}
\caption{
Comparison of two algorithms: SBOE in \cite{kotsalis2022simple} and our \ref{eq:ARCOG_scheme}  to solve a constrained convex-concave instance of \eqref{eq:minimax_exam1} on $2$ experiments.
The plot is on the average of $10$ problem instances.
}
\label{fig:exam4}
\vspace{-2ex}
\end{figure}

Figure~\ref{fig:exam4} shows that while both methods essentially have the same per-iteration cost, \ref{eq:ARCOG_scheme} highly outperforms  SBOE on this experiment.
Clearly, this observation is not surprised since \ref{eq:ARCOG_scheme} is an accelerated method with a faster convergence rate, while SBOE is not.

\beforesubsec
\subsection{Regularized logistic regression with ambiguous features.}\label{subsec:logistic_reg_with_ambiguous_feature}
\aftersubsec
We consider a standard regularized logistic regression model associated with a given dataset $\sets{(\hat{X}_i, y_i)}_{i=1}^N$, where $\hat{X}_i$ is an i.i.d. sample of a feature vector and $y_i \in \{0, 1\}$ is the corresponding label of $\hat{X}_i$.
We assume that $\hat{X}_i$ is ambiguous, i.e. it belongs to one of $m$ possible examples  $\{X_{ij} \}_{j=1}^m$.
Since we do not know $\hat{X}_i$ to evaluate the loss, we consider the worst-case loss $f_i(w) := \max_{1\leq j \leq m}\ell( \iprods{X_{ij}, w}, y_i)$ computed from $m$ examples, where $\ell(\tau, s) := \log(1 + \exp(\tau)) - s\tau$ is the standard logistic loss.

Using the fact that $\max_{1\leq j\leq m}\ell_j(\cdot) = \max_{v\in\Delta_m}\sum_{j=1}^m v_j \ell_j(\cdot)$, where $\Delta_m$ is  the standard simplex in $\R^m$, we can model this regularized logistic regression into the following minimax problem:
\myeq{eq:logistic_reg_exam}{
\min_{w \in\R^d} \max_{v\in \R^m } \Big\{ \mathcal{L}(w, v) := \frac{1}{N} \sum_{i=1}^{N} \sum_{j=1}^m v_j  \ell ( \iprods{X_{ij}, w}, y_i ) + \tau R(w) - \delta_{\Delta_m}(v) \Big\},
}
where $R(w) := \norms{w}_1$ is an $\ell_1$-norm regularizer, $\tau > 0$ is a regularization parameter, and $\delta_{\Delta_m}$ is the indicator of $\Delta_m$ that handles the constraint $v \in \Delta_m$.

First, we denote $x := [w; v]$, and
\myeqn{
\begin{array}{lcl}
G_i := \big[ \sum_{j=1}^mv_j \ell'(\iprods{X_{ij},w}, y_i) X_{ij};\ -\ell(\iprods{X_{i1}, w}, y_i); \cdots; -\ell(\iprods{X_{im}, w}, y_i) \big], \ \textrm{and} \  T := [\tau \partial{R}(\cdot); \partial{\delta_{\Delta_m}}(\cdot)],
\end{array}
}
 where $\ell'(\tau, s) = \frac{\exp(\tau)}{1+\exp(\tau)} - s$.
Then, the optimality condition of \eqref{eq:logistic_reg_exam} can be written as \eqref{eq:FedNI}.
Next,  we reformulate this condition into \eqref{eq:NE} as $S_{\lambda}x^{\star} = 0$, where $S_{\lambda}$ is  Tseng's FBFS operator defined by \eqref{eq:S_operator2} and $\lambda := 0.5$.
Finally,  we implement \ref{eq:OG_scheme}, \ref{eq:naRCOG_scheme}, and \ref{eq:ARCOG_scheme} to solve $S_{\lambda}x^{\star} = 0$.
We use a fine tuning procedure to select appropriate stepsizes for all methods.

We test these algorithms on two real datasets: \texttt{w8a} (300 features and 49749 samples) and \texttt{leukemia} (7129 features and 38 samples) downloaded from \texttt{LIBSVM} \cite{CC01a}.
We first normalize the feature vector $\hat{X}_i$ such that each sample has unit norm, and add a column of  all ones to address the bias term.
To generate ambiguous features, we take the nominal feature vector $\hat{X}_i$ and add a random noise generated from a normal distribution of zero mean and variance $\sigma = 0.25$.
In our test, we choose $\tau := 5\times 10^{-4}$ and $m := 10$.
We run two variants of each algorithm: \ref{eq:naRCOG_scheme} or \ref{eq:ARCOG_scheme}  with $n=5$ and $n=10$, respectively. 
The relative norm $\norms{S_{\lambda}x^k}/\norms{S_{\lambda}x^0}$ against the number of full dimension passes is  plotted in Figure~\ref{fig:exam5} for both datasets.

\begin{figure}[ht!]
\vspace{-0ex}
\centering
\includegraphics[width=1\textwidth]{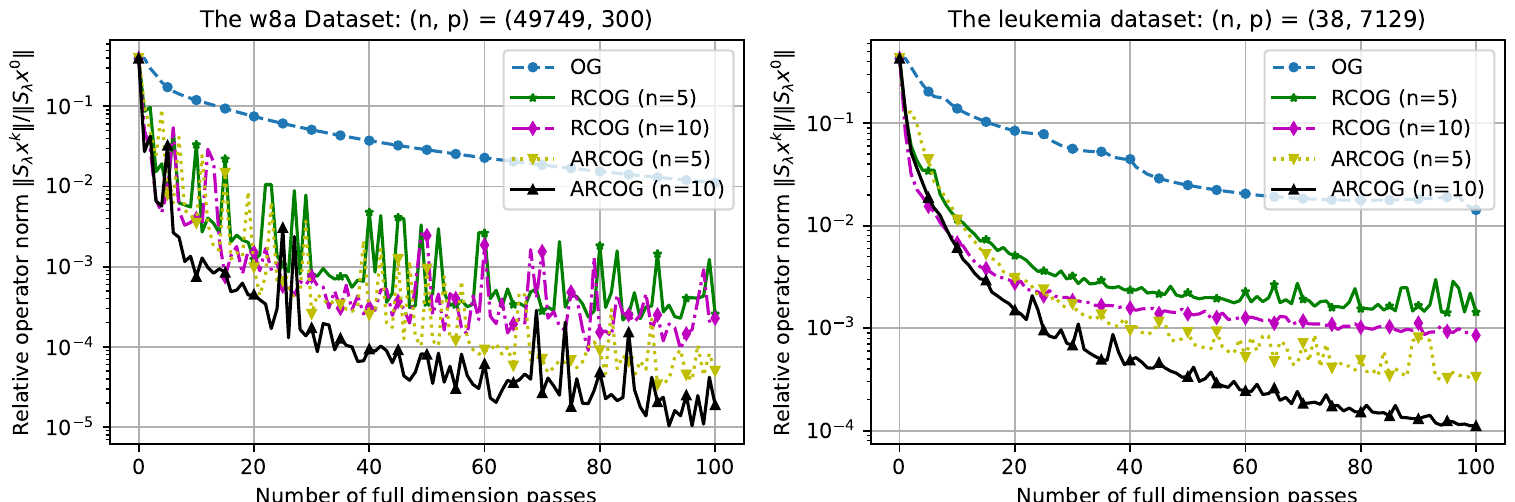}
\vspace{-2ex}
\caption{
The performance of 3 algorithms: \ref{eq:OG_scheme}, \ref{eq:naRCOG_scheme}, and \ref{eq:ARCOG_scheme}  for solving \eqref{eq:logistic_reg_exam} on \texttt{w8a}  and \texttt{leukemia}.
}
\label{fig:exam5}
\vspace{-2ex}
\end{figure}

As we can observe from Figure~\ref{fig:exam5} that both \ref{eq:naRCOG_scheme} and \ref{eq:ARCOG_scheme} perform much better than \ref{eq:OG_scheme}, but  \ref{eq:ARCOG_scheme} performs best.
If the number of block coordinates is $n=10$, our methods seem to work  better than the case $n=5$.
For \texttt{w8a}, both \ref{eq:naRCOG_scheme} and \ref{eq:ARCOG_scheme} are oscillated, while for \texttt{leukemia}, they have a slightly smoother progress.
Note that three methods have similar computational costs at each full dimension pass, but we do not implement any enhancement such as restart for  \ref{eq:ARCOG_scheme}.

%%%%%% 6.5 Experiments of Algorithms 1 and 2.
\beforesubsec
\subsection{Experiments of Algorithms \ref{alg:A1} and \ref{alg:A2}}\label{subsec:logistic_reg_for_A1_and_A2}
\aftersubsec
Finally, we validate our Algorithm \ref{alg:A1} and Algorithm \ref{alg:A2} and compare them with a baseline optimization algorithm in federated learning, called \textbf{Fed}erated \textbf{Av}era\textbf{g}ing (\texttt{FedAvg}) method in \cite{mcmahan2017communication}.
Since \texttt{FedAvg} was developed to solve optimization, we test three algorithms on a federated logistic regression problem as in \cite{mcmahan2017communication}.
Specifically, given a dataset $\{ (X_i, y_i) \}_{i=1}^N$, we partition it into $n$ [nearly] equal sub-datasets $\{ (X_{ij}, y_{ij}) \}_{j=1}^{N_i}$, $i = 1, \cdots, n$, and $\sum_{i=1}^n N_i = N$, where each sub-dataset corresponds to an user in a FL system. 
The underlying optimization problem of this FL system becomes
\begin{equation}\label{eq:fed_logistic_reg_min}
{\displaystyle\min_{w \in \R^p}} \Big\{ f(w; (\mathbf{X}, \mathbf{y})) = \frac{1}{n} \sum_{i=1}^n f_i \big( w; \{ (X_{ij}, y_{ij}) \}_{j=1}^{N_i} \big) \Big\},
\end{equation}
where $f_i \big( w; \sets{ (X_{ij}, y_{ij}) }_{j=1}^{N_i} \big) = \frac{n}{N} \sum_{j=1}^{N_i} \ell(w; (X_{ij}, y_{ij}))$ and  $\ell(w; (X_j, y_j)) := \log(1 + \exp(\iprods{ X_j, w} )) - y_j \iprods{ X_j, w}$ is the standard logistic loss. 
The optimality condition of \eqref{eq:fed_logistic_reg_min} is a special case of  \eqref{eq:fed_logistic_reg_FNI} and can be explicitly written as
\begin{equation}\label{eq:fed_logistic_reg_FNI}
\begin{array}{lcl}
Gw^\star \equiv \frac{1}{n} \sum_{i=1}^n G_i w^\star = 0,
\end{array} 
\end{equation}
where $G_i w := \nabla_w f_i\big( w,  \sets{ (X_{ij}, y_{ij}) }_{j=1}^{N_i} \big) = \frac{n}{N} \sum_{j=1}^{N_i} \nabla_w \ell(w; (X_{ij}, y_{ij}))$.

In this test, we implement  our Algorithm \ref{alg:A1} (\texttt{FedOG}) and Algorithm~\ref{alg:A2} (\texttt{AcFedDR}), and compare them with \texttt{FedAvg}.
Since it requires significant effort to estimate the model constants, we instead manually tune three algorithms to find the best possible parameters that work well for each algorithm.
In particular, for \texttt{FedAvg}, we choose the fraction of active users in each communication round to be $C = 0.2$ and the number of local training epochs to be  $E = 5$ as suggested in \cite{mcmahan2017communication}.

Similar to Subsection \ref{subsec:logistic_reg_with_ambiguous_feature}, we test these algorithms on two real datasets downloaded from \texttt{LIBSVM}  \cite{CC01a}: \texttt{w7a} (300 features and 24692 samples) and \texttt{a9a} (123 features and 32561 samples). 
For  preprocessing, we normalize the feature vector $X_i$ to obtain unit norm samples and add a column of all ones to handle the bias term. 
In the context of federated learning, we assume that our model has 30 users for the \texttt{w7a} dataset and 20 users for the \texttt{a9a} dataset.
Since we only test the performance of these algorithms in the context of optimization, we only report the relative operator norm $\frac{\norms{Gw^k}}{\norms{Gw^0}}$ against the number of communication rounds, and do not report the training and test errors as in machine learning.
The results of our test are plotted in Figure \ref{fig:exam6} for both datasets.

\begin{figure}[ht!]
\vspace{-2.0ex}
\centering
\includegraphics[width=1\textwidth]{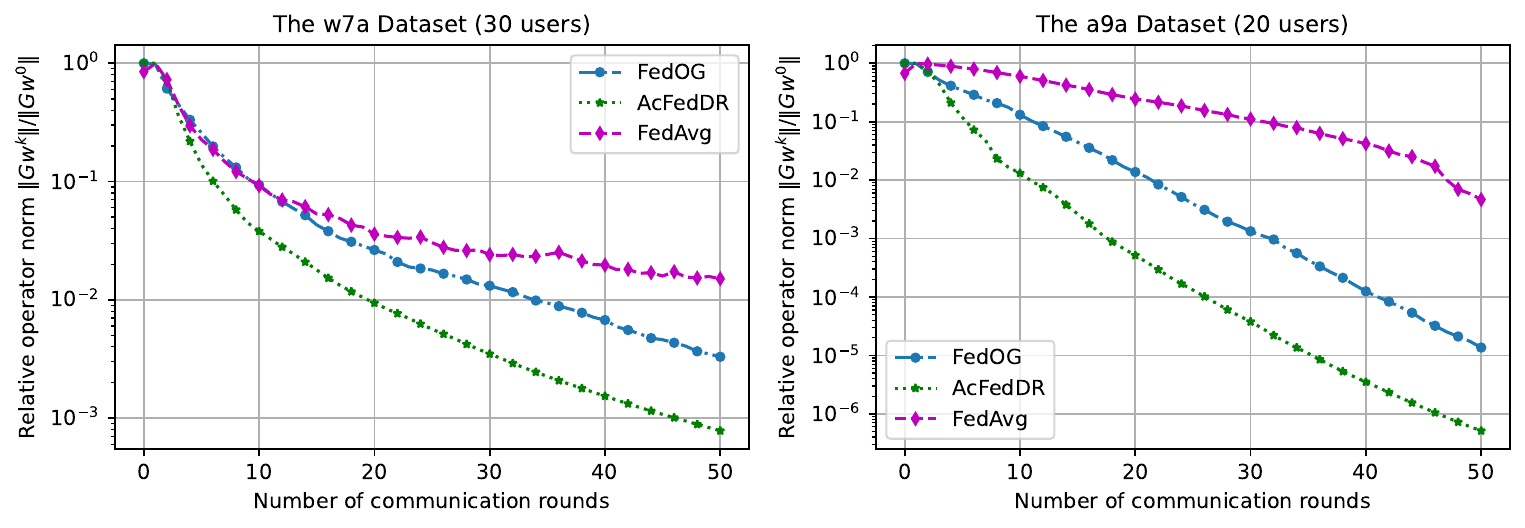}
\vspace{-2.5ex}
\caption{The performance of 3 algorithms: Algorithm \ref{alg:A1}, Algorithm \ref{alg:A2}, and \texttt{FedAvg} for solving \eqref{eq:fed_logistic_reg_FNI} with the two datasets: \texttt{w7a}  and \texttt{a9a} after $50$ communication rounds. }
\label{fig:exam6}
\vspace{-2ex}
\end{figure}

From Figure \ref{fig:exam6}, we observe that on the \texttt{w7a} dataset, both \texttt{FedOG} (Algorithm \ref{alg:A1}) and \texttt{AcFedDR} (Algorithm \ref{alg:A2}) significantly outperform \texttt{FedAvg}, and  \texttt{AcFedDR} achieves the best performance, reducing the overall relative operator norm to below $10^{-3}$ after $50$ communication rounds. 
\texttt{FedAvg} exhibits a reasonable performance but it is significantly slower than  \texttt{FedOG} and \texttt{AcFedOG}, probably due to the heterogeneity of the data.
Note that our methods are designed to avoid such an issue.

On the \texttt{a9a} dataset, the behavior of these algorithms remains consistent with their performance on the \texttt{w7a} dataset. 
However, both \texttt{FedOG} and \texttt{AcFedDR} achieve significantly better accuracy, where the relative operator norm reduces to approximately $10^{-5}$ and $10^{-6}$, respectively.

%%%% Acknowledgements.
\section*{Acknowledgements.}
This paper is based upon work partially supported by the National Science Foundation (NSF), grant no. NSF-RTG DMS-2134107, and the Office of Naval Research (ONR), grant No. N00014-20-1-2088 and N00014-23-1-2588.
The authors would like to thank Mr. Nghia Nguyen-Trung for his implementation and experiments on the last example in this paper.

%%%% Appendix 1.1.
\begin{APPENDICES}
\beforesec
\section{Technical Lemmas and Proofs.}\label{apdx:tech_results}
\aftersec
In this appendix, we provide some supporting lemmas and the full proof of our technical results in the main text.

\beforesubsec
\subsection{Technical lemmas.}\label{apdx:subsec:tech_results}
\aftersubsec
The following results will be used in Section~\ref{sec:FedNI_algorithms}.

%%%% Lemma 9
\begin{lemma}\label{le:resolvent_oper}
Let $\Lc := \sets{\mbf{x} = [x_1, \cdots, x_n] \in \R^{np} : x_i = x_1,  \forall i = 2, \cdots, n}$ and $\mbf{T}_1\mbf{x} := [n Tx_1, 0, \cdots, 0]$ for a given maximally monotone operator $T$.
Then, for any $\beta > 0$ and $\mbf{u} := [u_1, \cdots, u_n]$, we have 
\myeqn{
J_{\beta (\mbf{T}_1 + \partial{\delta_{\Lc}})}\mbf{u} = [\hat{u}, \cdots, \hat{u}], \quad \text{where} \quad \hat{u} := J_{\beta T}\Big(\tfrac{1}{n}\sum_{i=1}^nu_i \Big).
} 
\end{lemma}

%%%% Proof of Lemma 9.
\proof{\textbf{Proof.}}
Computing the resolvent $J_{\beta (\mbf{T}_1 + \partial{\delta_{\Lc}})}\mbf{u}$ of $\beta(\mbf{T}_1  + \partial{\delta_{\Lc}})$ at $\mbf{u} := [u_1, \cdots, u_n]$ is equivalent to solving the following system in $\tilde{u}_i$ and $\tilde{s}_i$ ($i \in [n]$):
\myeq{eq:resolvent_of_B}{
\arraycolsep=0.2em
\left\{\begin{array}{lcl}
0 \in n \beta  T\tilde{u}_1 + \beta \tilde{s}_1 + \tilde{u}_1 - u_1, \vspace{1ex}\\
0 = \beta \tilde{s}_i + \tilde{u}_i - u_i, \quad i = 2, \cdots, n,
\end{array}\right.
}
where $\tilde{\mbf{s}} := [\tilde{s}_1, \cdots, \tilde{s}_n] \in \partial{\delta}_{\Lc}(\tilde{\mbf{u}}) \equiv \Lc^{\perp} := \sets{ \tilde{\mbf{s}} := [\tilde{s}_1, \cdots, \tilde{s}_n] : \sum_{i=1}^n\tilde{s}_i = 0}$.
The last line of \eqref{eq:resolvent_of_B} leads to $\tilde{u}_i = u_i - \beta \tilde{s}_i$ for $i=2,\cdots, n$.
Therefore, since $\tilde{\mbf{u}} \in \Lc$ and $\sum_{i=1}^n \tilde{s}_i = 0$, we have  $\tilde{u}_i = \tilde{u}_1$ for $i=2,\cdots, n$ and
\myeqn{
(n-1)\tilde{u}_1 = \sum_{i=2}^nu_i - \beta\sum_{i=2}^n\tilde{s}_i = \sum_{i=2}^nu_i - \beta\sum_{i=1}^n\tilde{s}_i + \beta \tilde{s}_1 = \sum_{i=2}^nu_i + \beta \tilde{s}_1.
}
This equation implies that $\tilde{u}_1 + \beta \tilde{s}_1 = n\tilde{u}_1 - \sum_{i=2}^nu_i$.
Substituting this expression into the first line of \eqref{eq:resolvent_of_B}, we get $0 \in n \beta T\tilde{u}_1 + n\tilde{u}_1 - \sum_{i=1}^nu_i$, or equivalently, $0 \in \beta T\tilde{u}_1 + \tilde{u}_1 - \frac{1}{n}\sum_{i=1}^nu_i$.
Solving this inclusion, we obtain  $\tilde{u}_1 = J_{\beta T}\left(\frac{1}{n}\sum_{i=1}^nu_i\right)$.
Let us define $\hat{u} := \tilde{u}_1 = J_{\beta T}\left(\frac{1}{n}\sum_{i=1}^nu_i\right)$.
Then, since $\tilde{u}_i = \tilde{u}_1 = \hat{u}$ for $i=2,\cdots, n$, we have $J_{\beta (\mbf{T}_1 + \partial{\delta_{\Lc}})}\mbf{u} = [\hat{u}, \cdots, \hat{u}]$ as desired.
\Eproof
\endproof
%%% End of proof.

The following lemma is modified  from \cite[Proposition 1]{ioan2023relaxed} to cover nonmonotone operators (i.e. \eqref{eq:FedNI_reform} possesses a weak Minty  solution).

%%% Lemma 10.
\begin{lemma}\label{le:property_of_S_oper}
Let $T : \R^p \rightrightarrows 2^{\R^p}$ be maximally monotone and $G : \R^p\to\R^p$ be $L$-Lipschitz continuous, and $\iprods{u, x - x^{\star}} \geq -\rho\norms{u}^2$ for any $(x, u) \in \gra{G+T}$ and $x^{\star} \in \zer{G+T}$, where $L > 0$ and $\rho \geq 0$ are given. 
For some $\lambda > 0$, we consider the following Tseng's FBFS operator 
\myeq{eq:S_operator2_new}{
S_{\lambda}x := x - J_{\lambda T}(x - \lambda Gx) - \lambda\left( Gx - G \circ J_{\lambda T}(x - \lambda Gx) \right).
}
Then, $S_{\lambda}$ is $L_s$-Lipschitz continuous with $L_s := (1 + \lambda L)(2 + \lambda L)$.
Moreover, $x^{\star} \in \zer{G+T}$ iff $S_{\lambda}x^{\star} = 0$.
For $\hat{\rho} := \frac{1 - \lambda L}{(1 + \lambda L)^2} - \frac{\rho}{\lambda}$, the following property holds:
\myeq{eq:S_property}{
\iprods{ S_{\lambda}x, x - x^{\star}} \geq \hat{\rho} \norms{S_{\lambda}x}^2, \quad \forall x \in \dom{G+T}.
}
If $8L\rho \leq 1$, then for any $\lambda$ such that $0 <  \frac{1 - 2L\rho - \sqrt{1 - 8L\rho}}{2L(1 + L\rho)} \leq \lambda \leq \frac{1 - 2L\rho + \sqrt{1 - 8L\rho}}{2L(1 + L\rho)}$, we have $\hat{\rho} \geq 0$.
\end{lemma}

%%% Proof of Lemma 10.
\proof{\textbf{Proof.}}
The statement $x^{\star} \in \zer{G+T}$ iff $S_{\lambda}x^{\star} = 0$ and the $L_s$-Lipschitz continuity of $S_{\lambda}$ were proven in  \cite[Proposition 1]{ioan2023relaxed}.
For any $x\in\dom{G+T}$, if we denote $z := J_{\lambda T}(x - \lambda Gx)$, then by \eqref{eq:S_operator2_new}, we have $S_{\lambda}x = x - z - \lambda(Gx - Gz)$.
Moreover, $u := \frac{1}{\lambda}S_{\lambda}x = \frac{1}{\lambda}(x - z) - (Gx - Gz) \in (G + T)z$, which shows that $(z, u) \in \gra{G+T}$.
In addition, by \cite[Proposition 1]{ioan2023relaxed}, we have $\norms{S_{\lambda}x} \leq (1 + \lambda L)\norms{x - z}$.

Now, for any $x^{\star}\in\zer{G+T}$, since $u := \frac{1}{\lambda}S_{\lambda}x \in (G+T)z$, by our assumption $\iprods{u, z - x^{\star}} \geq -\rho\norms{u}^2$, we have $\iprods{S_{\lambda}x, z - x^{\star}} \geq -\frac{\rho}{\lambda}\norms{S_{\lambda}x}^2$.
This inequality leads to
\myeqn{
\arraycolsep=0.2em
\begin{array}{lcl}
\iprods{S_{\lambda} x, x - x^{\star}} &\geq & \iprods{x - z - \lambda(Gx - Gz), x - z} - \frac{\rho}{\lambda}\norms{S_{\lambda}x}^2 \vspace{1ex}\\
& \geq &(1 - \lambda L)\norms{x - z}^2 - \frac{\rho}{\lambda}\norms{S_{\lambda}x}^2 \vspace{1ex}\\
& \geq & \left( \frac{1 - \lambda L}{(1 + \lambda L)^2} - \frac{\rho}{\lambda} \right)\norms{S_{\lambda}x}^2,
\end{array} 
}
which  proves \eqref{eq:S_property}.
Finally, since $\hat{\rho} := \frac{1 - \lambda L}{(1 + \lambda L)^2} - \frac{\rho}{\lambda}$, if $L(1 + L\rho)\lambda^2 - (1 - 2L\rho)\lambda + \rho \leq 0$, then $\hat{\rho} \geq 0$.
The last condition holds if $0 <  \frac{1 - 2L\rho - \sqrt{1 - 8L\rho}}{2L(1 + L\rho)} \leq \lambda \leq \frac{1 - 2L\rho + \sqrt{1 - 8L\rho}}{2L(1 + L\rho)}$, provided that $8L\rho \leq 1$.
\Eproof
\endproof
%%% End of proof.

%%% Proof of Proposition 1.
\beforesubsec
\subsection{The proof of Proposition~\ref{pro:practical_variant}.}\label{apdx:pro:practical_variant}
\aftersubsec
From \ref{eq:ARCOG_scheme}, we have $x^{k+1} - x^k = \theta_k(x^k - x^{k-1}) - \frac{\eta_k }{\mbf{p}_{i_k}} d^k_{[i_k]}$, where $d^k := Gx^k - \gamma_kGx^{k-1}$.
%%%
Let us assume that $\theta_k = \frac{\tau_{k+1}}{\tau_{k}}$ for a given positive sequence $\sets{\tau_k}$.
This relation leads to $\tau_{k+1} = \tau_k\theta_k$.
Moreover, we can write \ref{eq:ARCOG_scheme} as 
\myeqn{
\tfrac{1}{\tau_{k+1}}(x^{k+1} - x^k) = \tfrac{1}{\tau_k}(x^k - x^{k-1}) - \tfrac{\eta_k }{\mbf{p}_{i_k}\tau_{k+1}} d^k_{[i_k]}.
}
Now, if we introduce $w^k := \frac{1}{\tau_k}(x^k - x^{k-1})$, then \ref{eq:ARCOG_scheme} can be rewritten as
\myeqn{
x^k := x^{k-1} + \tau_k w^k \quad\text{and} \quad w^{k+1} := w^k - \tfrac{\eta_k }{\mbf{p}_{i_k}\tau_{k+1}} d^k_{[i_k]}.
}
By induction, we have $x^k = x^0 + \sum_{i=1}^k \tau_i w^i$.
Let us express this representation as
\myeqn{
\arraycolsep=0.3em
\begin{array}{lcl}
x^k &= & x^0 - \tau_1(w^2 - w^1) - (\tau_1 + \tau_2)(w^3 - w^2) - (\tau_1 + \tau_2 + \tau_3)(w^4 - w^3) \vspace{1ex}\\
&& - {~} \cdots - (\tau_1 + \cdots + \tau_{k-1} )(w^k - w^{k-1}) +  (\tau_1 + \cdots + \tau_{k-1} + \tau_k)w^k. 
\end{array}
}
Therefore, if we define $c_k := \sum_{i=1}^k\tau_i$ with a convention that $c_0 := 0$, and $\Delta{w}^k := w^{k+1} - w^k$, then we can write $x^k$ as
\myeqn{
\arraycolsep=0.3em
\begin{array}{lcl}
x^k &= & x^0 - \sum_{i=1}^{k-1} c_i\Delta{w}^i + c_kw^k.
\end{array}
}
If we introduce $z^k := x^0 - \sum_{i=1}^{k-1}c_i\Delta{w}^i$, then $z^k = z^{k-1} - c_{k-1}\Delta{w}^{k-1}$ with $z^0 := x^0$, and hence, $x^k = z^k + c_kw^k$.
Therefore, we can summarize our derivation above as
\myeqn{
\arraycolsep=0.3em
\left\{\begin{array}{lcl}
x^k & := & z^k + c_kw^k, \vspace{1ex}\\
w^{k+1} & := &  w^k - \frac{ \eta_k }{\mbf{p}_{i_k}\tau_{k+1}} d^k_{i_k}, \vspace{1ex}\\
z^{k+1} &:= & z^k - c_k\Delta{w}^k = z^k + \frac{c_k \eta_k }{\mbf{p}_{i_k}\tau_{k+1}} d^k_{i_k}.
\end{array}\right.
}
Eliminating $x^k$ and $x^{k-1}$ from $d^k$ of the last scheme, we can write \ref{eq:ARCOG_scheme} equivalently to \eqref{eq:ARCOG_practice_variant}.
Here, we set $z^0 = z^{-1} := x^0$ and $w^0 = w^{-1} := 0$.
Moreover, the parameters $\tau_k$ and $c_k$ are respectively updated as in \eqref{eq:ARCOG_practice_variant_param}.
From the above derivation, we conclude that $x^k := z^k + c_kw^k$ is identical to the one generated by \ref{eq:ARCOG_scheme}.
\Eproof
%%% End of proof.

\beforesubsec
\subsection{The proof of Lemma~\ref{le:RCOG_descent}.}\label{apdx:subsec:le:RCOG_descent}
\aftersubsec
For simplicity of notation, we will use the following conventions.
For given vectors $\eta, \gamma, \sigma \in\R^n_{++}$, operator $G$, and vectors $x, v \in \R^p$, we denote $\gamma\eta := (\gamma_i\eta_i)_{i=1}^n$, $\eta^2 := (\eta_i^2)_{i=1}^n$, (elementwise product), $\eta \circ Gx := \sum_{i=1}^n\eta_i [Gx]_i$, $\eta \circ \iprods{Gx, v} := \sum_{i=1}^n\eta_i\iprods{[Gx]_i, v_i}$, and $\eta \circ \norms{Gx}^2 := \sum_{i=1}^n\eta_i\norms{[Gx]_i}^2$.
Let us divide our proof into several steps as follows.

%%% Step 1.
\vspace{1ex}
\noindent\textit{\underline{Step 1}: Expanding $\norms{x^{k+1}  - x^{\star} + \gamma \eta \circ Gx^k}^2$}. 
First, we expand  $\norms{x^{k+1} - x^{\star} + \gamma \eta \circ Gx^k }^2$ of $\Pc_{k+1}$ as
\myeq{eq:RBC_PEG_proof1}{
\arraycolsep=0.2em
\begin{array}{lcl}
\Tc_{k+1} & := &  \norms{x^{k+1}  - x^{\star} +   \gamma\eta \circ Gx^k }^2  =  \norms{x^{k+1} - x^{\star}}^2 + 2\gamma\eta \circ \iprods{ Gx^{k}, x^{k+1} -x^{\star}}  +  \gamma^2\eta^2 \circ \norms{Gx^k}^2.
\end{array}
}
Next, using $x^{k+1}$ from \ref{eq:naRCOG_scheme}, we can expand the first term  of \eqref{eq:RBC_PEG_proof1} as follows:
\myeq{eq:RBC_PEG_proof2}{
\arraycolsep=0.2em
\begin{array}{lcl}
\norms{x^{k+1} -x^{\star}}^2 & \overset{\tiny\ref{eq:naRCOG_scheme}}{=} & \big\Vert x^k - x^{\star} - \frac{\eta_{i_k}}{\mbf{p}_{i_k}} \big( G_{[i_k]}x^k - \gamma_{i_k}  G_{[i_k]}x^{k-1} \big) \big\Vert^2  \vspace{0.5 ex}\\
& = & \norms{x^k - x^{\star}}^2 +  \frac{ \eta^2_{i_k} }{\mbf{p}_{i_k}^2}\norms{ [Gx^k]_{i_k} - \gamma_{i_k}  [Gx^{k-1}]_{i_k} }^2 \vspace{0.5ex}\\
&& - {~} \frac{ 2 \eta_{i_k}   }{ \mbf{p}_{i_k}} \iprods{ [Gx^k]_{i_k}, x^k_{i_k} -x^{\star}_{i_k} }  +  \frac{2 \gamma_{i_k}\eta_{i_k} }{\mbf{p}_{i_k}} \iprods{ [Gx^{k-1}]_{i_k}, x^k_{i_k} - x^{\star}_{i_k} }.
\end{array}
}
Alternatively, applying again \ref{eq:naRCOG_scheme}, we can process the second term of \eqref{eq:RBC_PEG_proof1} as
\myeq{eq:RBC_PEG_proof3}{
\hspace{-3ex}
\arraycolsep=0.1em
\begin{array}{lcl}
\iprods{Gx^k, x^{k+1} -x^{\star}}  &= & \iprods{Gx^k, x^k - x^{\star} -  \frac{\eta_{i_k} }{\mbf{p}_{i_k}} \big( G_{[i_k]}x^k  - \gamma_{i_k}\  G_{[i_k]}x^{k-1} \big) }   \vspace{1 ex}\\
& = & \iprods{ Gx^{k}, x^k - x^{\star}}  - \frac{\eta_{i_k}  }{\mbf{p}_{i_k}} \norms{ [Gx^k]_{i_k}}^2  + \frac{\gamma_{i_k} \eta_{i_k}   }{\mbf{p}_{i_k}} \iprods{ [Gx^k]_{i_k}, [Gx^{k-1}]_{i_k} }.
\end{array}
\hspace{-3ex}
}
Substituting \eqref{eq:RBC_PEG_proof2} and \eqref{eq:RBC_PEG_proof3} into \eqref{eq:RBC_PEG_proof1}, we can show that
\myeq{eq:RBC_PEG_proof4}{
\hspace{-2ex}
\arraycolsep=0.2em
\begin{array}{lcl}
\Tc_{k+1} & := & \norms{x^{k+1} - x^{\star} + \gamma\eta \circ Gx^k }^2  =  \norms{x^k - x^{\star}}^2  + 2 \gamma\eta \circ \iprods{Gx^{k}, x^k - x^{\star}}   +  \gamma^2\eta^2 \circ \norms{Gx^k}^2  \vspace{1ex}\\
&& + {~}  \frac{\eta^2_{i_k} }{\mbf{p}_{i_k}^2}\norms{ [Gx^k]_{i_k} - \gamma_{i_k} [Gx^{k-1}]_{i_k} }^2 - \frac{ 2 \eta_{i_k}   }{ \mbf{p}_{i_k}} \iprods{ [Gx^k]_{i_k}, x^k_{i_k} -x^{\star}_{i_k} }  \vspace{1ex}\\
&& + {~}  \frac{2 \gamma_{i_k}\eta_{i_k} }{\mbf{p}_{i_k}} \iprods{ [Gx^{k-1}]_{i_k}, x^k_{i_k} - x^{\star}_{i_k} } - \frac{2\gamma_{i_k}\eta_{i_k}^2 }{\mbf{p}_{i_k}} \norms{ [Gx^k]_{i_k} }^2  +  \frac{2\gamma_{i_k}^2\eta_{i_k}^2 }{\mbf{p}_{i_k}} \iprods{[Gx^k]_{i_k}, [Gx^{k-1}]_{i_k} }.
\end{array}
\hspace{-2ex}
}

%%% Step 2.
\noindent\textit{\underline{Step 2}: Applying the conditional expectation to $\Tc_{k+1}$}.
Using \eqref{eq:prob_cond1}, we have
\myeq{eq:RBC_PEG_proof5}{
\hspace{-3ex}
\arraycolsep=0.2em
\left\{\begin{array}{lcl}
\E_k\big[  \frac{ \eta_{i_k} }{ \mbf{p}_{i_k} } \iprods{ [Gx^k]_{i_k}, x^k_{i_k} - x^{\star}_{i_k} } \big] & = & \eta \circ \iprods{ Gx^k, x^k - x^{\star}}, \vspace{1ex}\\
\E_k\big[ \frac{ \gamma_{i_k}\eta_{i_k}   }{ \mbf{p}_{i_k}}  \iprods{ [Gx^{k-1}]_{i_k}, x^k_{i_k} - x^{\star}_{i_k} } \big] & = & \gamma\eta \circ \iprods{  Gx^{k-1}, x^k - x^{\star} }, \vspace{1ex}\\
\E_k\big[ \frac{ \gamma_{i_k}\eta_{i_k}^2  }{ \mbf{p}_{i_k} } \norms{ [Gx^k]_{i_k} }^2 \big] &= &  \gamma\eta^2 \circ \norms{Gx^k}^2, \vspace{1ex}\\
\E_k\big[ \frac{ \gamma_{i_k}^2 \eta_{i_k}^2 }{ \mbf{p}_{i_k} } \iprods{ [Gx^k]_{i_k}, [Gx^{k-1}]_{i_k} } \big] &= & \gamma^2\eta^2 \circ \iprods{ Gx^k, Gx^{k-1} }, \vspace{1ex}\\
\E_k\big[  \frac{ \eta_{i_k}^2 }{\mbf{p}_{i_k}^2 }\norms{ G_{[i_k]}x^k - \gamma_{i_k} G_{[i_k]}x^{k-1}}^2 \big] &= & \sum_{i=1}^n \frac{\eta_i^2 }{\mbf{p}_i}\norms{ [Gx^k]_i  - \gamma_i  [Gx^{k-1}]_i }^2.
\end{array}\right.
\hspace{-10ex}
}
Taking the conditional expectation $\E_k[\cdot ] = \E_{i_k}[ \cdot\mid \Fc_k ]$ of \eqref{eq:RBC_PEG_proof4} and then using \eqref{eq:RBC_PEG_proof5}, we can prove that
\myeq{eq:RBC_PEG_proof7}{
\hspace{-5ex}
\arraycolsep=0.2em
\begin{array}{lcl}
\E_k\big[ \Tc_{k+1} \big] 
& = & \norms{x^k - x^{\star}}^2  + 2\gamma \eta \circ \iprods{Gx^{k-1}, x^k - x^{\star}}  +  \gamma^2\eta^2 \circ \norms{Gx^k}^2  \vspace{1ex}\\
&& - {~}   2(1 - \gamma)\eta \circ \iprods{Gx^k, x^k - x^{\star}}  -   2\gamma\eta^2 \circ \norms{Gx^k}^2   +  2\gamma^2\eta^2 \circ \iprods{Gx^k, Gx^{k-1} }  \vspace{1ex} \\
&& + {~}  \sum_{i=1}^n \frac{\eta^2_i }{\mbf{p}_i}\norms{  [Gx^k]_i  - \gamma_i [Gx^{k-1}]_i }^2.
\end{array}
\hspace{-6ex}
}

%%% Step 3.
\noindent\textit{\underline{Step 3}: Using Assumption~\ref{as:A1}.}
From \eqref{eq:CE_assumption_A1_b} of Assumption~\ref{as:A1}, we also have
\myeq{eq:RBC_PEG_proof6}{
\hspace{-3ex}
\arraycolsep=0.1em
\left\{\begin{array}{lcl}
-\iprods{Gx^k, x^k - x^{\star}} & \leq  &  \rho \circ \norms{Gx^k}^2, \vspace{1ex}\\
2 \iprods{Gx^k, Gx^{k-1}} &= &  \norms{Gx^k}^2  +  \norms{ Gx^{k-1}}^2   -  \norms{Gx^k - Gx^{k-1}}^2.
\end{array}\right.
\hspace{-3ex}
}
Substituting \eqref{eq:RBC_PEG_proof6} into \eqref{eq:RBC_PEG_proof7}, we can further derive
\myeqn{ 
\arraycolsep=0.2em
\begin{array}{lcl}
\E_k\big[ \Tc_{k+1} \big] 
& \leq & \norms{x^k - x^{\star} + \gamma\eta \circ Gx^{k-1} }^2   -  2\gamma(1 - \gamma)\eta^2 \circ \norms{Gx^k}^2  -  \gamma^2\eta^2 \circ \norms{Gx^k - Gx^{k-1}}^2 \vspace{1ex}\\ 
&& + {~}  \sum_{i=1}^n \frac{\eta_i^2  }{\mbf{p}_i}\norms{ [Gx^k]_i  - \gamma_i [Gx^{k-1}]_i }^2  +   2(1 - \gamma) \eta \rho \circ \norms{Gx^k}^2.
\end{array}
}
Applying Young's inequality, for all $i \in [n]$, we can show that
\myeq{eq:RBC_PEG_proof8}{
\hspace{-2ex}
\arraycolsep=0.2em
\begin{array}{lcl}
 \norms{  [Gx^k]_i  -  \gamma_i [Gx^{k-1}]_i }^2 &\leq &  2\gamma_i^2 \norms{ [Gx^k]_i - [Gx^{k-1}]_i}^2 +  2(1 - \gamma_i)^2 \norms{[Gx^k]_i}^2.
\end{array}
\hspace{-2ex}
}
Substituting \eqref{eq:RBC_PEG_proof8} into last inequality of $\E_k\big[ \Tc_{k+1} \big]$, we can upper bound it as
\myeq{eq:RBC_PEG_proof10}{
\hspace{-2ex}
\arraycolsep=0.2em
\begin{array}{lcl}
\E_k\big[ \Tc_{k+1} \big] &\leq &  \norms{x^k - x^{\star} +  \gamma\eta \circ Gx^{k-1} }^2 +  \sum_{i=1}^n \gamma_i^2\eta^2_i \big( \frac{2}{\mbf{p}_i} - 1 \big) \cdot \norms{ [Gx^k]_i - [Gx^{k-1}]_i}^2. \vspace{1ex}\\
&& - {~}    \sum_{i=1}^n 2 \eta_i(1 - \gamma_i) \big[ \gamma_i \eta_i  -   \rho_i  - \frac{\eta_i(1 -\gamma_i)}{\mbf{p}_i}  \big]  \cdot \norms{ [Gx^k]_i}^2.
\end{array}
\hspace{-4ex}
}

%%% Step 4.
\noindent\textit{\underline{Step 4}: Bounding $\E_k\big[\norms{x^{k+1} - x^k}^2_{\sigma} \big]$.}
Using again  \ref{eq:naRCOG_scheme}, \eqref{eq:RBC_PEG_proof5}, and \eqref{eq:RBC_PEG_proof8}, we have
\myeq{eq:RBC_PEG_proof9}{
\hspace{-2ex}
\arraycolsep=0.2em
\begin{array}{lcl}
\E_k\big[\norms{x^{k+1} - x^k}^2_{\sigma} \big] & \overset{\tiny\ref{eq:naRCOG_scheme}}{=} &   \E_k\big[ \frac{\sigma_{i_k}\eta^2_{i_k} }{\mbf{p}_{i_k}^2}\norms{ [Gx^k]_{i_k} -  \gamma_{i_k} [Gx^{k-1}]_{i_k} }^2 \big]  
\overset{\tiny\eqref{eq:RBC_PEG_proof5}}{=}  \sum_{i=1}^n\frac{ \sigma_i \eta_i^2  }{\mbf{p}_i}\norms{ [Gx^k]_i  - \gamma_i [Gx^{k-1}]_i }^2 \vspace{0.5ex}\\
& \overset{\tiny\eqref{eq:RBC_PEG_proof8}}{\leq} &   \sum_{i=1}^n\frac{2  \sigma_i \gamma_i^2 \eta_i^2 }{\mbf{p}_i}\norms{ [Gx^k]_i - [Gx^{k-1}]_i}^2 +  \sum_{i=1}^n \frac{2 \sigma_i \eta_i^2(1 - \gamma_i)^2 }{\mbf{p}_i} \norms{[Gx^k]_i}^2.
\end{array}
\hspace{-3ex}
}

%%% Step 5.
\noindent\textit{\underline{Step 5}: Deriving \eqref{eq:RCOG_descent}.}
First, adding \eqref{eq:RBC_PEG_proof9} to \eqref{eq:RBC_PEG_proof10} and then utilizing \eqref{eq:RBC_PEG_potential}, we get
\myeqn{
\arraycolsep=0.2em
\begin{array}{lcl}
\E_k\big[\Pc_{k+1}\big] & \overset{\tiny\eqref{eq:RBC_PEG_potential}}{:=} & \E_k\big[ \Tc_{k+1} \big] +  \E_k\big[\norms{x^{k+1} - x^k}^2_{ \sigma } \big]  \vspace{1ex}\\
&\leq &  \norms{x^k  - x^{\star} + \gamma\eta \circ Gx^{k-1} }^2  +  \sum_{i=1}^n \frac{\gamma_i^2 \eta_i^2 ( 2  - \mbf{p}_i  +  2\sigma_i ) }{\mbf{p}_i} \cdot \norms{ [Gx^k]_i - [Gx^{k-1}]_i}^2. \vspace{1ex}\\
&& - {~}    \sum_{i=1}^n 2 \eta_i(1 - \gamma_i) \big[ \gamma_i\eta_i   -   \rho_i  - \frac{\eta_i(1 - \gamma_i)(1 + \sigma_i)}{\mbf{p}_i} \big] \cdot \norms{ [Gx^k]_i}^2.
\end{array}
}
Next, leveraging \eqref{eq:CE_assumption_A1_a} and imposing the condition $\frac{\gamma_i^2 \eta_i^2  (2 - \mbf{p}_i  + 2\sigma_i )}{\mbf{p}_i} \leq \frac{\sigma_i}{L_i^2}$ for all $i\in [n]$, we can derive from the last estimate that
\myeqn{ 
\hspace{-3ex}
\arraycolsep=0.2em
\begin{array}{lcl}
\E_k\big[\Pc_{k+1}\big] & \leq &  \Pc_k = \norms{x^k +  \gamma\eta \circ Gx^{k-1} - x^{\star}}^2   + \norms{x^k - x^{k-1}}^2_{\sigma} \vspace{1ex}\\
&& - {~}   \sum_{i=1}^n 2 \eta_i(1 - \gamma_i)\big[ \gamma_i\eta_i  - \rho_i  -  \frac{\eta_i(1 - \gamma_i)(1 + \sigma_i) }{\mbf{p}_i} \big] \cdot \norms{ [Gx^k]_i}^2,
\end{array}
\hspace{-3ex}
}
which proves  \eqref{eq:RCOG_descent}.
Finally, notice that the condition $\frac{\gamma_i^2 \eta_i^2  (2  - \mbf{p}_i + 2 \sigma_i ) }{\mbf{p}_i} \leq \frac{\sigma_i}{L_i^2}$ above is equivalent to $L_i^2\gamma_i^2\eta_i^2 ( 2 - \mbf{p}_i + 2 \sigma_i )  \leq \mbf{p}_i \sigma_i$ for all $i\in [n]$ as stated in Lemma~\ref{le:RCOG_descent}.
\Eproof
%%% End of proof.

%%% Proof of Lemma 2.1.
\beforesubsec
\subsection{The proof of Lemma~\ref{le:ARCOG_key_est1}.}\label{apdx:subsec:le:ARCOG_key_est1}
\aftersubsec
Given $r > 0$ and $\mu > 0$, we introduce the following function:
\myeq{eq:Qck_func}{
\Qc_k := \norms{r(x^{k-1} - x^{\star}) + t_k(x^k - x^{k-1}) }^2 + \mu r \norms{x^{k-1} - x^{\star}}^2.
}
Now, we divide the proof of Lemma~\ref{le:ARCOG_key_est1} into several steps as follows.

%%% Step 1. 
\vspace{0.5ex}
\noindent\textit{\underline{Step 1}: Expanding $\Qc_k - \Qc_{k+1}$}.
If we denote $d^k := Gx^k  -  \gamma_kGx^{k-1}$, then we can rewrite \ref{eq:ARCOG_scheme} as $x^{k+1} - x^k = \theta_k(x^k - x^{k-1}) -  \frac{ \eta_k }{\mbf{p}_{i_k}}d_{[i_k]}^k$.
Using this representation, we can show that
\myeqn{ 
\arraycolsep=0.2em
\begin{array}{lcll}
 \norms{r(x^k - x^{\star}) + t_{k+1}(x^{k+1} - x^k) }^2 & = & \norms{r(x^k - x^{\star}) + t_{k+1}\theta_k(x^k - x^{k-1}) -    t_{k+1} \eta_k \mbf{p}_{i_k}^{-1}d_{[i_k]}^k }^2 \vspace{1ex}\\
&= & r^2 \norms{x^k - x^{\star}}^2 + t_{k+1}^2\theta_k^2\norms{x^k - x^{k-1}}^2 +  t_{k+1}^2  \eta_k^2  \norms{\mbf{p}_{i_k}^{-1}d^k_{[i_k]}}^2 \vspace{1ex}\\
&& + {~}   2r t_{k+1}\theta_k\iprods{x^k - x^{k-1}, x^k - x^{\star}}  - 2r  t_{k+1} \eta_k \iprods{\mbf{p}_{i_k}^{-1}d^k_{[i_k]}, x^k - x^{\star}}  \vspace{1ex}\\
&& - {~}  2 t_{k+1}^2 \theta_k \eta_k \iprods{\mbf{p}_{i_k}^{-1}d^k_{[i_k]}, x^k - x^{k-1}}.
\end{array}
}
Alternatively, we can expand
\myeqn{
\arraycolsep=0.1em
\begin{array}{lcll}
\norms{r(x^{k-1} - x^{\star}) + t_k(x^k - x^{k-1}) }^2 &= & \norms{r(x^k - x^{\star}) + (t_k - r)(x^k - x^{k-1})}^2 \vspace{1ex}\\
&= & r^2 \norms{x^k - x^{\star}}^2 + 2r(t_k - r)\iprods{x^k - x^{k-1}, x^k - x^{\star}} +  (t_k - r)^2\norms{x^k - x^{k-1}}^2.
\end{array}
}
In addition, for $r > 0$ and $\mu \geq 0$, we also have the following identity
\myeqn{  
\arraycolsep=0.3em
\begin{array}{lcl}
\mu r \norms{x^{k-1} - x^{\star}}^2 - \mu r \norms{x^k - x^{\star}}^2 &= & \mu r \norms{x^k - x^{k-1}}^2 - 2\mu r \iprods{x^k - x^{k-1}, x^k - x^{\star}}.
\end{array}
}
Combining the last three expressions, and using  $\Qc_k$ from \eqref{eq:Qck_func}, we can show that
\myeq{eq:le4.1_proof1}{
\arraycolsep=0.2em
\begin{array}{lcl}
\Qc_k - \Qc_{k+1} &= & \left[ (t_k - r)^2 - t_{k+1}^2 \theta_k^2  + \mu r \right]\norms{x^k - x^{k-1}}^2  -  t_{k+1}^2 \eta_k^2 \norms{\mbf{p}_{i_k}^{-1}d_{[i_k]}^k}^2 \vspace{1ex}\\
&& + {~}   2r t_{k+1} \eta_k  \iprods{ \mbf{p}_{i_k}^{-1}d^k_{[i_k]}, x^k - x^{\star}} +  2 t_{k+1}^2\theta_k \eta_k \iprods{ \mbf{p}_{i_k}^{-1} d^k_{[i_k]}, x^k - x^{k-1}} \vspace{1ex}\\
&& + {~} 2r \left( t_k - r - \mu -  t_{k+1} \theta_k  \right) \iprods{ x^k - x^{k-1}, x^k - x^{\star} }.
\end{array}
}

%%% Step 2.
\noindent\textit{\underline{Step 2}: Applying the conditional expectation.}
Next, since $d^k_{[i_k]} = [0, \cdots, 0, d_{i_k}^k, 0, \cdots, 0]$ and $d^k = Gx^k  -  \gamma_kGx^{k-1}$, applying \eqref{eq:prob_cond1}, we can easily show that
\myeq{eq:le4.1_proof2}{
\arraycolsep=0.2em
\left\{\begin{array}{lclcl}
\E_k\big[ \norms{\mbf{p}_{i_k}^{-1}d_{[i_k]}^k}^2 \big]  & = &  \sum_{i=1}^n\mbf{p}_i^{-1}\norms{ d^k_i}^2 & = &  \sum_{i=1}^n\mbf{p}_i^{-1}\norms{ [Gx^k]_i - \gamma_k [Gx^{k-1}]_i }^2,   \vspace{1ex}\\
\E_k\big[ \iprods{\mbf{p}_{i_k}^{-1} d_{[i_k]}^k, x^k - x^{\star}} \big]  & = &   \sum_{i=1}^n\iprods{d^k_i , x_i^k - x_i^{\star}} & = &  \iprods{ Gx^k - \gamma_k Gx^{k-1}, x^k - x^{\star}}, \vspace{1ex}\\
\E_k\big[ \iprods{\mbf{p}_{i_k}^{-1} d_{[i_k]}^k, x^k - x^{k-1}} \big]  & = &  \sum_{i=1}^n \iprods{ d^k_i , x_i^k - x_i^{k-1}} & = &  \iprods{ Gx^k - \gamma_k Gx^{k-1} , x^k - x^{k-1}}.
\end{array}\right.
}
Taking the conditional expectation $\Exps{k}{\cdot}$ both sides of \eqref{eq:le4.1_proof1} and then using \eqref{eq:le4.1_proof2}, we get 
\myeqn{
\arraycolsep=0.2em
\begin{array}{lcl}
\Qc_k - \Exps{k}{\Qc_{k+1}} &= & \left[ (t_k - r)^2 - t_{k+1}^2  \theta_k^2 + \mu r \right]\norms{x^k - x^{k-1}}^2  - t_{k+1}^2  \eta_k^2 \sum_{i=1}^n \frac{1}{\mbf{p}_i}\norms{ [Gx^k]_i -   \gamma_k[Gx^{k-1}]_i}^2  \vspace{1ex}\\
&& + {~} 2r t_{k+1}\eta_k  \iprods{Gx^k, x^k - x^{\star}}  -  2r t_{k+1}\gamma_k\eta_k  \iprods{ Gx^{k-1}, x^{k-1} - x^{\star}}  \vspace{1ex}\\
&& + {~}  2 t_{k+1}^2\theta_k\eta_k   \iprods{ Gx^k  - Gx^{k-1}, x^k - x^{k-1}} \vspace{1ex}\\
&& + {~} 2 t_{k+1}\eta_k \left[ t_{k+1}\theta_k(1 - \gamma_k)  - r \gamma_k  \right]   \iprods{ Gx^{k-1}, x^k - x^{k-1}} \vspace{1ex}\\
&& +  {~}  2r \left( t_k - r - \mu - t_{k+1}\theta_k \right) \iprods{  x^k - x^{k-1}, x^k - x^{\star} }.
\end{array}
}

%%% Step 3.
\noindent\textit{\underline{Step 3}:  Lower bounding $\Qc_k - \Exps{k}{\Qc_{k+1}}$ by using  \eqref{eq:CE_assumption_A2}.}
Utilizing \eqref{eq:CE_assumption_A2} in Assumption~\ref{as:A2} into the last expression, we arrive at
\myeqn{ 
\hspace{-1ex}
\arraycolsep=0.2em
\begin{array}{lcl}
\Qc_k - \Exps{k}{\Qc_{k+1}} & \geq & \big[ (t_k - r)^2 - t_{k+1}^2 \theta_k^2  + \mu r \big]\norms{x^k - x^{k-1}}^2  -  t_{k+1}^2 \eta_k^2 \sum_{i=1}^n\frac{1}{\mbf{p}_i}\norms{  [Gx^k]_i -   \gamma_k[Gx^{k-1}]_i}^2   \vspace{1ex}\\
&& + {~}    2r t_{k+1}\eta_k \iprods{Gx^k, x^k - x^{\star}} - 2r t_{k+1}\gamma_k\eta_k  \iprods{Gx^{k-1}, x^{k-1} - x^{\star}}  \vspace{1ex}\\
&& + {~} 2  t_{k+1}^2\theta_k\eta_k  \bar{\beta} \circ \norms{Gx^k  - Gx^{k-1}}^2 \vspace{1ex}\\
&& + {~}  2  t_{k+1}\eta_k \big[ t_{k+1}\theta_k(1 - \gamma_k)  -  r\gamma_k  \big] \iprods{Gx^{k-1}, x^k - x^{k-1}} \vspace{1ex}\\
&& + {~} 2 r \big( t_k - r - \mu - t_{k+1}\theta_k \big) \iprods{ x^k - x^{k-1}, x^k - x^{\star} }.
\end{array}
\hspace{-1ex}
}
Adding and subtracting $2  t_{k+1}^2\theta_k\eta_k \beta \circ \norms{Gx^k  - Gx^{k-1} }^2$,  we can expand the last estimate as
\myeq{eq:le4.1_proof5}{
\hspace{-1ex}
\arraycolsep=0.2em
\begin{array}{lcl}
\Qc_k - \Exps{k}{\Qc_{k+1}} & \geq & \big[ (t_k - r)^2 - t_{k+1}^2\theta_k^2 + \mu r  \big] \norms{x^k - x^{k-1}}^2  \vspace{1ex} \\
&& + {~} 2  t_{k+1}^2\theta_k\eta_k   (\bar{\beta} - \beta) \circ \norms{ Gx^k  -  Gx^{k-1} }^2 \vspace{1ex}\\
&& +  {~}  2r t_{k+1}\eta_k \big[ \iprods{Gx^k, x^k - x^{\star}}   -  \beta \circ \norms{ Gx^k }^2 \big] \vspace{1ex}\\
&& - {~}   2r t_{k+1}\gamma_k\eta_k \big[ \iprods{Gx^{k-1}, x^{k-1} - x^{\star}} -   \beta \circ \norms{ Gx^{k-1} }^2 \big] \vspace{1ex}\\
&& + {~}  t_{k+1}\eta_k \sum_{i=1}^n \frac{ [ 2\beta_i\mbf{p}_i (t_{k+1}\theta_k  + r )  -  t_{k+1}\eta_k ] }{\mbf{p}_i} \norms{[Gx^k]_i}^2 \vspace{1ex}\\
&& + {~}  t_{k+1}\eta_k \sum_{i=1}^n \frac{ [ 2\beta_i\mbf{p}_i (t_{k+1}\theta_k - r \gamma_k)   - t_{k+1} \eta_k \gamma_k^2 ] }{\mbf{p}_i} \norms{[Gx^{k-1}]_i}^2 \vspace{1ex}\\
&& - {~}   2 t_{k+1}^2\eta_k \sum_{i=1}^n \frac{ \left( 2 \beta_i\mbf{p}_i \theta_k -  \gamma_k\eta_k \right) }{\mbf{p}_i}   \iprods{[Gx^k]_i, [Gx^{k-1}]_i} \vspace{1ex}\\ 
&& + {~} 2  t_{k+1}\eta_k \big[ t_{k+1}\theta_k(1 - \gamma_k)  -  r\gamma_k  \big] \iprods{Gx^{k-1}, x^k - x^{k-1}} \vspace{1ex}\\
&& + {~} 2r \big( t_k - r - \mu -  t_{k+1} \theta_k \big) \iprods{ x^k - x^{k-1}, x^k - x^{\star} }.
\end{array}
\hspace{-2ex}
}

\noindent\textit{\underline{Step 4}:  Simplifying to obtain \eqref{eq:ARCOG_scheme_main_est1}.}
We first impose the following two conditions:
\myeq{eq:le4.1_para_conds_00}{
t_k - r - \mu - t_{k+1} \theta_k = 0 \quad \text{and} \quad  t_{k+1}\theta_k(1 - \gamma_k)  -  r \gamma_k   =  0.
}
These conditions lead to $\theta_k := \frac{t_k - r- \mu}{t_{k+1}}$ and $\gamma_k := \frac{t_{k+1}\theta_k}{t_{k+1}\theta_k + r} = \frac{t_k - r - \mu}{t_k - \mu}$, respectively as stated in \eqref{eq:ARCOG_scheme_para_cond2}.
Exploiting the condition \eqref{eq:le4.1_para_conds_00}, we can show that
\myeqn{
\hspace{-1ex}
\arraycolsep=0.2em
\left\{ \begin{array}{lclcl}
A_i^k & := &  \frac{ t_{k+1}\eta_k [ 2\beta_i\mbf{p}_i (t_{k+1}\theta_k  + r) -  t_{k+1}\eta_k ] }{\mbf{p}_i} 
& = & \frac{ t_{k+1}\eta_k [ 2\beta_i\mbf{p}_i (t_{k+1}\theta_k + r) - t_{k+1}\eta_k ] }{\mbf{p}_i}, \vspace{1ex}\\
%%%
B^k_i & := &  \frac{ t_{k+1}^2 \eta_k( 2 \beta_i\mbf{p}_i \theta_k -   \gamma_k \eta_k ) }{\mbf{p}_i} 
&= &  \frac{ t_{k+1}\eta_k [ 2\beta_i \mbf{p}_i ( t_{k+1}\theta_k  + r) -  t_{k+1}\eta_k ] }{ \mbf{p}_i  } \cdot \gamma_k, \vspace{1ex}\\
%%%
C^k_i &:= &   \frac{ t_{k+1}\eta_k [  2\beta_i\mbf{p}_i (t_{k+1}\theta_k  - r \gamma_k)  -  t_{k+1} \eta_k \gamma_k^2 ] }{\mbf{p}_i} 
&=& \frac{ t_{k+1} \eta_k [ 2\beta_i\mbf{p}_i (t_{k+1}\theta_k + r) -  t_{k+1}\eta_k ] }{ \mbf{p}_i } \cdot \gamma_k^2.
%%%
\end{array}\right.
\hspace{-4ex}
}
Next, we utilize these three coefficients $A_i^k$, $B_i^k$, and $C_i^k$ to prove that
\myeqn{
\arraycolsep=0.3em
\begin{array}{lcl}
\widehat{\Tc}_{[1]}^{[i]} &:= & \frac{ t_{k+1}\eta_k [  2  \beta_i\mbf{p}_i  (t_{k+1}\theta_k  + r) -   t_{k+1}\eta_k ] }{\mbf{p}_i}  \norms{[Gx^k]_i}^2  -   \frac{ 2  t_{k+1}^2\eta_k ( 2 \beta_i\mbf{p}_i \theta_k  -   \gamma_k\eta_k ) }{\mbf{p}_i}  \iprods{[Gx^k]_i, [Gx^{k-1}]_i } \vspace{1ex}\\ 
&& + {~} \frac{ t_{k+1} \eta_k [ 2  \beta_i\mbf{p}_i  (t_{k+1}\theta_k  - r\gamma_k)  - t_{k+1} \eta_k \gamma_k^2 ] }{\mbf{p}_i}   \norms{[Gx^{k-1}]_i}^2  \vspace{1ex}\\
&= &  \frac{  t_{k+1}\eta_k [ 2\beta_i \mbf{p}_i ( t_{k+1}\theta_k + r) -    t_{k+1}\eta_k ] }{ \mbf{p}_i} \norms{  [Gx^k]_i - \gamma_k [Gx^{k-1}]_i }^2.
\end{array}
}
Finally, using again \eqref{eq:le4.1_para_conds_00} and $\widehat{\Tc}_{[1]}^{[i]}$ above,  we can simplify \eqref{eq:le4.1_proof5} as
\myeqn{%eq:le4.1_proof6}{
\hspace{-1ex}
\arraycolsep=0.2em
\begin{array}{lcl}
\Qc_k - \Exps{k}{\Qc_{k+1}} & \geq & \mu( 2t_k - r - \mu) \norms{x^k - x^{k-1}}^2  +   2r  t_{k+1}\eta_k  \big[ \iprods{Gx^k, x^k - x^{\star}}   -  \beta \circ \norms{  Gx^k }^2 \big] \vspace{1ex} \\
&& - {~}  2r t_{k+1}\gamma_k\eta_k  \big[ \iprods{Gx^{k-1}, x^{k-1} - x^{\star}} -  \beta \circ \norms{ Gx^{k-1} }^2 \big] \vspace{1ex}\\
&& + {~} 2 t_{k+1}^2\theta_k\eta_k   (\bar{\beta}  - \beta ) \circ \norms{ Gx^k  -  Gx^{k-1} }^2 \vspace{1ex}\\
&& + {~} \sum_{i=1}^n \frac{ t_{k+1} \eta_k [  2 \beta_i \mbf{p}_i ( t_{k+1}\theta_k + r ) - t_{k+1}\eta_k  ] }{ \mbf{p}_i } \norms{ [Gx^k]_i - \gamma_k [Gx^{k-1}]_i }^2.
\end{array}
\hspace{-4ex}
}
Finally, rearranging the last expression and using  the definition of $\hat{\Pc}_k$ from \eqref{eq:ARCOG_scheme_Lyfunc}, we obtain \eqref{eq:ARCOG_scheme_main_est1}.
\Eproof
%%%% End of proof.
%%%%

\textbf{\color{blue}We will need the following version of Opial's lemma, whose proof is similar to \cite[Proposition 4.1]{davis2022variance}. (This is a correction for the published version on MOR).}

%%% Lemma A4.
{\color{blue}
\begin{lemma}\label{le:opial_lemma}
%Let $\sets{x^k}$ be a random sequence in $\R^p$ and  $\Sc^{\star}$ be a given subset in $\R^p$.
%Suppose that  $\lim_{k\to\infty} \norms{x^k - x^{\star}}^2$ exists almost surely  for any $x^{\star} \in \Sc^{\star}$ and  every \textrm{almost sure} cluster point of $\sets{x^k}$ belongs to $\Sc^{\star}$.
%Then, $\sets{x^k}$ converges almost surely to $x^{\star} \in \Sc^{\star}$.
Suppose that $G : \R^p\to\R^p$ is continuous and $\zer{G}\neq\emptyset$.
Let $\sets{x^k}$ be a sequence of random vectors such that for all $x^{\star} \in \zer{G}$, the sequence $\sets{\norms{x^k - x^{\star}}^2}$ almost surely converges to a $[0, \infty)$-valued random variable. 
In addition, assume that $\sets{\norms{Gx^k}}$ also almost surely converges to zero.
Then, $\sets{x^k}$ almost surely converges to a $\zer{G}$-valued random variable.
\end{lemma}}

%We will need the following version of Opial's Lemma, whose proof is similar to \cite[Lemma 26]{peypouquet2010evolution}.
%
%%%% Lemma A4.
%\begin{lemma}\label{le:opial_lemma}
%Let $\sets{x^k}$ be a random sequence in $\R^p$ and  $\Sc^{\star}$ be a given subset in $\R^p$.
%Suppose that  $\lim_{k\to\infty} \norms{x^k - x^{\star}}^2$ exists almost surely  for any $x^{\star} \in \Sc^{\star}$ and  every \textrm{almost sure} cluster point of $\sets{x^k}$ belongs to $\Sc^{\star}$.
%Then, $\sets{x^k}$ converges almost surely to $x^{\star} \in \Sc^{\star}$.
%\end{lemma}

To prove the almost sure convergence in Theorem~\ref{th:ARCOG_convergence} we will use the following lemma.

%%% Lemma A.3.
\begin{lemma}\label{le:ARCOG_xk_convergence}
Under the same conditions as in Theorem~\ref{th:ARCOG_convergence}, almost surely, we have
\myeq{eq:ARCOG_almost_sure}{
\arraycolsep=0.2em
\begin{array}{ll}
& \sum_{k=0}^{\infty}(2k + r + 1) \norms{x^k - x^{k-1}}^2 < +\infty, \vspace{1.25ex}\\
&\sum_{k=0}^{\infty}(k+r)^2  \norms{ Gx^k - \gamma_k Gx^{k-1} }^2 < +\infty, \vspace{1.25ex}\\
& \sum_{k=0}^{+\infty}(k + r + 1) \norms{Gx^k}^2  < +\infty, \vspace{1.5ex} \\ 
& \lim_{k\to\infty} k^2\norms{x^{k+1} - x^k}^2 = 0, \quad \textrm{and} \quad \lim_{k\to\infty} k^2\norms{Gx^k}^2 = 0.
\end{array}
}
Moreover, $\lim_{k\to\infty} \norms{x^k - x^{\star}}^2$ exists almost surely for any $x^{\star} \in \zer{G}$.
\end{lemma}

%%% Proof of Lemma A.3.
\proof{\textbf{Proof.}}
We divide this proof into the following steps.

%%% Step 1.
\vspace{0.5ex}
\noindent\textit{\underline{Step 1}: The first two summable bounds in \eqref{eq:ARCOG_almost_sure}.}
First, similar to the proof of \eqref{eq:ARCOG_scheme_proof2001}, under the conditions of $\beta_i$ and $\omega$, we have  $\underline{\Lambda}_0 := \min\big\{ \frac{ 2\beta_i\mbf{p}_i - \omega }{ \mbf{p}_i } : i \in [n]  \big\} > 0$.
Thus \eqref{eq:ARCOG_scheme_main_est1} reduces to
\myeqn{
\arraycolsep=0.2em
\begin{array}{lcl}
\E_k\big[ \hat{\Pc}_{k+1} \big] & \leq & \hat{\Pc}_k - (2k + r + 1) \norms{x^k - x^{k-1}}^2 -   \omega  \underline{\Lambda}_0 (k+r)^2  \norms{ Gx^k - \gamma_k Gx^{k-1} }^2.
\end{array}
}
Now, applying the Robbins and Siegmund Supermartingale Convergence Theorem in \cite[Theorem 1]{robbins1971convergence} we conclude that $\lim_{k\to\infty} \hat{\Pc}_k$ exists almost surely.
Moreover, almost surely, we also have 
\myeq{eq:ARCOG_almost_sure_proof1}{
\arraycolsep=0.2em
\begin{array}{ll}
\sum_{k=0}^{\infty}(2k + r + 1) \norms{x^k - x^{k-1}}^2 < +\infty \quad \textrm{and} \quad \sum_{k=0}^{\infty}(k+r)^2  \norms{ Gx^k - \gamma_k Gx^{k-1} }^2 < +\infty.
\end{array}
}
These prove the first two summable bounds in \eqref{eq:ARCOG_almost_sure}.

%%% Step 2.
\vspace{0.5ex}
\noindent\textit{\underline{Step 2}: The third summable bound in \eqref{eq:ARCOG_almost_sure}.}
Next, similar to the proof of \eqref{eq:ARCOG_scheme_2023_proof1}, but without taking the full expectation, with $v^k := x^{k+1} - x^k$ and $\psi := \max_{i \in [n]}\big( \frac{1-\mbf{p}_i}{\mbf{p}_i} \big)$, we have
\myeqn{
\hspace{-3ex}
\arraycolsep=0.1em
\begin{array}{ll}
(k+r-1)(k+r+2) & \Exps{k}{\norms{v^k + \eta_k Gx^k}^2}  \leq   (k+r-2)(k + r + 1) \norms{v^{k-1} +  \eta_{k-1} Gx^{k-1}}^2  \vspace{1ex}\\
& - {~} (r-2)(k + r + 1) \norms{ v^{k-1} + \eta_{k-1} Gx^{k-1}}^2  \vspace{1ex}\\
& + {~}  \psi \omega^2 (k + r)^2 \cdot \norms{  Gx^k  -  \gamma_k Gx^{k-1}  }^2 +  \frac{ 4(k+r+1)}{r} \cdot \norms{x^k - x^{k-1}}^2.
\end{array}
\hspace{-6ex}
}
Since the last two terms are summable almost surely as proven in \eqref{eq:ARCOG_almost_sure_proof1}, applying  again \cite[Theorem 1]{robbins1971convergence}, we conclude that  $\lim_{k\to\infty}(k+r-2)(k + r + 1) \norms{v^{k-1} +  \eta_{k-1} Gx^{k-1}}^2$ exists almost surely.
Moreover, we also obtain $\sum_{k=0}^{\infty}(k + r + 1) \norms{ v^{k-1} + \eta_{k-1} Gx^{k-1}}^2 < +\infty$ almost surely.
Combining both statements, almost surely, we conclude that 
\myeq{eq:ARCOG_almost_sure_proof2}{
\arraycolsep=0.2em
\begin{array}{ll}
& \lim_{k\to\infty}(k+r-2)(k + r + 1) \norms{v^{k-1} +  \eta_{k-1} Gx^{k-1}}^2 = 0, \vspace{1.25ex}\\
& \sum_{k=0}^{\infty}(k + r + 1) \norms{ v^{k-1} + \eta_{k-1} Gx^{k-1}}^2 < +\infty.
\end{array}
}
Utilizing \eqref{eq:RCHP_scheme00_proof4}, the first summable bound in \eqref{eq:ARCOG_almost_sure_proof1}, and the second line of \eqref{eq:ARCOG_almost_sure_proof2}, we can easily show that $ \sum_{k=0}^{+\infty}(k + r + 1) \norms{Gx^k}^2 < +\infty$ almost surely, which proves the third line of \eqref{eq:ARCOG_almost_sure}.

%%% Step 3.
\vspace{0.5ex}
\noindent\textit{\underline{Step 3}: The last two limits in \eqref{eq:ARCOG_almost_sure}.}
Now, similar to the proof of \eqref{eq:RCHP_scheme00_proof7_01} but without taking the full expectation, we get
\myeqn{
\hspace{-2ex}
\arraycolsep=0.2em
\begin{array}{lcl}
(k+r+2)^2\Exps{k}{\norms{x^{k+1} - x^k}^2} & \leq & (k+r+1)^2  \norms{x^k - x^{k-1}}^2  - (r+2)(k + r)  \norms{x^k - x^{k-1}}^2 \vspace{1ex}\\
&& + {~}  \frac{\omega^2(k+r)^2}{\mbf{p}_{\min}} \cdot \norms{ Gx^k - \gamma_k Gx^{k-1} }^2 +  \omega^2r k \cdot \norms{Gx^k}^2.
\end{array}
\hspace{-4ex}
}
Since the last two terms are summable almost surely due to \eqref{eq:ARCOG_almost_sure_proof1} and \eqref{eq:ARCOG_almost_sure_proof2}, applying one more time \cite[Theorem 1]{robbins1971convergence}, we obtain that $\lim_{k\to\infty}  (k+r+1)^2  \norms{x^k - x^{k-1}}^2$ exists almost surely and $\sum_{k=0}^{\infty} (k + r)  \norms{x^k - x^{k-1}}^2 < +\infty$ almost surely.
These two statements imply the first limit in the last line of \eqref{eq:ARCOG_almost_sure}, that is
\myeq{eq:ARCOG_almost_sure_proof3}{
\arraycolsep=0.2em
\begin{array}{ll}
& \lim_{k\to\infty}  (k+r+1)^2  \norms{x^k - x^{k-1}}^2 = 0 \quad \textrm{almost surely}.
\end{array}
}
Combining the limit in the first line of \eqref{eq:ARCOG_almost_sure_proof2}, the limit in \eqref{eq:ARCOG_almost_sure_proof3}, and \eqref{eq:RCHP_scheme00_proof4}, we can prove that $\lim_{k\to\infty} k^2\norms{Gx^k}^2 = 0$ almost surely, which proves the second limit in the last line of \eqref{eq:ARCOG_almost_sure}.

%%% Step 4.
\vspace{0.5ex}
\noindent\textit{\underline{Step 4}: The existence of $\lim_{k\to\infty} \norms{x^k - x^{\star}}^2$.}
Furthermore, for $\mu = 1$, it is obvious to see that $r \norms{x^k - x^{\star}}^2 \leq \hat{\Pc}_k$ for all $k\geq 0$.
Combining this relation and the existence of $\lim_{k\to\infty}\hat{\Pc}_k$ almost surely, there exists $M > 0$ such that almost surely
\myeq{eq:lmA3_proof1}{
\norms{x^k - x^{\star}}^2  \leq  M, \quad \textrm{for all} \ k\geq 0.
}
Since $\frac{\omega r}{r+2} \leq \eta_k \leq \omega$ for all $k\geq 0$ due to \eqref{eq:ARCOG_scheme_update_pars} and $t_k := k + r + 1$, by the Cauchy-Schwarz inequality, \eqref{eq:lmA3_proof1}, and the last limit of \eqref{eq:ARCOG_almost_sure}, we can show that almost surely
\myeqn{
\arraycolsep=0.2em
\begin{array}{lcl}
\vert t_k\eta_{k-1}\iprods{Gx^{k-1}, x^{k-1} - x^{\star}}  \vert^2 & \leq & \vert  t_k\iprods{Gx^{k-1}, x^{k-1} - x^{\star}}  \vert^2 \vspace{1ex}\\
& \leq &   \norms{x^{k-1} - x^{\star}}^2 \cdot t_k^2\norms{Gx^{k-1}}^2 \vspace{0.5ex}\\
& \overset{\tiny\eqref{eq:lmA3_proof1}}{\leq} & M  t_k^2\norms{Gx^{k-1}}^2   \overset{\tiny \eqref{eq:ARCOG_almost_sure} }{ \to } 0 \quad \textrm{as} \ k \to \infty.
\end{array}
}
Thus we get $\lim_{k\to \infty} \vert t_k\eta_{k-1}\iprods{Gx^{k-1}, x^{k-1} - x^{\star} }\vert = 0$ almost surely.
Moreover, from the third line of \eqref{eq:ARCOG_almost_sure}, we have $\lim_{k\to \infty} t_k\eta_{k-1} \norms{Gx^{k-1}}^2  = 0$ almost surely.
Combining the last two limits yields 
\myeq{eq:lmA3_proof2}{
\lim_{k\to\infty} t_k\eta_{k-1} \big[ \iprods{Gx^{k-1}, x^{k-1} - x^{\star}} -  \beta \circ \norms{Gx^{k-1}}^2 \big] = 0 \quad\textrm{almost surely}.
}
Similarly, by the Cauchy-Schwarz inequality, \eqref{eq:lmA3_proof1}, and the last line of \eqref{eq:ARCOG_almost_sure}, we have, almost surely 
\myeqn{
\arraycolsep=0.2em
\begin{array}{lcl}
\vert  t_k\iprods{x^{k-1} - x^{\star}, x^k - x^{k-1}}  \vert^2 & \leq &   \norms{x^{k-1} - x^{\star}}^2 \cdot t_k^2\norms{x^k - x^{k-1}}^2 \vspace{1ex}\\
& \overset{\tiny\eqref{eq:lmA3_proof1}}{\leq} & M  t_k^2\norms{x^k - x^{k-1}}^2  \overset{\tiny \eqref{eq:ARCOG_almost_sure} }{ \to } 0 \quad \textrm{as} \ k \to \infty.
\end{array}
}
Hence, we conclude that
\myeq{eq:lmA3_proof3}{
\lim_{k\to\infty}\vert t_k \iprods{x^{k-1} - x^{\star}, x^k - x^{k-1}}  \vert = 0 \quad \textrm{almost surely}.
}
Finally, from \eqref{eq:ARCOG_scheme_Lyfunc}, with $\mu = 1$, we can rewrite $\hat{\Pc}_k$ as
\myeqn{
\arraycolsep=0.2em
\begin{array}{lcll}
 \hat{\Pc}_k  & = & 2 r t_k\eta_{k-1} \big[  \iprods{Gx^{k-1}, x^{k-1} - x^{\star}} -  \beta \circ \norms{Gx^{k-1}}^2  \big] + 2r t_k  \iprods{x^{k-1} - x^{\star}, x^k - x^{k-1}}     \vspace{1ex} \\
&& + {~} t_k^2 \norms{x^k - x^{k-1} }^2 +  r(r+1) \norms{x^{k-1} - x^{\star}}^2. 
\end{array}
}
Since, almost surely, the limit on the left-hand side of this expression exists, the limit of the first two terms on the right-hand side is zero due to \eqref{eq:lmA3_proof2} and \eqref{eq:lmA3_proof3}, respectively.
The limit of the third term $t_k^2 \norms{x^k - x^{k-1}} ^2$ on the right-hand side is zero almost surely due to the last line of \eqref{eq:ARCOG_almost_sure}.
Therefore, we conclude that $\lim_{k\to\infty} \norms{x^{k-1} - x^{\star}}^2 $ exists almost surely, which proves our claim.
\Eproof
%%% End of Proof.
\end{APPENDICES}

%%%%%%%%%%%%%%%%%%%%%%%%%%%%%%%%%%%%%%%%%%%%%%%
%+ References.
%%%%%%%%%%%%%%%%%%%%%%%%%%%%%%%%%%%%%%%%%%%%%%%
\bibliographystyle{informs2014} 
%\bibliography{/Users/quoctd/Dropbox/E-Books/tran_bibtex_new}

\end{document}